\documentclass[11pt,letterpaper,oneside,reqno]{amsart}
\usepackage{t1enc}
\usepackage[latin1]{inputenc}
\usepackage[english]{babel}
\usepackage{latexsym}
\usepackage[dvips]{graphicx}
\usepackage{graphicx}
\usepackage{float}
\usepackage[natural]{xcolor}
\usepackage{enumerate}
\usepackage{appendix}
\usepackage{multirow}
\usepackage{xcolor}
\usepackage{hyperref}
\usepackage{lineno}
\usepackage{setspace}
\usepackage{fullpage}
\usepackage[shortlabels]{enumitem}

\usepackage{enumitem, amsmath, amsthm, comment, amssymb}
\usepackage[norelsize, ruled, boxed, commentsnumbered]{algorithm2e}
\usepackage[letterpaper, margin=1in]{geometry}
\usepackage{graphicx,float}

\usepackage{listings}
\setlist{itemsep=0.25em, topsep=0.4em}
\usepackage{setspace}
\newtheorem{theorem}{Theorem}\numberwithin{theorem}{section}
\newtheorem{lemma}[theorem]{Lemma}\numberwithin{theorem}{section}
\newtheorem{proposition}[theorem]{Proposition}
\numberwithin{theorem}{section}
\newtheorem{fact}{Fact}

\newtheorem{example}[theorem]{Example}
\newtheorem{definition}[theorem]{Definition}

\newtheorem{conjecture}[theorem]{Conjecture}
\newtheorem{claim}[theorem]{Claim}

\newtheorem{setup}[theorem]{Setup}
\numberwithin{equation}{section}

\def\C{\mathcal{C}}

\def\F{\mathcal{F}}

\def\cP{\mathcal{P}}
\def\Q{\mathcal{Q}}

\def\T{\mathcal{T}}
\def \a{\alpha}
\def \e{\varepsilon}
\def \g{\gamma}
\def \d{\delta}
\def \b{\beta}
\def \cP{\mathcal{P}}

\def \bfi{\mathbf{i}}
\def \bfv{\mathbf{v}}

\def\int{\textrm{int}}

\begin{document}

%\onehalfspacing
%\linenumbers
%\baselineskip 0.56cm

%\begin{titlepage}
    
\title{Decision problem for Hamilton $2$-cycles in $4$-graphs}
\author{Luyining Gan}
\address{LG. School of Mathematical Sciences, Beijing University of Posts and Telecommunications, Beijing,
China; Key Laboratory of Mathematics and Information Networks (BUPT), Ministry of Education,
Beijing, China\\
Email: \texttt{(LG) elainegan@bupt.edu.cn.}}
\author{Jie Han}
\address{JH and BW. School of Mathematics and Statistics, Beijing Institute of Technology, China\\
Email: \texttt{(JH) han.jie@bit.edu.cn, (BW) bin.wang@bit.edu.cn.}}
\author{Bin Wang}

\begin{abstract}
A $4$-uniform $2$-cycle in a $4$-uniform hypergraph of length $t$ is a cyclic ordering of $2t$ vertices $v_1v_2\cdots v_{2t}v_1$ such that $v_{2i+1}v_{2i+2}v_{2i+3}v_{2i+4}$ are edges for $0\le i\le t-1$ while the addition is modulo $2t$. 

For every $\gamma>0$ and large $n$, we characterize the $n$-vertex $4$-uniform hypergraphs such that every triple of vertices is contained in at least $(1/3+\gamma)n$ edges and admits a Hamilton $2$-cycle.
Up to the error term $\gamma n$, the assumption on the minimum codegree is best possible and verifies a conjecture of Garbe and Mycroft.
As a consequence, this gives a polynomial-time algorithm that decides whether an $n$-vertex $4$-uniform hypergraph with minimum codegree $(1/3+\gamma)n$ contains a Hamilton $2$-cycle.
This stands as a steep contrast to the graph case where such a hardness gap has size $o(n)$.
\end{abstract}
\maketitle 
%\end{titlepage}

\section{Introduction}

The Hamilton cycle problem looks for a spanning cycle of a given graph.
It is one of the original 21 NP-complete problems by Karp~\cite{Karp}.
The study of Hamiltonicity has been one of the central topics in graph theory and computer science~\cite{BM08, Di17}, due to its broad applications and connections with real-world problems, such as TSP, routing, and scheduling~\cite{Karp, LLRS85, GJ79}.
Given its importance, it is desirable to extend the study of Hamilton cycles to hypergraphs, which have strong potential for modeling more sophisticated multi-object relations in applications.

A $k$-uniform graph $P$ is an $\ell$-path if $P$ has no isolated vertices and moreover, the vertices of $P$ can be ordered in such a way that every edge of $P$ consists of $k$ consecutive vertices and every two consecutive edges intersect in precisely $\ell$ vertices.
Furthermore, if the ordering is cyclic, then $P$ is called an $\ell$-cycle.
Naturally, we say that an $n$-vertex $k$-graph contains a Hamilton $\ell$-cycle if there is a subgraph of $H$ which is a copy of $\ell$-cycle covering all vertices of $H$.
Note that it is necessary that $(k-\ell)$ divides $n$.
The \emph{length} of an $\ell$-path or $\ell$-cycle is the number of edges it contains.
When $\ell=k-1$, it is also referred to as a tight path or tight cycle, which is considered in general as the most interesting case -- one simple reason is that it contains other $\ell$-paths or cycles as subgraphs.

Let $H$ be an $n$-vertex $k$-uniform graph ($k$-graph for short).
The degree of a set $S$ of vertices is the number of edges containing $S$.
Denote by $\delta_d(H)$ its minimum codegree, that is, the minimum of the degree of $S$ (denoted by $\deg(S)$) over all $S$ of size $d$ where $d\in[k-1]$.
We refer to $\delta_{k-1}(H)$ as the minimum codegree of $H$.
%Denote by $\delta_{k-1}(H)$ its minimum codegree, that is, the minimum of the degree of $S$ over all $S$ of size $k-1$.
There has been a strong focus on the study of minimum (co)degree conditions forcing the existence of Hamilton $\ell$-cycles in $k$-graphs, pioneered by the celebrated work of R\"odl, Ruci\'nski and Szemer\'edi~\cite{MR2399020} on an approximate version of the Dirac theorem in $k$-graphs, as well as the absorbing method developed by them in this thread of research.
In particular, in~\cite{MR2399020} it is shown that a minimum codegree $(1/2+o(1))|H|$ guarantees the existence of a tight Hamilton cycle in a $k$-graph $H$.

\subsection{Algorithmic problems -- below-guarantee}
Beyond sufficient conditions, the below-guarantee problem was first investigated by Garbe and Mycroft~\cite{GaMy}.
For tight Hamilton cycles, they showed a strong hardness result saying that for $k\ge 3$, there exists a constant $C$ such that the decision problem for tight Hamilton cycles in $n$-vertex $k$-graphs with minimum codegree $n/2-C$ is NP-complete.
That is, in terms of minimum codegree assumptions, the gap between intractable and trivially tractable is $o(n)$.
%This stands as a contrast to the graph case ($k=2$) where it is known that for the family of graphs of minimum graph $n/2-O(\log n)$ polynomial-time algorithms exist for the Hamilton cycle problem.

On the other hand, the main result of Garbe and Mycroft~\cite{GaMy} shows that for 4-uniform $2$-cycles, the corresponding gap is as large as $\Theta(n)$.
That is, they showed that there exists $\e>0$, such that for  every $n$-vertex $4$-graph $H$ with $\delta_3(H)\ge \frac n 2 - \varepsilon n$ the Hamilton $2$-cycle problem is tractable.
Moreover, they achieve it by characterizing all such $4$-graphs $H$ with Hamilton $2$-cycles. 
Roughly speaking, the obstructions preventing the existence of Hamilton $2$-cycles in $H$ are some parity constructions defined on a bipartition $\{A, B\}$ of $V(H)$.
In particular, given the minimum codegree condition, it is still possible that almost all edges of $H$ contain an odd (or even, respectively) number of vertices of $A$, where some natural obstructions for Hamilton $2$-cycles arise.
Garbe and Mycroft~\cite{GaMy} completely characterized such parity obstructions. To avoid cluttering the presentation, we postpone the somewhat lengthy definitions of even-goodness and odd-goodness to Section 2, see Definition~\ref{def:evenoddgood}.

\begin{theorem}\cite{GaMy}
\label{GM}
There exists $\varepsilon>0$ such that the following holds for sufficiently large even $n$.
Let $H$ be a $4$-graph on $n$ vertices with $\delta_3(H)\ge (1/2-\varepsilon) n$.
Then $H$ admits a Hamilton $2$-cycle if and only if every bipartition of $V(H)$ is both even-good and odd-good.
   Moreover, there is a polynomial-time algorithm that runs in time  $O(n^{25})$ and decides whether $H$ contains a Hamilton $2$-cycle.
\end{theorem}

Moreover, they conjectured that the minimum codegree condition $\delta_3(H)\ge \frac n 2 - \varepsilon n$ in Theorem~\ref{GM} can be relaxed to $n/3$.

\begin{conjecture}\cite{GaMy}
\label{conj}
There exists $n_0$ such that the following holds. 
Let $H$ be a $4$-graph on $n\geq n_0$ vertices with $\delta_3(H)\ge \frac n 3$.
Then $H$ admits a Hamilton $2$-cycle if and only if every bipartition of $V(H)$ is both even-good and odd-good.
\end{conjecture}
%Using subroutine Procedure ListPartition developed in \cite{KKM2015}, it is shown in \cite{GaMy} that even-goodness and odd-goodness (for all potential bipartitions) can be checked in time $O(n^{25})$, which resolves the decision problem for Hamilton $2$-cycles.

\subsection{The decision problem ${\rm HAM}(k,\ell,\delta)$}
We now formalize the natural decision problem for Hamilton cycles in dense $k$-graphs first considered by Karpi\'nski, Ruci\'nski and Szyma\'nska~\cite{KRS10}, which generalizes the classical Hamilton cycle problem.
\begin{definition}
For $k, \ell\in \mathbb N$ and $\delta\in [0,1]$ such that $1\le \ell<k$, let ${\rm HAM}(k,\ell,\delta)$ be the decision problem for Hamilton $\ell$-cycle in $k$-graphs $H$ satisfying $\delta_{k-1}(H) \ge \delta n$.
\end{definition}

For instance, ${\rm HAM}(k,\ell,0)$ considers all $k$-graphs, and is a classical NP-complete problem.
A method of Karpi\'nski, Ruci\'nski and Szyma\'nska~\cite{KRS10} shows that ${\rm HAM}(k,\ell,\d)$ is NP-complete for all $\delta$ below the \emph{space barrier} for Hamilton $\ell$-cycles, that is, for 
\[
\delta < \begin{cases}
    \frac1k \hfill\text{ for } k-\ell \mid k, \\
    \frac{1}{\lceil \frac{k}{k-\ell}\rceil(k-\ell)} \text{ for } k-\ell \nmid k.
\end{cases}
\]
On the other hand, the results on minimum codegree forcing $\ell$-Hamiltonicity yield that ${\rm HAM}(k,\ell,\d)$ is (trivially) in P if $\d > 1/2$ for $k-\ell \mid k$ and if $\d > \frac{1}{\lceil \frac{k}{k-\ell}\rceil(k-\ell)}$ otherwise.
Therefore, for $k,\ell$ with $k-\ell \nmid k$, the problem is essentially closed.
For the divisible case, by the aforementioned result, Garbe and Mycroft essentially closed the problem for tight cycles, by showing that ${\rm HAM}(k,k-1,\d)$ is NP-complete for all $\d<1/2$ (indeed, a codegree $n/2-C$ already suffices for intractability).
Theorem \ref{GM} further shows that there is $\e>0$ such that ${\rm HAM}(4,2,1/2-\e)$ is in P, and thus there is a non-trivial gap for the Hamilton cycle problem in this case.

\subsection{Our results.}
In this paper we show an asymptotic form of Conjecture \ref{conj}, that is, we prove it under the minimum codegree $n/3+o(n)$, which also implies tractability for Hamilton $2$-cycles in such $4$-graphs.

\begin{theorem}[Main result]
    \label{main}
  For $\g>0$, the following holds for sufficiently large even $n$.
   Let $H$ be a $4$-graph on $n$ vertices with $\delta_3(H)\ge \frac n 3+\g n$.
   Then $H$ admits a Hamilton $2$-cycle if and only if every bipartition of $V(H)$ is both even-good and odd-good.
      Moreover, there is a polynomial-time algorithm that runs in time  {$O(n^{13})$} and decides whether $H$ contains a Hamilton $2$-cycle.
\end{theorem}

% \begin{theorem}
%     \label{main}
%   Let $\varepsilon>0$, the following holds for sufficiently large even $n$.
%    Let $H$ be a $4$-graph on $n$ vertices with $\delta(H)\ge \frac n 3+\varepsilon n$.
%    Then $H$ admits a Hamilton $2$-cycle if and only if every bipartition of $V(H)$ is both even-good and odd-good.
% \end{theorem}

%\red{JH. Will update here a bit later: if we can show that $H$ has a HC iff $H$ is even-good and odd-good, then we will include the definitions here as a subsection: Parity barriers.}
The moreover part of Theorem \ref{main} says that ${\rm HAM}(4,2,\d)$ is in P for all $\d>1/3$.
Garbe and Mycroft indeed constructed an algorithm that examines whether a given $4$-graph is even-good and odd-good in time $O(n^{25})$ (by checking all potential bipartitions of the vertex set $V(H)$), which they used in \cite{GaMy} to obtain their complexity result as in Theorem \ref{GM}.
By using this algorithm (which they explicitly stated for $4$-graphs $H$ with $\delta_3(H) \ge |H|/3$), we also obtain an algorithm that resolves the decision problem in time $O(n^{25})$.
We note that our proof of Theorem \ref{main} indeed provides a faster algorithm with running time $O(n^{13})$.
The key is in our proof we locate \emph{two} important bipartitions of $V(H)$ so that we only need to check the goodness of them.

Finally, note that Theorem~\ref{main} does not give a polynomial-time algorithm that constructs a Hamilton $2$-cycle, so for the algorithmic aspect of the proof we only focus on the decision problem (see Procedure~\ref{proc:hc}).
Nevertheless, our probabilistic proof of correctness yields a randomized algorithm that constructs the Hamilton cycle with high probability.
% It is not immediately apparent that the criterion of Theorem \ref{main} can be tested in polynomial
% time, but in Section   we will explain why, due to the minimum codegree of $H$, this is in fact the
% case. 
% Consequently, we can determine in polynomial time whether or not a $4$-graph $H$ whose
% codegree is larger that $n/3+o(n)$ admits a Hamilton $2$-cycle. Moreover, by derandomising
% the proof of Theorem \ref{main} we can actually find such a cycle, should one exist, giving the following theorem.
% \begin{theorem}
% \label{mainalgorithm}
%    Let $\varepsilon>0$, the following holds for sufficiently large even $n$.
%     Let $H$ be a $4$-graph with $n$ vertices with $\delta(H)\ge \frac n 3+\varepsilon n$. Then there exists a polynomial-time algorithm with running time {\color{red} $O()$} that can decide $H$ contains a Hamilton $2$-cycle or outputs a certificate that no such cycle exists.
% \end{theorem}

\subsection{Related works}
Since the work of R\"odl, Ruci\'nski and Szemer\'edi~\cite{MR2399020}, there has been a large number of results on minimum degree conditions forcing Hamilton $\ell$-cycles in $k$-graphs, see e.g.~\cite{Markstr2011Perfect,H2010Dirac,2010Loose,MR2652102,VOJTECH2006A,2011Dirac,Bu2013Minimum,Han2015Minimum,Bastos2016Loose,J2017Loose,Polcyn2020Minimum,Reiher2019Minimum,2020Minimum,2020On,han2021minimum,SHWlarge,lang2024tight} and the references therein.

There is also a recent line of work~\cite{KKM2015,Han14_Poly,HT2020,MR4912875} in the past decade focusing on below-guarantee problems for minimum degree conditions in $k$-graphs, (partially) resolving an algorithmic problem of Yuster~\cite{Yuster07} on $F$-factors, the complexity problem for hypergraph perfect matchings of Karpi\'nski, Ruci\'nski and Szyma\'nska~\cite{KRS10} and a more general conjecture of Keevash, Knox and Mycroft~\cite{KKM2015}.

Yet another recent line of work (see e.g.~\cite{jansen2019hamiltonicity,fomin2022algorithmic,gan2023algorithmic}) pursues significant refinements of existing results by looking closely to the critical windows, and the study so far yields strong stability of the algorithmic results and reveals the phase transition as fixed-parameter tractability. 
For an example related to the topic of this paper, Jansen, Kozma, and Nederlof~\cite{jansen2019hamiltonicity} showed that for graphs satisfying $\d(G)\ge n/2-c$, this decision problem for Hamilton cycles is fixed-parameter tractable with parameter  $c$, that is, they constructed an algorithm that resolves the problem in time $2^{O(c)} n^{O(1)}$.

\section{Core lemmas and highlights of the proofs}

% Let $H$ be a $4$-graph.
% Given a bipartition $\{A,B\}$ of $V(H)$, we say that an edge $e\in E(H)$ is \emph{odd} if $|e\cap A|$ is odd (or equivalently, if $|e\cap B|$ is odd) and \emph{even} otherwise.
% Denote the subgraph of $H$ consisting only of the even edges of $H$ by $H_{{\rm even}}$ and the subgraph of $H$ consisting only of the odd edges of $H$ by $H_{{\rm odd}}$. 
% We say that a pair $p$ of distinct vertices of $H$ is a \emph{split} pair if $|p\cap A|=1$ and a \emph{connate} pair otherwise.
% Note that these terms are all dependent on the bipartition $\{A,B\}$ and these will be clear.
% The \emph{total $2$-pathlength} of $H$ is the maximum sum of lengths of vertex-disjoint $2$-paths in $H$.
% For example, $H$ having total $2$-pathlength 2 indicates the presence in $H$ of 2 vertex-disjoint edges (i.e. two $2$-paths of length 1) or a $2$-path of length 2.

In this section we present three key results which together prove Theorem~\ref{main}.
%The following definitions separate the $4$-graphs $H$ into three categories.
We indeed distinguish the $4$-graphs $H$ into three categories and deal with them by very different techniques. 
To motivate them, we first give the following example.
Given a bipartition $A\cup B$, we call a set $S\subseteq A\cup B$ \emph{even} (otherwise \emph{odd}) if $|S\cap A|$ is even.

\begin{example}\label{Exp}
Let $V=A\cup B$ be a bipartition of vertices with $|A|$ odd and $H$ be a $4$-graph that consists of all even $4$-sets.
If $|V|=n$ is even, $|A|\approx n/2$ and $|A|\neq |B|$, then we have $\delta_3(H) \ge n/2-O(1)$.
Such $H$ has no Hamilton $2$-cycle: suppose it contains a Hamilton $2$-cycle $C$, then it induces a cyclic list $\mathcal L$ of $n/2$ pairs of vertices, where every consecutive two pairs form an edge of $H$.
Since all edges of $H$ are even, all consecutive pairs in $\mathcal L$ share the same parity.
If all of them are odd, namely, each pair contains one vertex in $A$ and one vertex in $B$, then it implies $|A|=|B|$, which is impossible; otherwise all of them are even, then summing up all pairs in $\mathcal L$ their intersection with $A$, we obtain that $|A|$ is even, again a contradiction.
\end{example}

It is known that the parity constructions arising from the example above play an essential role on the existence of perfect matchings and  Hamilton $k/2$-cycles in $k$-graphs.
To distinguish $k$-graphs far from such constructions, we introduce the following definition.
Given a set $V$, we write $\binom{V}{k}$ for the set of subsets of $V$ of size $k$. 

% \begin{definition}
% \label{def:evengood}
%     Let $H$ be an $n$-vertex $4$-graph where $n$ is even and let $\{A,B\}$ be a bipartition of $V(H)$.
%     We say that $\{A,B\}$ is \emph{even-good} if at least one of the following statements holds.
%      \begin{enumerate}[$(i)$]
%     \item $|A|$ is even or $|A|=|B|$.
%     \item $H_{{\rm odd}}$ contains edges $e$ an $e'$ such that either $e\cap e'=\emptyset$ or $e\cap e'$ is a split pair.
%     \item $|A|=|B|+2$ and $H_{{\rm odd}}$ contains edges $e$ and $e'$ with $e\cap e'\in\binom{A}{2}$.
%      \item $|B|=|A|+2$ and $H_{{\rm odd}}$ contains edges $e$ and $e'$ with $e\cap e'\in\binom{B}{2}$.  
%        \end{enumerate}

%  Now let $m\in\{0,2,4,6\}$ and $d\in\{0,2\}$ be such that $m\equiv n\ {\rm mod}\ 8$ and $d\equiv|A|-|B|\ {\rm mod}\ 4$.
% Then we say that $\{A,B\}$ is \emph{odd-good} if at least one of the following statements holds.
%   \begin{enumerate}[$(i)$, start=5]
%    \item $(m,d)\in\{(0,0),(4,2)\}$.
%    \item $(m,d)\in\{(2,2),(6,0)\}$ and $H_{{\rm even}}$ contains an edge.
%    \item $(m,d)\in\{(4,0),(0,2)\}$ and $H_{{\rm even}}$ has total $2$-pathlength at least 2.
%    \item $(m,d)\in\{(6,2),(2,0)\}$ and either there is an edge $e\in E(H_{{\rm even}})$ with $|e\cap A|=|e\cap B|=2$ or $H_{{\rm even}}$ has total $2$-pathlength at least 3.
%   \end{enumerate}
% \end{definition}

% \begin{proposition}
% Let $H$ be a $4$-graph. 
% If $H$ contains a Hamilton $2$-cycle, then every bipartition of $V(H)$ is both even-good and odd-good.
% \end{proposition}

%\section{Proof of Theorem \ref{main}}
\begin{definition}[Even-extremal, odd-extremal]
Let $c,\gamma>0$, and $H$ be an $n$-vertex $k$-graph with a bipartition $\{A,B\}$ of $V(H)$.
If $|A|, |B|\ge n/3 + {\gamma n}$ and $H$ contains at most $c{n}^{k}$ odd (resp. even) edges, then we say that $\{A,B\}$ is $(c,\gamma)$-\emph{even-extremal} (resp. $(c,\gamma)$-\emph{odd-extremal}).
If $V(H)$ admits a $(c,\gamma)$-even-extremal (resp. $(c,\gamma)$-odd-extremal) bipartition, then we say that $H$ is $(c,\gamma)$-\emph{even-extremal} (resp. $(c,\gamma)$-\emph{odd-extremal}).
We call $H$ $(c,\gamma)$-extremal if $H$ is either $(c,\gamma)$-{even-extremal} or $(c,\gamma)$-{odd-extremal}.
% Given $c,\gamma>0$, an $n$-vertex $k$-graph $H$ and a bipartition $\{A,B\}$ of $V(H)$.
% If $|A|, |B|\ge n/3 + {\gamma n}$ and $H$ contains at most $c{n}^{k}$ odd edges, then we say that $\{A,B\}$ is $(c,\gamma)$-\emph{even-extremal}.
% If $|A|, |B|\ge n/3 + {\gamma n}$ and $H$ contains at most $c{n}^{k}$ even edges, then we say that $\{A,B\}$ is $(c,\gamma)$-\emph{odd-extremal}.
% If $V(H)$ admits a $(c,\gamma)$-even-extremal bipartition ($(c,\gamma)$-odd-extremal bipartition), then we say that $H$ is $(c,\gamma)$-\emph{even-extremal} ($(c,\gamma)$-\emph{odd-extremal}).
% We call $H$ $(c,\gamma)$-extremal if $H$ is either $(c,\gamma)$-{even-extremal} or $(c,\gamma)$-{odd-extremal}.
\end{definition}

Now suppose that $H$ is a $4$-graph. 
%Given a bipartition $\{A,B\}$ of $V(H)$, we say that an edge $e\in E(H)$ is \emph{odd} if $|e\cap A|$ is
%odd, or equivalently, if $|e\cap B|$ is odd, and \emph{even} otherwise.
We denote the subgraph of $H$ consisting only of the even edges of $H$ by
$H_{\mathrm{even}}$, and the subgraph of $H$ consisting only of the odd
edges of $H$ by $H_{\mathrm{odd}}$. Also, we say that a pair $p$ of
distinct vertices of $H$ is a \emph{split pair} if $|p\cap A|=1$, and a
\emph{connate pair} otherwise. 
These terms are all defined with respect to a
bipartition $\{A,B\}$ of $V(H)$.
The \emph{total $2$-pathlength} of $H$ is the maximum sum of lengths of vertex-disjoint $2$-paths in $H$. For example, $H$ having total $2$-pathlength $3$ indicates the presence in $H$ of three disjoint edges (i.e. three $2$-paths of length $1$), or of a $2$-path of length $3$, or of two vertex-disjoint $2$-paths, one of length $1$ and one of length $2$.

\begin{definition}\label{def:evenoddgood}
Let $H$ be an $n$-vertex 4-graph where $n$ is even.
Let $\{A,B\}$ be a bipartition of $V$.
We say that $\{A,B\}$ is
\emph{even-good} if at least one of the following conditions holds:
\begin{enumerate}[label=$(E\arabic*)$, ref=$E\arabic*$, start=0]
    \item $|A|$ is even, or $|A|=|B|$;
    \label{E0}
    \item $H_{\rm odd}$ contains two vertex-disjoint edges;
    \label{E1}
    \item $H_{\rm odd}$ contains two edges $e,e'$ such that
    $e\cap e'$ is a split pair;
    \label{E2}
    \item $|A|=|B|+2$, and $H_{\rm odd}$ contains two edges
    $e,e'$ such that $e\cap e'\in \binom{A}{2}$;
    \label{E3}
    \item $|B|=|A|+2$, and $H_{\rm odd}$ contains two edges
    $e,e'$ such that $e\cap e'\in \binom{B}{2}$.
    \label{E4}
\end{enumerate}
Now let $m \in \{0,2,4,6\}$ and $d \in \{0,2\}$ be such that
$m \equiv n \pmod{8}$ and $d \equiv |A|-|B| \pmod{4}$.
We say that $\{A,B\}$ is \emph{odd-good} if at least one of the
following statements holds.

\begin{enumerate}[label=$(O\arabic*)$, ref=$O\arabic*$, start=0]
    \item $(m,d) \in \{(0,0),(4,2)\}$.
    \label{O0}
    \item $(m,d) \in \{(2,2),(6,0)\}$ and $H_{\mathrm{even}}$
    contains an edge.
     \label{O1}
    \item $(m,d) \in \{(4,0),(0,2)\}$ and $H_{\mathrm{even}}$
    has total $2$-pathlength at least two.
     \label{O2}
    \item $(m,d) \in \{(6,2),(2,0)\}$ and either there is an edge
    $e \in E(H_{\mathrm{even}})$ with
    $|e \cap A| = |e \cap B| = 2$, or $H_{\mathrm{even}}$ has total
    $2$-pathlength at least three.
     \label{O3}
\end{enumerate}
\end{definition}

For an integer $k$ we write $[k]$ for the set of integers from $1$ to $k$.
We write $x \ll y \ll z$ to mean that we choose constants in a hierarchical manner from right to left, that is, for any $z > 0$, there exist functions $f$ and $g$ such that, whenever $y \leq f(z)$ and $x \leq g(y)$, the subsequent statement holds. 
Hierarchies of other lengths are defined analogously.
We omit floors and ceilings throughout this paper where they do not affect the argument.
Throughout the paper, for disjoint sets $A, B$ define $A^iB^j$ as the family of $(i+j)$-sets that consist of $i$ elements of $A$ and $j$ elements of $B$.

Now we present our core lemmas. 

\begin{lemma}[\textbf{Non-extremal case}]
\label{nonextremal}
Suppose that $1/n\ll \e \ll \gamma$ with $n\in 2\mathbb N$ and let $H$ be an $n$-vertex $4$-graph with $\delta_3(H)\geq n/3+\gamma n$.
Then either $H$ is $(65\e,\gamma/3)$-extremal, or $H$ contains a Hamilton $2$-cycle.
Moreover, in the former case a $(65\e,\gamma/3)$-extremal partition can be found in time $O(n^{13})$.
\end{lemma}

\begin{lemma}[\textbf{Even-extremal case}]
\label{evenextremal}
Suppose that $1/n\ll c'\ll\gamma\ll1$ and let $H$ be an $n$-vertex $4$-graph with $\delta_3(H)\geq n/3+\gamma n$ and $\{A',B'\}$ be a $(c',\g/3)$-even-extremal bipartition of $V$.
Then we construct a bipartition $\{A, B\}$ of $V$ in time $O(n^4)$ such that $H$ contains a Hamilton $2$-cycle if and only if $\{A, B\}$ is even-good, which can be checked in time $O(n^8)$.
\end{lemma}

\begin{lemma}[\textbf{Odd-extremal case}]
\label{oddextremal}
Suppose that $1/n\ll c'\ll\gamma\ll1$ and let $H$ be an $n$-vertex $4$-graph with $\delta_3(H)\geq n/3+\gamma n$ and $\{A',B'\}$ be a $(c',\g/3)$-odd-extremal bipartition of $V$.
Then we construct a bipartition $\{A, B\}$ of $V$ in time $O(n^4)$ such that $H$ contains a Hamilton $2$-cycle if and only if $\{A, B\}$ is odd-good, which can be checked in time $O(n^{12})$.
\end{lemma}

Therefore, a proof of Theorem \ref{main} combines the three lemmas above in the na\"ive way: we apply Lemma \ref{nonextremal} to $H$, which either constructs a Hamilton $2$-cycle or outputs an extremal partition, and then the extremal partition is sent to Lemma \ref{evenextremal} or Lemma \ref{oddextremal} accordingly.

The necessity of Lemmas \ref{evenextremal} and \ref{oddextremal} follows immediately from Lemma~\ref{lem:hcevenoddgood} below. Indeed, if $H$ contains a Hamilton $2$-cycle, then every bipartition of $V(H)$ is both odd-good and even-good. 
It remains to prove the sufficiency parts, which are presented in Sections \ref{sec:even} and \ref{sec:odd}.

\begin{lemma}\cite[Proposition 3.1]{GaMy}
\label{lem:hcevenoddgood}
Let $H$ be a $4$-graph.
If $H$ contains a Hamilton $2$-cycle, then every bipartition of $V$ is both even-good and odd-good.
\end{lemma}

\subsection{The main algorithm}

Now we state our algorithm for the decision problem HAM$(4,2,\d)$ for $\d>1/3$ (See Procedure \ref{proc:hc}).
We omit the standard reductions: if $n$ is not even, then we halt with negative output; if $n$ is bounded (too small to apply any of the results in the algorithm), then we solve the problem by brute force in $O(1)$ time.
Otherwise, Procedure \ref{proc:hc}($H$) applies Lemma~\ref{lem:conn} and Lemma~\ref{lem:abs} to $H$.
If they construct the absorbing set and the reservoir set successfully, then we output that $H$ contains a Hamilton 2-cycle; otherwise either of it outputs a $(65\varepsilon,\gamma/3)$-extremal bipartition of $V(H)$, and then we execute either Procedure \ref{proc:evenext}$(H)$ or Procedure OddExt$(H)$ accordingly.

\medskip

\begin{procedure}[ht]
\SetAlgoVlined
\SetAlgoNoEnd
\caption{HamCycle($H$)}
\KwData{An $n$-vertex $4$-graph $H$ with $\delta_3(H)\ge (1/3+\gamma)n$.}
\KwResult{YES if $H$ contains a Hamilton $2$-cycle or a certificate for non-existence.}
\BlankLine

Apply Lemma~\ref{lem:conn} to $H$.\\
\If{Lemma~\ref{lem:conn} returns a $(65\varepsilon,\gamma/3)$-even-extremal bipartition $\{A',B'\}$ of $V(H)$}{
    \Return{\ref{proc:evenext}$(H,\{A',B'\})$}
}

\Else{Apply Lemma~\ref{lem:abs} to $H$.\\
\If{Lemma~\ref{lem:abs} returns a $(65\varepsilon,\gamma/3)$-extremal bipartition $\{A',B'\}$ of $V(H)$}{
    \If{$\{A',B'\}$ is a $(65\varepsilon,\gamma/3)$-even-extremal bipartition}{
    \Return{\ref{proc:evenext}$(H,\{A',B'\})$}
    }
    \ElseIf{$\{A',B'\}$ is a $(65\varepsilon,\gamma/3)$-odd-extremal bipartition}{
        \Return{\text{OddExt}$(H,\{A',B'\})$}
    }   
    }
     \Else{\Return{YES.}
}
}
\label{proc:hc}
\end{procedure}
\medskip

\begin{procedure}[H]
\SetAlgoVlined
\SetAlgoNoEnd
\caption{EvenExt($H,\{A',B'\}$)}
\KwData{An $n$-vertex $4$-graph $H$ with $\delta_3(H)\ge (1/3+\gamma)n$ and a $(c',\g/3)$-even-extremal bipartition $\{A',B'\}$ of $V(H)$.}
\KwResult{YES if $H$ contains a Hamilton $2$-cycle; otherwise, a certificate
for non-existence.}
\BlankLine
Set $A_{\text{bad}}=\{a\in A':a\ \text{is}\ (\beta_1',\beta_2')\text{-bad with respect to }\{A',B'\} \}$.\\ Set $B_{\text{bad}}=\{b\in B':b\ \text{is}\ (\beta_1',\beta_2')\text{-bad with respect to } \{A',B'\}\}$.\\
Set $A=(A'\setminus A_{\mathrm{bad}})\cup B_{\mathrm{bad}}$ and $B=(B'\setminus B_{\mathrm{bad}})\cup A_{\mathrm{bad}}$.
\BlankLine
\If{$\{A,B\}$ is even-good}{
 {\Return{YES}}.}
\Else{
  \Return{NO, with $\{A,B\}$ as a certificate.}}
\label{proc:evenext}
\end{procedure}
\medskip
We omit the statement of \text{OddExt}($H,\{A',B'\}$) for an odd-good bipartition $\{A,B\}$ of $V(H)$, which is identical to \ref{proc:evenext}($H,\{A',B'\}$) by replacing ``even'' by ``odd''.
%So we omit it.
%\medskip

%\red{JH. Need to add the algorithm here, maybe assuming subroutine Procedure HAMCYCLE NONEXT.}

\subsection{Proof ideas}
It remains to prove the three lemmas above, and now we briefly discuss on the proof ideas.
%\red{JH. working.}

\medskip
\noindent\textbf{The non-extremal case.}
For the case where $H$ is far away from the extremal structure, the celebrated absorption framework has been shown to be an essential tool, particularly for $k$-graphs, as the regularity--blow up line of tools have only limited applications to the $k$-graph case.
As in the example above, one essential difficulty for constructing a Hamilton cycle is to cover precisely \textit{all} vertices of the $k$-graph.
Roughly speaking, the absorption method reduces the construction of a Hamilton cycle into three tasks.
It first finds a short path (absorbing path) which has the property that one can enlarge it by adding a small number but arbitrary collection of vertices, while keeping the end vertices unchanged.
Then the remaining task is to extend the absorbing path to a long cycle covering all but $o(|H|)$ vertices of $H$.
This normally splits to two separate lemmas: the path cover lemma that finds a constant number of vertex-disjoint paths whose union covers all but $o(n)$ vertices of $H$, and a connecting lemma that can connect a constant number of disjoint paths end-to-end to a cycle.

For our proof of Lemma~\ref{nonextremal}, we use the same absorbing strategy, but with essential improvements in several aspects.
First, the path cover part is standard and follows essentially from existing works.
Second, the absorbing part combines the recent swap-absorb idea and the lattice-based absorption, which has been successfully used in recent works~\cite{HaKe20,HSW26}.
This allows us to use the reachability theory to classify the vertices that can play similar roles in a path.

% The most important novel ideas lie in the proof of the connecting lemma, where we classify the robustly connectable pairs.
% Roughly speaking, with a minimum codegree condition $\delta_3(H)\geq n/3+\gamma n$, it is easy to show that among every three pairs of vertices, there exist two of them that have a significant number of common neighbours and thus are robustly connectable.
% Using this observation, we show that there are essentially two robustly connectable components defined on pairs of vertices, so on $\binom{V(H)}2$.
% The difficult part is to show that this separation of components must induce a bipartition of $V(H)$ so that almost all edges are \textit{even}.
% This step is done by a careful study of the two components of pairs viewed as two edge-disjoint graphs $R$ and $B$ defined on $V(H)$ which are almost complement to each other.
% Then a bipartition of $V(H)$ is defined via symmetric differences of the neighborhoods in $R$ and in $B$: after choosing appropriate $v\in V(H)$, we show that almost every other vertex $u$ must satisfy either i) $|N_R(u)\triangle N_R(v)|=o(n)$ and  $|N_B(u)\triangle N_B(v)|=o(n)$ or ii) $|N_B(u)\triangle N_R(v)|=o(n)$ and  $|N_R(u)\triangle N_B(v)|=o(n)$.
The main new ingredient in the connecting lemma is a pair-level stability argument which replaces the direct common-neighbour method used near the Dirac threshold. In the earlier $n/2-o(n)$ setting, two endpoint pairs either have many short connecting paths or the usual red/blue colouring on the vertex set quickly yields an even-extremal bipartition. At the lower threshold $\delta_3(H)\ge (1/3+\gamma)n$, such direct overlap is no longer available. Instead, we work in the auxiliary graph whose vertices are pairs of vertices of $H$, and study robust connectability among these pair-vertices. We first show that every pair is one-step connectable to a positive proportion of all pairs; if two pairs are nevertheless not connected within bounded length, their one-step reachable sets separate almost all pairs into two large classes $\Q_1,\Q_2$. The absence of many connectors then forces very few hyperedges to be formed by one pair from each class. Treating $\Q_1,\Q_2$ as red and blue graphs, this gives a strong compatibility rule for almost all hyperedges, from which we recover an even-extremal bipartition. Thus the improvement is that global pair-connectivity at the $1/3$ threshold is obtained through a stability analysis of the pair graph, rather than through direct vertex-level neighbourhood overlap.

\medskip
\noindent\textbf{The extremal cases.}
The proof of the extremal cases rely heavily the (strong) structure of $H$.
Although it splits to two very different cases, their proofs share a lot of similarities.
The most important point of the proofs, in both cases, is to find a short path (bridge) that serves as a parity breaker that overcomes the parity issues (cf. Example~\ref{Exp}).
Obviously, such a path must utilize the \textit{minority} edges.
After that, the plan is to cover the vertices atypical to the structure of $H$ via short paths and attach them to the bridge and then extend it to a Hamilton cycle.

Garbe--Mycroft~\cite{GaMy} had an easier task in the sense that due to the minimum codegree condition $n/2-o(n)$, in both cases, the bipartition of $V(H)$ is almost balanced, and the majority edges are \emph{almost full}.
This is no longer true in our setting, and for each type of edges, the density might be barely above $1/2$.
This poses significant challenges to the steps of building the bridge and completing the Hamilton cycle (finding a Hamilton path in the remaining $4$-graph with given ends), due to the lower codegree condition.
More precisely, most of the auxiliary structures in the proofs of~\cite{GaMy} are almost complete and in our work either we only have a much sparser one or the auxiliary structure completely breaks down so that we have to look for alternative ones.
%\red{JH. Maybe add an example }
%\section{Extremal Case}

At last, we emphasize again that our algorithm focuses on the decision problem which allows us to use probabilistic methods to show the existence of the Hamilton cycle in our proof of correctness of the algorithm.

\section{The non-extremal Case }

Under the celebrated framework of the absorbing method, the proof of the non-extremal case consists of the following three lemmas.

Let us include a few definitions.
Let $\binom V2$ be the family of pairs of vertices of $H$, where $V=V(H)$.
Given $e, f\in \binom V2$, $\beta>0$ and $t\in \mathbb N$, we say that $e$ and $f$ are \textit{$(\beta, t)$-connectable} if there exist at least $\beta n^{2i}$ choices of sets $v_1\dots v_{2i}$ such that $e v_1\cdots v_{2i} f$ is a $2$-path in $H$ for some $i\in [t]$ \footnote{Concatenation works for the connectability defined here, however, one may notice that disjoint $t$-connectable set for $u, v$ and $t'$-connectable set for $v, w$ give rise to $(t+t'+1)$-connectable sets for $u$ and $w$, different from the normal setting.}.
Each such $2i$-set is called a \emph{$2i$-connector} for $e$ and $f$.

\begin{lemma}[Connection]
\label{lem:conn}
Suppose $1/n\ll \a\ll \beta \ll \e \ll \gamma$.
Let $H$ be an $n$-vertex $4$-graph with $\delta_{3}(H)\ge (1/3+\gamma)n$.
% which is not $65\e$-even-extremal.
Then one of the following holds.
\begin{itemize}
\item Every two pairs of vertices of $H$ are $(\beta, 4)$-connectable to each other, or
\item we can find a $(65\e,\gamma/3)$-even-extremal bipartition in time $O(n^{12})$.
\end{itemize}
Moreover, in the former case, for any given vertex set $W$ of size at most $3\beta^2 n$, there is a vertex set $R\subseteq V(H)\setminus W$ of size $\a n$ such that for every two pairs of vertices there are at least $\alpha^{9} n$ vertex-disjoint connectors in $R$ each of size at most 8.
\end{lemma}

\begin{lemma}[Absorption]
\label{lem:abs}
Suppose $1/n\ll \alpha \ll \beta\ll \e\ll \gamma$.
Let $H$ be an $n$-vertex $4$-graph with $\delta_3(H)\geq n/3+\gamma n$. Then either $H$ is $(65\e,\gamma/3)$-extremal, or $H$ contains a $2$-path $P_{A}$ of length at most $\beta^2 n$ associated with a set of vertices $S^{*}$ of size at most $7\alpha n$ such that for any set $U$ of vertices satisfying that $|U|\le 2\alpha n$, $U\cap (V(P_{A})\cup S^{*})=\emptyset$, and $|U\cup S^*|\in 2\mathbb N$, there is another $2$-path on $V(P_{A})\cup U\cup S^{*}$ which has the same ends as $P_{A}$.
Moreover, in the former case a $(65\e,\gamma/3)$-extremal bipartition can be found in time $O(n^{13})$.
\end{lemma}

Our path cover lemma aims to find a constant number of vertex-disjoint paths whose union covers almost all vertices of the hypergraph.
We state and prove a much stronger result for $k$-graphs: we only need a minimum codegree essentially $n/k$, and the lemma indeed finds \emph{tight} paths.

\begin{lemma}[Path cover]
\label{pathcover}
   Suppose  
$1/n \ll 1/p \ll \alpha, \gamma, 1/k$. 
Let $H$ be a $k$-graph on $n$ vertices with  
$\delta_{k-1}(H) \geq (1/k+\gamma)n$. Then we can find a family of vertex-disjoint tight paths in $H$ consisting of at most $p$ paths, whose union covers all but at most $\alpha n$ vertices of $H$.
\end{lemma}

Note that a tight path on $m$ vertices contains a $k/2$-path as a spanning subgraph if $m\in (k/2)\mathbb N$.
So in our proof of Lemma~\ref{nonextremal} we can use Lemma~\ref{pathcover} at the price of discarding at most one vertex from each tight path returned by the lemma and then regard them as $2$-paths.

We now prove Lemma~\ref{nonextremal} by combining these three lemmas.

%As in typical applications of the absorbing method, we first find an absorbing path $P_{A}$ by Lemma \ref{lem:abs}, which has the property that 

\begin{proof}[Proof of Lemma~\ref{nonextremal}]
Choose additional constants satisfying $1/n \ll 1/p \ll \alpha \ll \beta \ll \e \ll \gamma$ and $n\in 2\mathbb N$.
Let $H$ be an $n$-vertex $4$-graph with  
$\delta_3(H) \geq (1/3+\gamma)n$. 
We start with applying Lemmas~\ref{lem:conn} and~\ref{lem:abs} to $H$.
If either lemma finds a $(65\e,\gamma/3)$-extremal bipartition of $H$ (in time $O(n^{13})$), then we output the bipartition and halt.
So we may assume that neither of the lemmas finds a $(65\e,\gamma/3)$-extremal bipartition of $H$ and thus we can use the structural properties in them.
Below we show the existence of a Hamilton $2$-cycle.

Let $P_{A}$ be the absorbing path and $S^{*}$ be the set returned by Lemma~\ref{lem:abs}.
Let $W:=V(P_{A})\cup S^{*}$, and note that $|W|\le 2\b^2 n + 2 + 7\a n\le 3\b^2 n$.
Then by Lemma~\ref{lem:conn}, take $R\subseteq V(H)\setminus W$ with properties as stated in the lemma. 

%Let $H':=H-(V(P_{A})\cup S^{*})$ be the induced subgraph of $H$ obtained by removing $V(P_{A})\cup S^{*}$.
%Note that $|V(H')|\ge n - (2\beta n +2) - 7\alpha n \ge (1-3\beta)n$, and thus $\delta_{3}(H')\ge (1/3+\gamma/2)|V(H')|$ as $\b \ll \gamma$.
%We next apply Lemma~\ref{lem:conn} to $H'$, which either outputs a $65\e$-even-extremal partition of $H'$, or a reservoir set $R$ of vertices with properties as stated in the lemma.
%If it outputs a $65\e$-even-extremal partition of $H'$, by adding $V(P_{A})\cup S^{*}$ to any side of the bipartition, we obtain a bipartition of $V(H)$ that has at most $65\e n^{4} + 3\beta n^{4}\le 66\e n^{4}$ odd edges and has both parts of size at least $(1/3+\gamma/3)n$.
%Therefore, it is a $66\e$-even-extremal partition of $H$.
%We then output it and halt.
%So it remains to consider the case Lemma~\ref{lem:conn} outputs $R$.

Now let $H':=H-(W\cup R)$, and as $|W|\le 3\b^2 n$ and $|R|\le \e n$, we have that $|V(H')|\ge n -  2\e n$ and $\delta_{3}(H')\ge (1/3+\gamma/3) |V(H')|$.
By Lemma~\ref{pathcover} applied to $H'$, we find a family $\Q$ of $p_0$ disjoint tight paths of $H'$ with $p_0\le p$, whose union covers all but at most $\alpha n$ vertices.
As explained above this proof, we may discard at most one vertex in each tight path and obtain a family of $p_0$ disjoint $2$-paths whose union covers all but at most $\alpha n + p_0$ vertices of $H'$.
Then we use the property of $R$ and greedily connect the paths in $\Q$ and $P_{A}$ to a $2$-cycle $C$ of $H$.
This is possible because we need to do $p_0+1$ connections, each of which consumes at most $8$ vertices of $R$, while the property of $R$ guarantees that every two pairs of vertices have at least $\alpha^{9} n > 8(p_0 +1) $ vertex-disjoint connectors.

Now we obtain a $2$-cycle $C$ of $H$ and let $U=V(H)\setminus (V(C)\cup S^{*})$.
Then $|U|\le |R|-2(p_0+1) +\alpha n+p_0 < 2\alpha n$ and $U\cap (V(P_{A})\cup S^{*})=\emptyset$.
Moreover, as $C$ is a $2$-cycle, we have $|V(C)|$ is even, which together with $n\in 2\mathbb N$, implies that $|U\cup S^*|=n-|V(C)|$ is even.
Therefore, by Lemma~\ref{lem:abs},
% by the absorbing property of $P_{A}$ and $S^{*}$, 
we can find a $2$-path $P'$ on $V(P_{A})\cup U\cup S^{*}$ which has the same ends as $P_{A}$.
Thus, in $C$ we can replace the subpath $P_{A}$ by $P'$ which gives a Hamilton $2$-cycle of $H$.
%
%The proof is completed and the overall running time is $O(n^{13})$.
\end{proof}

\subsection{The path cover lemma}
We should remark here that a proof of Lemma~\ref{pathcover} is essentially presented as the proof of~\cite[Lemma 2.3]{MR2399020} -- even though the statement of their lemma assumes minimum codegree $\delta_{k-1}(H)\ge (1/2+\gamma)n$, their proof indeed only uses the weaker one $\delta_{k-1}(H)\ge (1/k+\gamma)n$ as in our lemma.
We nevertheless include such a proof using the components established in~\cite{MR2399020}.
%and shall not claim any credit for it.

In this section we provide a proof of the path cover lemma using the regularity method and results from hypergraph matchings, which are standard now for finding large structures in (hyper)graphs.
% prove stronger results: we 
%Most of the components of the proofs are somewhat either known or standard.
We will use the weak regularity lemma for hypergraphs, a straightforward extension of Szemer\'{e}di's regularity lemma for graphs~\cite{MR540024}.

Let $H=(V, E)$ be a $k$-graph and let $A_1, \ldots, A_k$ be mutually disjoint non-empty subsets of $V$. We define $e(A_1, \ldots, A_k)$ to be the number of edges with one vertex in each $A_i$, $i \in [k]$, and the density of $H$ with respect to $(A_1, \ldots, A_k)$ as
\[
d(A_1, \ldots, A_k) = \frac{e(A_1, \ldots, A_k)}{|A_1| \cdots |A_k|}.
\]

Given $\varepsilon, d \geq 0$, a $k$-tuple $(V_1, \ldots, V_k)$ of mutually disjoint subsets $V_1, \ldots, V_k \subset V$ is $(\varepsilon, d)$-regular if
\[
|d(A_1, \ldots, A_k) - d| \leq \varepsilon
\]
for all $k$-tuples of subsets $A_i \subset V_i$, $i \in [k]$, satisfying $|A_i| \geq \varepsilon |V_i|$. We say $(V_1, \ldots, V_k)$ is $\varepsilon$-regular if it is $(\varepsilon, d)$-regular for some $d \geq 0$. It is immediate from the definition that in an $(\varepsilon, d)$-regular $k$-tuple $(V_1, \ldots, V_k)$, if $V_i' \subset V_i$ has size $|V_i'| \geq c|V_i|$ for some $c \geq \varepsilon$, then $(V_1', \ldots, V_k')$ is $(\varepsilon/c, d)$-regular.

\begin{theorem}\cite{MR540024}
\label{regularity}
For all $t_0 \geq 0$ and $\varepsilon > 0$, there exist $T_0 = T_0(t_0, \varepsilon)$ and $n_0 = n_0(t_0, \varepsilon)$ so that for every $k$-graph $\mathcal{H} = (V, E)$ on $n > n_0$ vertices, then in time $O(n^{2.376})$, we can find a partition $V = V_0 \cup V_1 \cup \cdots \cup V_t$ such that

\begin{enumerate}[label=(\roman*)]
    \item $t_0 \leq t \leq T_0$,
    \item $|V_1| = |V_2| = \cdots = |V_t|$ and $|V_0| \leq \varepsilon n$,
    \item for all but at most $\varepsilon \binom{t}{k}$ sets $\{i_1, \ldots, i_k\} \in \binom{[t]}{k}$, the $k$-tuple $(V_{i_1}, \ldots, V_{i_k})$ is $\varepsilon$-regular.
\end{enumerate}
\end{theorem}

A partition as given in Theorem~\ref{regularity} is called an $\varepsilon$-regular partition of $\mathcal{H}$. For an $\varepsilon$-regular partition of $\mathcal{H}$ and $d \geq 0$ we refer to $\mathcal{Q} = \{V_i\}_{i \in [t]}$ as the family of clusters and define the cluster hypergraph $K = K(\varepsilon, d, \mathcal{Q})$ with vertex set $[t]$ and $\{i_1, \ldots, i_k\} \in \binom{[t]}{k}$ is an edge if and only if $(V_{i_1}, \ldots, V_{i_k})$ is $\varepsilon$-regular and $d(V_{i_1}, \ldots, V_{i_k}) \geq d$.
The following corollary shows that the cluster hypergraph inherits the minimum codegree condition of the original hypergraph.
The proof is standard and very similar to Proposition 16 in \cite{H2010Dirac}, so we omit it.
\begin{lemma}\cite{H2010Dirac}
\label{inherit}
    Let $\varepsilon\ll d\ll \g$.
    For an integer $k\geq3$  and an integer $t_0\geq2k^2/d$, there exist $T_0$ and $n_0$ such that the following holds.
    Given a $k$-graph $H$ on $n\geq n_0$ vertices with $\delta_{k-1}(H)\geq(1/k+\gamma)n$, there exists an $\varepsilon$-regular partition $V=V_0\cup V_1\cup\cdots\cup V_t$ with $t_0\leq t\leq T_0$.
    Furthermore, let $K=K(\varepsilon,d,\mathcal{Q})$ be the cluster hypergraph of $H$.
    Then the number of $(k-1)$-sets $S$ violating $\deg_K(S)\geq(1/k+\gamma/4)t$ is at most $\sqrt{\varepsilon}t^{k-1}$.
\end{lemma}

We use the following result from \cite{MR2399020} on finding disjoint tight paths in $\e$-regular $k$-tuples.
%\begin{proposition}[\cite{MR2399020}, Claim 4.1]
%    Given $c>0$ and $k\geq2$, every $k$-partite $k$-graph $H$ with at most $m$ vertices in each part and with at least $cm^k$ edges contains a path with at least $cm$ vertices.
%    Furthermore, this can be found in time $O(m^k)$.
%\end{proposition}

%\red{JH. I've revised this lemma below to tight paths. Should cite this lemma appropriately.}

\begin{lemma}\cite[Claim 4.2]{MR2399020}
\label{reducedpathcover}
    Fix an integer $k\geq 3$ and $\varepsilon,d>0$ such that $d>2\varepsilon$.
    Let $m>\frac{k}{\varepsilon(d-\varepsilon)}$.
    Suppose $(V_1,\ldots, V_k)$ is an $(\varepsilon,d)$-regular $k$-tuple with $|V_i|=m$ for each $i\in[k]$.
    Then we can find a family of at most $\frac{k}{(d-2\varepsilon)\varepsilon}$ pairwise vertex-disjoint tight paths whose union covers all but at most $k\varepsilon m$ vertices of $\bigcup_{i\in[k]}V_i$.
\end{lemma}

%For $k\geq 3$ and
%$\varepsilon>0$, we say that a $k$-graph $H$ is $\varepsilon$-\emph{extremal} if
%there is a vertex set $S\subseteq V(H)$ of size
%$\left\lfloor \frac{k-1}{k}n\right\rfloor$ such that
%$e(H[S])\leq \varepsilon n^k$.
%We use the following result of \cite{amlostmatchinggh}.
%    \begin{lemma}[\cite{amlostmatchinggh}, Lemma 4.3]\label{lem:almost-perfect-matching}
%For any integer $k\geq 3$ and $0<\varepsilon\ll \beta,\eta$, the following
%holds for sufficiently large $n$. Let $H=(V,E)$ be an $n$-vertex $k$-graph
%such that all but at most $\varepsilon n^{k-1}$ $(k-1)$-sets $S\subseteq V$
%satisfy that $\deg(S)\geq n/k-\eta n$. If $H$ is not $\eta$-extremal,
%then $H$ contains a matching that covers all but at most $\beta n$ vertices
%of $V$.
%\end{lemma}

%The following is a corollary of Lemma \ref{lem:almost-perfect-matching}.
We also use the following result  from \cite{MR2399020} on almost perfect matchings in the reduced $k$-graph, under a defect minimum codegree condition.

\begin{lemma}\cite[Claim 4.5]{MR2399020}
\label{almostperfect}
      Let $\beta,\gamma>0$, and suppose $1/n\ll \varepsilon\ll\eta \ll \beta,\gamma \ll1/k$. 
    Let $H$ be an $n$-vertex $k$-graph such that all but at most $\sqrt\varepsilon n^{k-1}$ $(k-1)$-sets $S\subseteq V(H)$ satisfy that  $\deg(S)\ge \frac n k+\gamma n$. Then $H$ contains a matching that covers all but at most $\beta n$ vertices of $V$.
\end{lemma}

Now we are ready to prove Lemma \ref{pathcover}.
\begin{proof}[Proof of Lemma \ref{pathcover}]
 Suppose $1/n \ll 1/p \ll 1/T_0 \ll 1/t_0 \ll \varepsilon\ll d \ll\beta\ll \alpha \ll \gamma, 1/k$.
Suppose $H$ is an $n$-vertex $k$-graph with $\delta_{k-1}(H) \geq (\frac{1}{k}+\gamma)n$. 
We apply Lemma \ref{inherit} with parameters $\varepsilon$, $\gamma$, $d$ and $t_0$ obtaining an $(\varepsilon, t)$-regular partition $V=V_0\cup V_1\cup\cdots\cup V_t$ with $t_0 \leq t \leq T_0$.
Let $\mathcal{Q} = \{V_i\}_{i \in [t]}$ and the cluster hypergraph $K = K(\varepsilon, d, \mathcal{Q})$ with vertex set $[t]$. Let $m \geq \frac{(1-\varepsilon)n}{t}$ be the size of each cluster $V_i$, $i \in [t]$. 
By Lemma \ref{inherit}, we obtain that for all but at most $\sqrt{\varepsilon}t^{k-1}$ many $(k-1)$-sets $S$,
\[
\deg_K(S) \geq \left( \frac{1}{k}+\frac{\gamma}{4}\right) t.
\]

Thus by Lemma \ref{almostperfect} with $\gamma/4$ in place of $\gamma$, $K$ contains a matching $M$ covering all but at most $\beta t$ vertices. For each edge $\{i_1, \ldots, i_k\} \in M$, the corresponding clusters $(V_{i_1}, \ldots, V_{i_k})$ is $(\varepsilon, \gamma')$-regular for some $\gamma' \geq d$. 
Thus we can apply Lemma \ref{reducedpathcover} to $(V_{i_1}, \ldots, V_{i_k})$ and get a family of at most $\frac{k}{(d-2\varepsilon)\varepsilon}$ tight paths leaving at most $k\varepsilon m$ vertices uncovered. We do this for each edge in $M$ and get at most $\frac{t}{k} \cdot \frac{k}{(d-2\varepsilon)\varepsilon} \le \frac{T_{0}}{(d-2\varepsilon)\varepsilon}\le p$ vertex-disjoint tight paths, which leaves at most
\[
|V_0| + k\varepsilon m \cdot \frac{t}{k} + \beta t \cdot m \leq \varepsilon n + \varepsilon n + \beta n \leq \alpha n
\]
vertices uncovered in $H$, as $\e, \beta \ll \a$.
\end{proof}

\subsection{Connection}
%\red{JH. Connection is done EXCEPT the construction of reservoir, in the proof of Lemma \ref{lem:conn}.}
As in other proofs using the absorbing method, we need a lemma that efficiently connects a small number of paths to a long one, using a small number of vertices.
We show that this is achievable if $H$ is not even-extremal.
For the above claimed connection lemma to work, it seems that a necessary condition is that the $4$-graph should be ``connected'' for $2$-paths in a robust sense.
This is trivial if $\delta_2(H)\ge (1/2+o(1))\binom{n}{2}$, as then any two pairs of vertices shall have many common neighbours.
However, when the minimum codegree drops below the threshold $1/2$, this is no longer the case, that is, we cannot connect arbitrary pairs of vertices.
%This motivates us to define the study ``connected components'' formed by the pairs of vertices.
%Indeed, we shall show that thanks to codegree condition $\delta_3(H)\ge (1/3+o(1)){n}$, there are at most two components and in which case we can study the structure of $H$, via an approach similar to the reachability (defined for vertices of $H$).

%Let us introduce the following notation.
%Let $\binom V2$ be the family of pairs of vertices of $H$, where $V=V(H)$.
%Given $e, f\in \binom V2$, $\beta>0$ and $t\in \mathbb N$, we say that $e$ and $f$ are \textit{$(\beta, t)$-connectable} if there exist $i\in [t]$ and $\beta n^{2i}$ $2i$-sets of vertices $v_1\dots v_{2i}$ such that $e v_1\cdots v_{2i} f$ is a $2$-path in $H$ \footnote{Concatenation still works for the connectability defined here, however, one may notice that disjoint $t$-connectable set for $u, v$ and $t'$-connectable set for $v, w$ give rise to $(t+t'+1)$-connectable sets for $u$ and $w$, different from the reachability.}.
%Each such $2i$-set is called a \emph{$2i$-connector} for $e$ and $f$.
%Then we can use the reachability of Lo and Markstr\"om and indeed show that either all pairs of vertices of $H$ are $(\beta, t)$-connectable for some bounded $\beta>0$ and $t\in \mathbb N$, or the pairs of vertices form exactly two connected components from which we derive that $H$ is even/odd-extremal, and indeed, we can take $t=2$.

 We first include the following useful result. 
 \begin{fact}
 \label{degreeinherit}
Let $1 \leq d' \leq d < k$ and $H$ be a $k$-graph on $n$ vertices. If $\delta_d(H) \geq x\binom{n-d}{k-d}$ for some $0 \leq x \leq 1$, then $\delta_{d'}(H) \geq x\binom{n-d'}{k-d'}$.
 \end{fact}
 \begin{proof}
     Since $\delta_d(H) \geq x\binom{n-d}{k-d}$, we have $
\delta_{d'}(H) \geq \binom{n-d'}{d-d'} x \binom{n-d}{k-d}/\binom{k-d'}{d-d'} = x\binom{n-d'}{k-d'}.$
 \end{proof}

Following the statement of the connecting lemma, we shall use the hierarchy $1/n\ll \a\ll \beta \ll \e \ll \gamma$ in this subsection.
We start with the following result, saying that every pair of vertices is $1$-connectable to more than $1/3$ of other pairs, and among every three pairs of vertices, two of them are $1$-connectable.

\begin{proposition}
\label{prop:1-conn}
     Suppose that $1/n\ll\beta\ll\gamma$.
     Let $H$ be an $n$-vertex $4$-graph with $\delta_2(H)\geq (1/3+\gamma) \binom {n-2}2$.
     Then for every pair of vertices $e$ of $H$, $e$ is $(\beta, 1)$-connectable to at least $(1/3+2\gamma/3)\binom n2$ pairs of vertices of $H$.
     Moreover, for every three pairs of vertices $e, f, g$ of $H$, there are two of them that are $(\beta, 1)$-connectable to each other.
\end{proposition}

\begin{proof}
%The former conclusion follows directly from the  minimum codegree condition; the latter one follows from the former one and the fact 
Note that two pairs $e$ and $g$ are $(\beta, 1)$-connectable if $|N_{H}(e)\cap N_{H}(g)|\ge \beta n^{2}$.
Then the second assertion follows from this and the minimum degree condition.
Let $\T\subseteq \binom V2$ be the family of pairs that are $(\beta, 1)$-connectable to $e$ and let $M$ be the number of choices of $(f,g)$ with $e\cup f, f\cup g\in E(H)$.
Then we have $\deg(e)\cdot (1/3+\gamma)\binom {n-2}2 \le M\le |\T|\deg(e) + \binom n2 \beta n^{2}$, which yields that 
$|\T| \ge (1/3+\gamma)\binom {n-2}2 - \beta n^{4}/\deg(e) \ge (1/3+2\gamma/3) \binom n2$, as desired.
\end{proof}

The next proposition states that given two pairs of vertices that are not $4$-connectable, we can construct two disjoint families of pairs whose union is almost $\binom{V(H)}{2}$, both of size greater than $(1/3+\gamma/2)\binom n2$, such that there are few edges $E$ that can be expressed as the union of two pairs, one from each family.

\begin{proposition}
\label{prop:pair-partition}
     Suppose that $1/n\ll\beta\ll\gamma$.
     Let $H$ be an $n$-vertex $4$-graph with $\delta_3(H)\geq{n}/{3}+\gamma n$.
     Then either every two pairs of vertices of $H$ are $(\beta, 4)$-connectable to each other, or there exist disjoint families of pairs $\mathcal Q_{1}, \mathcal Q_{2} \subseteq \binom{V(H)}2$ such that for $i=1,2$, we have $|\mathcal Q_{i}|\ge (1/3+\gamma/2)\binom n2$, $|\Q_{1}\cup \Q_{2}|\ge \binom n2 - 2\beta^{1/3} n^{2}$, and $H$ contains at most $2\beta^{1/3} n^{4}$ edges each of which consists of a pair in $\Q_{1}$ and a pair in $\Q_{2}$.
     Moreover, $\Q_{1}$ and $\Q_{2}$ can be constructed in time $O(n^{12})$.
\end{proposition}

\begin{proof}
%This can be proved by similar argument for the partition lemmas: if the first assertion fails,  take two pairs $Q_{1}$ and $Q_{2}$ which are not 2-connectable.
By Fact~\ref{degreeinherit}, we have $\delta_2(H)\geq (1/3+\gamma) \binom {n-2}2$ and thus we can apply Proposition \ref{prop:1-conn} to $H$.
Put $V=V(H)$ and suppose $e, f$ are two pairs of vertices which are not $(\beta, 4)$-connectable.
Let $N_{1}(e)$ and $N_{1}(f)$ be the family of pairs of vertices of $V$ that are $(\beta^{1/3}, 1)$-connectable to $e$ and to $f$, respectively.
As $e$ and $f$ are not $(\beta^{1/3}, 1)$-connectable (otherwise they are $(\beta, 4)$-connectable), by the second assertion of  Proposition \ref{prop:1-conn} with $\beta^{1/3}$ in place of $\beta$, we have $N_{1}(e)\cup N_{1}(f)=\binom V2$.
Note also that we have $|N_{1}(e)\cap N_{1}(f)| < 2\beta^{1/3} n^{2}$ -- otherwise, by concatenation, we obtain at least $2\beta n^{6}$ 6-tuples of vertices that connect $e$ and $f$, and removing the $O(n^{5})$ multi-sets from the 6-tuples gives that $e$ and $f$ are $(\beta, 3)$-connectable, a contradiction.

Let $\Q_{1}=N_{1}(e)\setminus N_{1}(f)$ and $\Q_{2}=N_{1}(f)\setminus N_{1}(e)$.
Therefore, we have $|\Q_{i}|\ge (1/3+\gamma/2)\binom n2$ by Proposition \ref{prop:1-conn} and $|\Q_{1}\cup \Q_{2}|\ge \binom n2 - 2\beta^{1/3} n^{2}$.
Moreover, note that $H$ contains at most $2\beta^{1/3} n^{4}$ edges each of which consists of a pair in $\Q_{1}$ and a pair in $\Q_{2}$.
Indeed, note that each choice of $e_0e'f'f_0$ such that 
\begin{enumerate}
    \item $e'\in N_{1}(e)$ and $f'\in N_{1}(f)$ with $e'\cup f'\in E(H)$,
    \item $e_0\in N_H(e)\cap N_H(e')$ and $f_0\in N_H(f)\cap N_H(f')$, and
    \item the pairs $e, e_0, e', f', f_0, f$ are mutually disjoint
\end{enumerate}
%i) $e'\in N_{1}(e)$ and $f'\in N_{1}(f)$ with $e'\cup f'\in E(H)$ and ii) $e_0\in N_H(e)\cap N_H(e')$ and $f_0\in N_H(f)\cap N_H(f')$ 
gives rise to an $8$-connector for $e$ and $f$.
%among which at most $O(n^{7})$ are multi-sets or intersect $e\cup f$.
Therefore, if there are $2\beta^{1/3} n^{4}$ choices for the edge $e'\cup f'$, then as there are each $\beta^{1/3}n^2$ choices for $e_0$ and for $f_0$, respectively, we obtain at least $2\b n^8$ 8-tuples satisfying the adjacency required in (1) and (2).
Since among them $O(n^{7})$ choices violate (3),
%Since 
we obtain 
$2\beta n^{8} - O(n^{7}) \ge \beta n^{8}$ $4$-connector for $e$ and $f$, that is, $e$ and $f$ are $(\beta, 4)$-connectable, a contradiction.
%Then by Proposition~\ref{prop:1-conn} all other pairs are 1-connectable to $Q_{1}$ or $Q_{2}$.
%Letting $\cP_{i}\subseteq \binom V2$ be the family of pairs that are 1-connectable to $Q_{i}$ and we have $|\cP_{i}|\ge (1/3+\gamma/2)\binom n2$.
%Notice that if $|\cP_{1}\cap \cP_{2}| \ge \beta n^{2}$ then by concatenating the paths, $Q_{1}$ and $Q_{2}$ are 2-connectable, a contradiction.
%Therefore, this intersection is small.
%By Proposition~\ref{prop:1-conn} and pigeonholing, every pair in $\cP_{1}\cap \cP_{2}$ is either 0-connectable to many pairs in $\cP_{1}\setminus \cP_{2}$ or $\cP_{2}\setminus \cP_{1}$.
%Redistributing the members of $\cP_{1}\cap \cP_{2}$ give the desired partition of $\binom V2$.
%The ``farthest'' possible pairs of members in $\Q_{i}$ are the ones got redistributed to $\Q_{i}$, and as they are both 0-connectable to many members of $\Q_{i}$ and these members are 1-connectable to each other, each $\Q_{i}$, $i=1,2$, is 3-connectable.

For the running time, note that most of the work is to check the connectability between pairs of vertices, which can be done in time $n^4\cdot O(n^8)=O(n^{12})$ (to exhaustively check 4-connectability among all pairs of pairs of vertices).
\end{proof}

If all pairs of vertices of $H$ are $4$-connectable, then we have a robust connecting lemma, that is, we can connect any two pairs of vertices even after a small number of vertices are forbidden from use.
Then we may assume the latter output of Proposition~\ref{prop:pair-partition}.
We now perform additional (structural) analysis in this case.
%
%We use the following auxiliary 2-edge-coloring: 
%We consider a red-blue edge coloring of $\binom{V(H)}2$ defined by the families $\Q_{1}, \Q_{2}$.
It is convenient to treat $\Q_{1}, \Q_{2}$ as graphs, that is, let $R$ and $B$ be graphs on $V(H)$ such that $E(R)=\Q_{1}$ and $E(B)=\Q_{2}$. 
Informally, we can think of them as red and blue edges.
%Let us denote the red graph by $R$ and the blue graph by $B$ both on $V(H)$, respectively.
%Recall that Proposition~\ref{prop:pair-partition} says that $|\mathcal Q_{i}|\ge (1/3+\gamma/2)\binom n2$ for $i=1,2$,  that is, there are at least $(1/3+\gamma/2)\binom n2$ red  (or blue) edges of $\C$.
%Recall also that if two pairs of vertices are 0-connectable if they form an edge of $H$.
%However, we note that even two pairs are connectable, they might still be put into different parts, i.e., one in $\Q_{1}$ and one in $\Q_{2}$.
%But there are few such edges.
%\begin{proposition}
%\label{prop:ideal-edges}
%Suppose $\Q_{1}$ and $\Q_{2}$ are returned by Proposition~\ref{prop:pair-partition}.
%If there exist $\beta n^{4}$ edges $E$ of $H$ such that $E=f\cup g$ with $f\in \Q_{1}$ and $g\in \Q_{2}$, then all pairs of $\binom V2$ are $(\beta^{3}, 7)$-connectable.
%\end{proposition}

%\begin{proof}
%If there exist $\beta n^{4}$ such edges, then we get many 7-connectable sets for every $e\in \Q_{1}$ and $e'\in \Q_{2}$.
%\end{proof}

The key property we have is that few edges of $H$ are the union of one red edge and one blue edge.
This gives the following compatibility rule: 
\begin{enumerate}[label=$(*)$]
\item for all but at most $4\beta^{1/3} n^{4}$ edges of $H$, it holds that the six pairs of vertices form three monochromatic matchings, that is, each pair of matching edges are either both in $R$ or both in $B$. \label{item:1}
\end{enumerate}
Indeed, the edges of $H$ violating the above rule either contain a pair not in $\Q_{1}\cup \Q_{2}$ or can be written as a union of a red edge and a blue edge, so there are at most $2\beta^{1/3} n^{4} + 2\beta^{1/3}n^{2}\cdot n^{2}\le 4\beta^{1/3} n^{4}$ such edges.

Now let $E_1\subseteq E(H)$ be the family of edges that satisfy~\ref{item:1} and thus $|E(H)\setminus E_{1}|\le 4\beta^{1/3} n^{4}$.
Let $V'\subseteq V$ be the set of vertices which are incident to at most $\beta^{1/6} n^{3}$ edges in $E(H)\setminus E_{1}$ \emph{and} incident to at most $\beta^{1/6} n$ uncoloured pairs of vertices.
 Since $|E(H)\setminus E_{1}|\le 4\beta^{1/3} n^{4}$ and there are at most $2\beta^{1/3} n^{2}$ uncoloured pairs, by counting, we have 
\[
|V\setminus V'|\le 4\cdot 4\beta^{1/3} n^{4}/ (\beta^{1/6} n^{3}) + 2\cdot 2\beta^{1/3} n^{2}/(\beta^{1/6} n) = 20 \beta^{1/6} n.
\]
We now study the structure of the graphs $R$ and $B$.

\begin{claim}
\label{clm:V'}
For $v\in V'$, $v$ is incident to at least $(1/3+\gamma/2)n$ red edges and blue edges, respectively.
\end{claim}

\begin{proof}
Fix a vertex $v\in V'$ and a blue edge $uw\in \binom V2$. 
By the minimum codegree condition, we have $\deg_{H}(uvw)\ge (1/3+\gamma)n$.
If such an edge $uvwx$ is also in $E_{1}$, then we have $vx$ is blue.
Since $v\in V'$, $v$ is incident to at most $ \beta^{1/6} n^{3}$ edges of $H$ not in $E_{1}$.
%Write $E_{b}$ for the blue edges under coloring $\C$ and l
Let $m_{v}$ be the number of (ordered) triples $(u,w,x)$ such that $uw\in E(B)$ and $vuwx\in E(H)$ and let $t$ be the number of blue edges not incident to $v$.
Then we have $t\ge (1/3+\gamma/2)\binom n2 - n$.
Therefore, if $v$ is incident to less than $(1/3+\gamma/2)n$ blue edges, then as $\beta \ll \gamma$ we have
\[
(1/3+\gamma)n t\le m_v < (1/3+\gamma/2)n\cdot t+ 3! \beta^{1/6} n^{3} < (1/3+\gamma)n t,
\]
a contradiction. The same proof shows the conclusion for red edges.
\end{proof}

Our next claim further reveals the asymptotic structure of $H$.
We call a triple of vertices $uvw$ \emph{good} if at most $\gamma n/2$ vertices $x$ satisfy $uvwx\in E(H)\setminus E_{1}$.
%$|N_{H}(uvw)\setminus  E_{1}|\le \gamma n/2$.
Then by $|E(H)\setminus E_{1}|\le 4\beta^{1/3} n^{4}$, we get that all but at most $(32\beta^{1/3}/\gamma)n^{3}$ triples of $V(H)$ are good.
We call $uv\in \binom V2$ \emph{good} (otherwise \emph{bad}) if there are less than $\e n/2$ choices of $w$ such that $uvw$ is not a good triple.
Similar counting shows that there are at most $3(32\beta^{1/3}/\gamma)n^{3} / (\e n/2-1) \le \e^{3}n^{2}$ bad pairs.
The next two claims are the main new ingredients of our proof of the connection step, with Property~\ref{item:1} as the main driving force.

\begin{claim}
\label{clm:BR}
If $u,v\in V'$ and $uv$ is good, then it satisfies one of the following properties:
\begin{enumerate}[label=$(N\arabic*)$]
\item $|N_{B}(u)\triangle N_{B}(v)|\le 2\e n$ and $|N_{R}(u)\triangle N_{R}(v)|\le 2\e n$, or \label{item:N1}
\item  $|N_{B}(u)\triangle N_{R}(v)|\le 2\e n$ and  $|N_{R}(u)\triangle N_{B}(v)|\le 2\e n$. \label{item:N2}
\end{enumerate}
\end{claim}

\begin{proof}
%Without loss of generality, assume $uv\in E(B)$.
Let $X:=N_{B}(u)\cap N_{B}(v)$, $Y:=N_{B}(u)\cap N_{R}(v)$, $Z:=N_{R}(u)\cap N_{B}(v)$, $W:=N_{R}(u)\cap N_{R}(v)$.
We first claim the following properties.
\begin{itemize}
\item If $|X|\ge \e n/2$, then $|X|> n/3$;
\item if $|W| \ge \e n/2$, then $|W|> n/3$;
\item if $|Y|\ge \e n/2$, then $|Z|> n/3$;
\item if $|Z|\ge \e n/2$, then $|Y|> n/3$.
\end{itemize}
Indeed, first suppose $|X|\ge \e n/2$.
As $uv$ is good and $|X|\ge \e n/2$, we can choose $x\in X$ such that $uvx$ is good, and thus there are more than $(1/3+\gamma/2)n$ choices of $a$ such that $xuva\in E_{1}$.
Then by $ux, vx\in E(B)$ and~\ref{item:1}, we have $va, ua\in E(B)$, and thus $a\in X$.
Therefore we obtain $|X|> n/3$.
The same proof shows also the second property.
Now assume $|Y|\ge \e n/2$. 
Similarly, we can choose $y\in Y$ such that $uvy$ is good, and thus there are more than $(1/3+\gamma/2)n$ choices of $a$ such that $yuva\in E_{1}$.
By $uy\in E(B)$, $vy\in E(R)$ and~\ref{item:1}, we have $va\in E(B)$, $ua\in E(R)$, yielding that $a\in Z$. Thus we get $|Z|> n/3$. 
The same proof also gives the last property.

Clearly, the last two properties imply that either both $|Y|, |Z| > n/3$ hold, or both $|Y|, |Z| \le \e n/2$.
Now to see the claim, if $|X| > \e n/2$, then $|X| > n/3$, and thus we have $|Y|, |Z| < \e n/2$.
This yields that $|N_{B}(u)\triangle N_{B}(v)|\le |Y|+|Z|+2\beta^{1/6}n+2\le 2\e n$ and $|N_{R}(u)\triangle N_{R}(v)|\le 2\e n$, giving~\ref{item:N1}.
Otherwise $|X| \le \e n/2$.
As $v\in V'$, $v$ is incident to at most $\beta^{1/6} n$ uncoloured pairs of vertices.
Together with Claim~\ref{clm:V'} and $u\in V'$, we get $|X|+|Y|\ge d_{B}(u) - \beta^{1/6} n > n/3$.
This implies $|Y|, |Z| > n/3$, which further yields that $|W|\le \e n/2$.
So we derive $|N_{B}(u)\triangle N_{R}(v)|\le |X|+|W|+2\beta^{1/6}n+2\le 2\e n$ and similarly $|N_{R}(u)\triangle N_{B}(v)|\le 2\e n$.
\end{proof}

The two possible outputs of the above claim say that the red/blue neighbourhoods of $u$ and $v$ are either almost the same, or almost the complement of each other.
This is enough for us to conclude the (asymptotic) structure of $H$.

\begin{claim}
\label{clm:extremal}
Suppose $\Q_{1}$ and $\Q_{2}$ are returned by Proposition~\ref{prop:pair-partition}.
%If $H$ contains few than $\beta n^{4}$ edges $E$ such that $E=f\cup g$ with $f\in \Q_{1}$ and $g\in \Q_{2}$, t
Then $H$ is $(65\e,\gamma/3)$-even-extremal and one can find a $(65\e,\gamma/3)$-even-extremal bipartition in time $O(n^{4})$.
\end{claim}

\begin{proof}
Recall that there are at most $\e^{3}n^{2}$ bad pairs in $\binom V2$.
Take a vertex $v \in V'$ which is in at most $\frac{2\e^{3}n^{2}}{|V'|}\le \e^{2} n$ bad pairs, and let $W$ denote the union of $V\setminus V'$, the set of the bad neighbours of $v$, the set of $u$ such that $uv$ is uncoloured, and the vertex $v$ itself.
Then $|W|\le 20\b^{1/6}n+\e^{2}n+ \beta^{1/6}n+1\le 2\e^{2} n$.
%Let $V_{1}$ consist of $v\notin W$ and all vertices $u$ satisfying~\ref{item:N1}, and $V_{2}$ consist of all vertices $u\notin W$ satisfying~\ref{item:N2}.
For each
$u\notin W$, assign $u$ to $V_1$ if~\ref{item:N1} holds and to $V_2$ if~\ref{item:N2} holds.
Then $\{V_{1}, V_{2}\cup W\}$ is a bipartition of $V$.

Note that defining $W$ can be done in time $O(n^{4})$ and checking \ref{item:N1} and \ref{item:N2} for every vertex $u$ takes time $O(n)$ (so $O(n^2)$ in total).
Therefore, the overall running time is $O(n^4)$.
It suffices to show that this bipartition sees at most $65\e n^{4}$ odd edges and $|V_{1}|, |V_{2}|\ge (1/3+\gamma/3)n$.
%By Claim~\ref{clm:BR}, if $u\in V_{2}\setminus W$, then $|N_{B}(u)\triangle N_{R}(v)|\le 2\e n$.

%Now, for every two vertices $v_{1}, v_{2}\in V_{1}$, we have
%\[
%|N_{B}(v_{1})\triangle N_{B}(v_{2})|\le |N_{B}(v_{1})\triangle N_{B}(v)| + |N_{B}(v_{2})\triangle N_{B}(v)|\le 4\e n.
%\]
%Similarly, for every two vertices $u_{1}, u_{2}\in V_{2}$, we have
%\[
%|N_{B}(u_{1})\triangle N_{B}(u_{2})|\le |N_{B}(u_{1})\triangle N_{R}(v)| + |N_{B}(u_{2})\triangle N_{R}(v)|\le 4\e n.
%\]

For $i=1,2$, let $V_{i,1}=V_{i}\cap N_{B}(v)$ and $V_{i,2}=V_{i}\cap N_{R}(v)$.
By the definition of $W$, we have $V_{i}=V_{i,1}\cup V_{i,2}$, and $N_{B}(v)\subseteq V_{1,1}\cup V_{2,1}\cup W$ and $N_{R}(v)\subseteq V_{1,2}\cup V_{2,2}\cup W$.
% and $V_{i,3}=V_{i}\setminus (V_{i,1}\cup V_{i,2})$.
%Note that as $v\notin V'$, $|V_{i,3}|\le \beta^{1/6}n$.
We first show that
\begin{equation}
\label{eq:Vij}
\min\{|V_{i,1}|, |V_{i,2}|\}\le 4\e n \text{ for } i=1,2.
\end{equation}
Indeed, suppose $|V_{i,1}|, |V_{i,2}|> 4\e n$ and first assume $i=1$.
Since $v$ sends red edges to $V_{i,2}$, all vertices of $V_{i,1}$ should send mostly red edges to $V_{i,2}$, that is, at least $|V_{i,1}|(|V_{i,2}|-2\e n) > |V_{i,1}||V_{i,2}|/2$ edges are red; on the other hand, since $v$ sends blue edges to $V_{i,1}$, all vertices of $V_{i,2}$ should send mostly blue edges to $V_{i,1}$, that is, at least $|V_{i,2}|(|V_{i,1}|-2\e n)> |V_{i,1}||V_{i,2}|/2$ edges are blue.
This gives a contradiction.
%More precisely, 
For the $i=2$ case, switching red and blue in the above arguments show the result.

Moreover,~\eqref{eq:Vij} leaves two possible cases: i) both $V_{1,2}$ and $V_{2,1}$ are small or ii) both $V_{1,1}$ and $V_{2,2}$ are small. This is because $|V_{1,1}|+|V_{2,1}| \ge |N_B(v)| - |W| \ge (1/3+\gamma/2)n - 2\e^{2}n > n/3$ by $v\in V'$ and Claim~\ref{clm:V'}, yielding that one of $V_{1,1}$ and $V_{2,1}$ must be large. Similarly one of $V_{1,2}$ and $V_{2,2}$ must be large.

%As $v\in V'$, by Claim~\ref{clm:V'}, we have $|N_{B}(v)|, |N_{R}(v)|\ge (1/3+\gamma/2)n$.
We now show that $|V_{1}|, |V_{2}|\ge (1/3+\gamma/3)n$.
Indeed, by~\eqref{eq:Vij}, we have the following two possible cases as explained in the previous paragraph.
\begin{itemize}
\item ($V_{1,2}$ and $V_{2,1}$ are small) In this case using $N_{B}(v)\subseteq V_{1,1}\cup V_{2,1}\cup W$, we get 
\[
|V_{1}|\ge |V_{1,1}|\ge |N_{B}(v)| - |V_{2,1}| -|W|\ge (1/3+\gamma/2)n - 4\e n - 2\e^{2} n \ge (1/3+\gamma/3)n.
\]
Similarly $|V_{2}|\ge |V_{2,2}|\ge (1/3+\gamma/3)n$. 
\item ($V_{1,1}$ and $V_{2,2}$ are small) Again similar arguments give $|V_{1}|\ge |V_{1,2}|\ge (1/3+\gamma/3)n$ and $|V_{2}|\ge |V_{2,1}|\ge (1/3+\gamma/3)n$.
\end{itemize}
Therefore the claimed inequalities hold in both cases.

%First, in either case, we have $|V_{1}|, |V_{2}|\ge (1/3+\gamma/2)n - 4\e n \ge (1/3+\gamma/3)n$.
It remains to bound the number of odd edges.
Now note that for $u\in V_{1}$ and $i=1,2$, by $V_{i,1}\subseteq N_{B}(v)$ and $|N_{B}(u)\triangle N_{B}(v)|\le 4\e n$, we obtain
\[
|V_{i,1}\setminus N_{B}(u)|\le |N_{B}(u)\triangle N_{B}(v)| \le 4\e n,
\]
that is, every vertex $u\in V_{1}$ sends at most $4\e n$ red edges to $V_{i,1}$.
%as $V_{1,1}\subseteq N_{B}(v)$ and $|N_{B}(u)\triangle N_{B}(v)|\le 4\e n$.
%Similarly, using $V_{1,2}\cup V_{2,2}= N_{R}(v)$ and $V_{2,1}\subseteq N_{B}(v)$, f
For $u\in V_{1}$, $w\in V_{2}$ and $i=1,2$,  similar arguments using~\ref{item:N1} and~\ref{item:N2} show that
\[
|V_{i,1}\setminus N_{B}(u)| \le 4\e n,\quad |V_{i,2}\setminus N_{R}(u)| \le 4\e n, \quad |V_{i,1}\setminus N_{R}(w)| \le 4\e n, \text{ and } |V_{i,2}\setminus N_{B}(w)| \le 4\e n.
\]
%The other three inequalities can be obtained similarly.
Now for the first case ($V_{1,2}$ and $V_{2,1}$ are small), we infer that almost all members of $\binom{V_{1}}2\cup \binom{V_{2}}2$ are \emph{blue} and almost all members of $V_{1}V_{2}$ are \emph{red}:
\[
\left|\left(\binom{V_{1}}2\cup \binom{V_{2}}2\right)\cap E(R)\right|\le 4\e n\cdot |V_{1,1}| + 4\e n \cdot |V_{2,2}|+n(|V_{1,2}|+|V_{2,1}|)\le 16\e n^{2}
\]
and 
\[
|(V_{1} V_{2})\cap E(B)|\le |V_{1}|\cdot 4\e n + n |V_{2,1}| \le 8\e n^{2}.
\]
Similarly, for the second case ($V_{1,1}$ and $V_{2,2}$ are small), almost all members of $\binom{V_{1}}2\cup \binom{V_{2}}2$ are red and almost all members of $V_{1} V_{2}$ are blue:
\[
\left|\left(\binom{V_{1}}2\cup \binom{V_{2}}2\right)\cap E(B)\right|\le 4\e n\cdot |V_{1,2}| + 4\e n \cdot |V_{2,1}|+n(|V_{1,1}|+|V_{2,2}|)\le 16\e n^{2}
\]
and 
\[
|(V_{1} V_{2})\cap E(R)|\le |V_{1}|\cdot 4\e n+n|V_{2,2}|\le 8\e n^{2}.
\]

Now we are ready to conclude the proof. 
First we consider the case that $V_{1,2}$ and $V_{2,1}$ are small.
Observe that an odd edge of $H$ must satisfy one of the following: i) it lies in $ V_{1,1}^3 V_{2,2}$ or $V_{2,2}^3 V_{1,1}$ satisfying the property in~\ref{item:1}, ii) it intersects $V_{1,2}\cup V_{2,1}\cup W$, or iii) it does not satisfy the property in~\ref{item:1}.
Indeed, if an odd edge does not intersect $V_{1,2}\cup V_{2,1}\cup W$, then it lies in $ V_{1,1}^3 V_{2,2}$ or $V_{2,2}^3 V_{1,1}$.
%and with three blue edges forming a triangle and three red edges forming a star.
%Since this violates~\ref{item:1}, $F\notin E_{1}$ and thus not an edge of $H$. 
To estimate the number of $ V_{1,1}^3 V_{2,2}$ edges satisfying i), note that all such edges must contain a monochromatic matching, with one edge in $V_{1,1}$ and the other in $V_{1,1}V_{2,2}$.
This can be upper bounded by the number of $4$-sets $xyzw\in V_{1,1}^3 V_{2,2}$ such that either $xyz$ spans at least one red edge or at least one of $xw, yw, zw$ is blue, which in turn is upper bounded by $16\e n^{4}+ 12\e n^{4} = 28\e n^{4}$; similarly the number of $4$-sets $xyzw\in V_{2,2}^3 V_{1,1}$ such that either $xyz$ spans at least one red edge or at least one of $xw, yw, zw$ is blue is also upper bounded by $28\e n^{4}$.
Therefore, the number of odd edges of $H$ is at most $2\cdot 28\e n^{4} + (8\e n+2\e^{2} n) n^{3} + 4\beta^{1/3} n^{4}\le 65\e n^{4}$.

The second case ($V_{1,1}$ and $V_{2,2}$ are small) can be treated similarly.
Observe that an odd edge of $H$ must satisfy one of the following: i) it is in $ V_{1,2}^3 V_{2,1}$ or $V_{2,1}^3 V_{1,2}$ satisfying the property in~\ref{item:1}, ii) it intersects $V_{1,1}\cup V_{2,2}\cup W$, or iii) it does not satisfy the property in~\ref{item:1}.
%Indeed, if none of the above happens, then such a $4$-set $F$ is contained in $V_{1,2}\cup V_{2,1}$ and with three red edges forming a triangle and three blue edges forming a star.
%Since this violates~\ref{item:1}, $F\notin E_{1}$ and thus not an edge of $H$. 
Indeed, if an odd edge does not intersect $V_{1,1}\cup V_{2,2}\cup W$, then it lies in $ V_{1,2}^3 V_{2,1}$ or $V_{2,1}^3 V_{1,2}$.
To estimate the number of $ V_{1,2}^3 V_{2,1}$ edges satisfying i), 
note that all such edges must contain a monochromatic matching, with one edge in $V_{1,2}$ and the other in $V_{1,2}V_{2,1}$.
This can be upper bounded by the number of $4$-sets $xyzw\in V_{1,2}^3 V_{2,1}$ such that either $xyz$ spans at least one blue edge or at least one of $xw, yw, zw$ is red, which in turn is upper bounded by $16\e n^{4}+ 12\e n^{4} = 28\e n^{4}$; similarly the number of $4$-sets $xyzw\in V_{2,1}^3 V_{1,2}$ such that either $xyz$ spans at least one blue edge or at least one of $xw, yw, zw$ is red is also upper bounded by $28\e n^{4}$.
Therefore, the number of odd edges of $H$ is at most $2\cdot 28\e n^{4} + (8\e n+2\e^{2} n) n^{3} + 4\beta^{1/3} n^{4}\le 65\e n^{4}$.
%Without loss of generality, assume 
%Then we output the partition 
%Computing the sizes of the sets $V_{i,1}, V_{i,2}$ takes $O(n)$ time.
\end{proof}

%Finally we show that if all members of $\binom V2$ are connectable, then we have the connecting lemma

%\subsection{A connecting lemma}
Now we are ready to derive our connecting lemma. For convenience we restate Lemma \ref{lem:conn}.

\medskip
\noindent%\begin{lemma}
\textbf{Lemma \ref{lem:conn}.}
%\label{lem:conn}
\emph{Suppose $1/n\ll \a\ll \beta \ll \e \ll \gamma$.
Let $H$ be a $4$-graph with $\delta_{3}(H)\ge (1/3+\gamma)n$.
% which is not $65\e$-even-extremal.
Then one of the following holds.
\begin{itemize}
\item Every two pairs of vertices of $H$ are $(\beta, 4)$-connectable to each other, or
\item we can find a $(65\e,\gamma/3)$-even-extremal bipartition in time $O(n^{12})$.
\end{itemize}
Moreover, in the former case, for any given vertex set $W$ of size at most $3\beta^2 n$, there is a vertex set $R\subseteq V(H)\setminus W$ of size $\a n$ such that for every two pairs of vertices there are at least $\alpha^{9} n$ vertex-disjoint connectors in $R$ each of size at most 8.}
%\end{lemma}

\begin{proof}
[Proof of Lemma \ref{lem:conn}]
We start with applying Proposition~\ref{prop:pair-partition} to $H$, which in time $O(n^{12})$ gives that either every two pairs of vertices of $H$ are $(\beta, 4)$-connectable to each other, or there exist disjoint families of pairs $\mathcal Q_{1}, \mathcal Q_{2} \subseteq \binom{V(H)}2$ such that for $i=1,2$, we have $|\mathcal Q_{i}|\ge (1/3+\gamma/2)\binom n2$, $|\Q_{1}\cup \Q_{2}|\ge \binom n2 - 2\beta^{1/3} n^{2}$, and $H$ contains at most $2\beta^{1/3} n^{4}$ edges each of which consists of a pair in $\Q_{1}$ and a pair in $\Q_{2}$.
%First we assume that Proposition~\ref{prop:pair-partition} outputs a partition $\mathcal Q_{1}\cup \mathcal Q_{2} = \binom{V(H)}2$ and $H$ contains less than $\beta n^{4}$ such edges as above.
In the latter case, Claim~\ref{clm:extremal} outputs a $(65\e,\gamma/3)$-even-extremal bipartition in time $O(n^{4})$ and we are done in this case.

For the former case, we have that for every choice of $S, T\in \binom V2$, there are at least $\beta n^{2i}$ $2i$-connectors for some $i=i(S, T)\le 4$.
Then for any set $W$ of $3\b^2 n$ vertices, there are at least $\beta n/3$ vertex-disjoint choices of $2i$-connectors in $V(H)\setminus W$.
Indeed, suppose $\C$ is a maximum collection of disjoint $2i$-connectors for $S, T$.
If $|\C| < \beta n/3$, then together with $W$, their vertices intersect at most $(2i|\C|+|W|)\binom n{2i-1}\le (2i\beta/3 +3\beta^2)n\binom n {2i-1} <i\b \frac {n^{2i}}{(2i-1)!} \le \beta n^{2i}$ $2i$-connectors, so that we can pick another connector disjoint from the members of $\C$, a contradiction.
Now let $R\subseteq V(H)\setminus W$ be a uniformly random subset of vertices of size $\a n$.
By Chernoff's bound and the union bound, we obtain that with positive probability, for every choice of $S, T\in \binom V2$ there are $\a^{2i}\beta n/6 \ge \a^9 n$ $2i$-connectors in $R$ for $i=i(S,T)$.
Fixing such an instance of $R$ finishes the proof.
%Moreover, among each family of such $\a^{2i}\beta n^{2i}/2 \ge \a^{9}\beta n^{2i}$ $2i$-connectors in $R$, one can pick $\a^{9}\beta n$ vertex-disjoint connectors each of size $2i$ greedily.
\end{proof}

\subsection{Absorbing}
For the absorption part, we use the swap-absorb method for Hamilton cycles originating from~\cite{Reiher2019Minimum}.
One important ingredient of this method is the following notion of reachability introduced by Lo and Markstr\"{o}m~\cite{MR3338027}, originally for matching and tiling problems.
% We shall introduce this notion in terms of
Note that the reachability notion we introduce here is a special form of the general one for $k$-uniform $(k/2)$-paths.
 
 Given a constant $\beta > 0$, an integer $i \geq 1$ and a $k$-graph $H$ on $n$ vertices, we say that two vertices $u, v$ are $(\beta, i)$-\emph{reachable} if there are at least $\beta n^{(3k/2)i-1}$ $((3k/2)i - 1)$-sets $T$ such that there exist vertex-disjoint $k/2$-paths $P_1, \ldots, P_i$ of length two with $V(P_1 \cup \cdots \cup P_i) = T \cup \{u\}$, and vertex-disjoint $k/2$-paths $P'_1, \ldots, P'_i$ of length two with $V(P'_1 \cup \cdots \cup P'_i) = T \cup \{v\}$, where $P_j$ and $P'_j$ have the same ends for all $j \in [i]$. 
 Moreover, we call $T$ a \emph{reachable set} for $\{u, v\}$. Given a vertex set $U \subseteq V(H)$, $U$ is said to be $(\beta, i)$-\emph{closed} if every two vertices in $U$ are $(\beta, i)$-reachable in $H$.
 Given any $v \in V(H)$, define $\tilde{N}_{\beta,i}(v, H)$ to be the set of vertices in $V(H)$ that are $(\beta, i)$-reachable to $v$ in $H$.

For an example of a reachable set, for $k=4$, consider vertices $u, v, v_{1}, v_{2}, v_{4}, v_{5}, v_{6}$ which span edges $v_{1}v_{2} u v_{4}$, $v_{1}v_{2} v v_{4}$, $u v_{4}v_{5} v_{6}$ and $v v_{4}v_{5} v_{6}$.
Therefore, $\{v_{1}, v_{2}, v_{4}, v_{5}, v_{6}\}$ is a reachable set for $u$ and $v$ with $P_{1}=v_{1}v_{2} u v_{4}v_{5}v_{6}$ and $P_{1}'=v_{1}v_{2} v v_{4}v_{5}v_{6}$.

 The following proposition controls the number of vertices 1-reachable to a given vertex $v$ using the minimum codegree condition. 
 \begin{proposition}
 \label{prop:tildeNv}
     Suppose that $1/n\ll\beta\ll\xi \ll \zeta \ll 1/k$ and $k\in 2\mathbb N$.
     Let $H$ be a $k$-graph with $\delta_{k-1}(H)\ge \zeta n$.
     Then for every vertex $v\in V(H)$, we have $|\tilde{N}_{\beta,1}(v, H)|\ge (\zeta - \xi) n$.
     Moreover, in every $\lfloor\zeta^{-1}\rfloor+1$ vertices of $H$, two of them are $(\beta, 1)$-reachable in $H$. 
 \end{proposition}
 
 \begin{proof}
Let $1/n\ll\beta\ll\xi \ll \zeta$.
We have $\delta_1(H)\geq {\zeta}\binom{n}{k-1}\ge {\zeta}{k^{-k}}n^{{k-1}}$ by Fact \ref{degreeinherit} and $\delta_{k-1}(H)\ge \zeta n$.
%Fix a vertex $v$ and any other vertex $u$ such that 
We first show that for vertices $u, v$, if $|N_H(u) \cap N_H(v)| \geq \xi^{2} n^{k-1}$, then they are $(\beta, 1)$-reachable in $H$.  
Let $P$ be the $k$-uniform $k/2$-path of length two. 
Then the link $(k-1)$-graph $P^{*}$ of a vertex of degree two in $P$ is $(k-1)$-partite. 
By the supersaturation result (see~\cite{MR183654}) and using $\beta \ll \xi$, we can find $(3k/2-1)! \beta n^{3k/2-1}$ copies of $P^{*}$ in $N_H(v) \cap N_H(u)$. 
Given any such copy of $P^{*}$ whose vertex set is denoted by $T$, we get that both $T \cup \{v\}$ and $T \cup \{u\}$ form a copy of $P$. 
Overall there are at least $\beta n^{3k/2-1}$ choices of $(3k/2-1)$-sets for $T$. 
So $v$ and $u$ are $(\beta, 1)$-reachable in $H$. 

The claim above together with $\delta_1(H)\geq {\zeta}\binom{n}{k-1}$ directly implies the latter assertion.
For the former one, by double counting, we have
\[
\zeta n\cdot |N_H(v)|\le \sum_{S\in N_H(v)} \deg(S) < |\tilde{N}_{\beta,1}(v, H)|
\cdot |N_H(v)| + n\cdot \xi^{2} n^{k-1},
\]
which together with $|N_H(v)|=\deg(v)\ge {\zeta}{k^{-k}}n^{{k-1}}$, yields $|\tilde{N}_{\beta,1}(v, H)| > \zeta n - (k^{k}\xi^{2}/\zeta )n\ge (\zeta-\xi)n$, as $\xi \ll \zeta \ll 1/k$.
 \end{proof}

%\red{JH. Need to copy back the proposition that shows among every 3 vertices there are two vertices that are 1-reachable.}

The following lemma from \cite{HT2020} provides a vertex partition $\cP=\{V_1,\ldots,V_r\}$ of a $k$-graph $H$ such that for every $i\in[r]$, $V_i$ is $(F,\beta,2^{c-1})$-closed in $H$, where $F$ is a fixed $k$-graph.
We state it under our setting, where we take $F$ to be the ($k$-uniform) $k/2$-path of length 2.

\begin{lemma}\cite{HT2020}
\label{lem:P}
Given $\delta'>0$, integer $c\ge 2$, $k\in 2\mathbb N$ and $0<\eta \ll 1/c, \delta', 1/k$, there exists a constant $\beta>0$ such that the following holds for all sufficiently large $n$. 
%Let $F$ be an $m$-vertex $k$-graph.
Assume $H$ is an $n$-vertex $k$-graph  and $S\subseteq V(H)$ is such that $|\tilde{N}_{\eta, 1}(v,H)\cap S| \ge \delta' n$ for any $v\in S$. 
Further, suppose every set of $c+1$ vertices in $S$ contains two vertices that are $(\eta, 1)$-reachable in $H$. 
Then in time $O(n^{2^{c-1} (3k/2)+1})$ we can find a partition $\cP$ of $S$ into $V_1,\dots, V_r$ with $r\le \min\{c, 1/\delta' \}$ such that for any $i\in [r]$, $|V_i|\ge (\delta' - \eta) n$ and $V_i$ is $(\beta, 2^{c-1})$-closed in $H$.
\end{lemma}

Proposition~\ref{prop:tildeNv} with $\zeta=1/3+\gamma$ allows us to apply Lemma~\ref{lem:P} to $H$ with $S=V(H)$, $\delta'=1/3+\gamma/2$, and $c=2$ and thus the resulting vertex partition of $V(H)$ has at most two parts.
%Our goal will be to show the existence of an absorbing path if $V(H)$ is closed, and otherwise output that $H$ is even (or odd) extremal.

We next study the distribution of edges over the partition $\cP$ via the robust vector theory from~\cite{KM1} and~\cite{Han14_Poly}.
%\red{JH. Using transferrals seems a wrong idea.}
%We use the following lemma from~\cite{Han16} that provides a sufficient condition on when we can merge the parts of the partition.
Suppose $\mathcal{P} = \{V_1, \dots, V_d\}$ is a vertex partition of a $k$-graph $H$ and $\mu>0$.
The \emph{index vector} of a set $S\subseteq V(H)$ is $\textbf{i}_{\cP}(S) = (|S\cap V_{1}|, \dots, |S\cap V_{d}|)$.
Define $I_{\mathcal{P}}^\mu(H)$ as the set of all vectors $\textbf{i}\in \mathbb{Z}^d$ for which $H$ contains at least $\mu n^k$ edges with index vector $\textbf{i}$. 
We now define the absorbers, which shows why the reachability is useful.
\begin{definition}[End-absorbers, $S$-absorbers]
\label{def:absorber}
Suppose $k\in 2\mathbb N$ and integer $t\ge 1$.
An end-absorber $Q$ for our proof is a $k$-uniform $k/2$-cycle of length $4$, that is, it
has vertex set $\{w_{1},\dots, w_{2k}\}$ written as disjoint $k/2$-sets $S_{1}, S_{2}, S_{3}, S_{4}$, and edges $S_{1}\cup S_{2}, S_{2}\cup S_{3}, S_{3}\cup S_{4}, S_{4}\cup S_{1}$.
Let its interior path $P_{Q}$ be $S_{3}\cup S_{4}$ and its exterior path be $S_{3}S_{2}S_{1}S_{4}$.
For a $k$-set $S=\{v_{1}, \dots, v_{k}\}$, an
$S$-\emph{absorber} is the union of a family of disjoint sets $T_{i}$ each of size $(3k/2)t-1$ $i\in [k]$ and an end-absorber $Q$ (also disjoint from $S$) such that for $i\in [k]$, $T_{i}$ is a reachable set for $v_{i}$ and $w_{i}$.
\end{definition}

We will use $S$-absorbers in the following manner.
For simplicity, let $k=4$.
By reachability, for $i\in [4]$, there are vertex-disjoint paths $P_{i,j}$, $j\in [t]$ and vertex-disjoint paths $P_{i,j}'$, $j\in [t]$ such that $T_{i}\cup\{w_{i}\}=\bigcup_{i\in [t]}V(P_{i,j})$ and $T_{i}\cup\{v_{i}\}=\bigcup_{i\in [t]}V(P_{i,j}')$, where each $P_{i,j}$ and $P_{i,j}'$ have the same ends.
Then we include the paths $P_{i,j}$, $i\in [4]$, $j\in [t]$ and the interior path of $Q$ as subpaths of our absorbing path (e.g., by connecting them into a single path by Lemma~\ref{lem:conn}).
If we need to absorb the set $S$, then we can take $P_{i,j}'$ instead of $P_{i,j}$ for all $i\in [4], j\in [t]$, and take the exterior path of $Q$, as subpaths of the absorbing path. Since each pair of the paths $P_{i,j}$ and $P_{i,j}'$ has the same ends, this transformation is valid and results in a path with $v_{1}, \dots, v_{4}$ absorbed.

We now show that every $k$-set $S$ of vertices with $\bfi_{\cP}(S)\in I_{\cP}^{\mu}(H)$ has many $S$-absorbers in $H$.
%\red{JH. Need to edit.}
\begin{proposition}
\label{prop:abs}
Suppose that $1/n\ll \beta\ll \mu, 1/k, 1/k', 1/t$ and $k\in 2\mathbb N$.
Let $H$ be a $k$-graph 
%with $\delta_{k-1}(H)\geq \zeta n$
 and let 
% such that $V(H)$ is $(\beta',t)$-closed. 
$\mathcal{P}=\{V_1,\ldots,V_r\}$ be a partition such that $r\le k'$ and for any $i\in[r]$, $|V_i|\ge \mu n$ and $V_i$ is $(\beta,t)$-closed in $H$.
Then every $k$-set $S$ with $\bfi_{\cP}(S)\in I_{\cP}^{\mu}(H)$ has at least ${\beta^{k+2}}n^b$ distinct $S$-absorbers, where $b=(3k^{2}/2)t+k$.
\end{proposition}

\begin{proof}
Write $S=\{v_1,v_2,\dots,v_k\}$, where $v_{i}\in V_{j_{i}}$ and $\bfi_{\cP}(S)\in I_{\cP}^{\mu}(H)$, and let $F$ be the complete $k$-partite
$k$-graph with two vertices in each class. 
By $\bfi_{\cP}(S)\in I_{\cP}^{\mu}(H)$, $H$ contains  $\mu n^k$ edges with index vector $\bfi_{\cP}(S)$. 
Hence, as $\beta \ll \mu$, by supersaturation, $H$ contains
$\beta n^{2k}$ labelled copies of $F$ consisting of these edges. 
In particular, each such copy has index vector $2\bfi_{\cP}(S)$.
All but
$O(n^{2k-1})$ of them are disjoint from $S$.

Fix $F'$ as such a copy and write its $k$ classes as $\{w_i,x_i\}\subseteq V_{j_{i}}$, $i\in[k]$.
For each $i\in[k]$, the vertices $v_i$ and $w_i$ are
$(\beta,t)$-reachable. We may therefore choose successively pairwise
disjoint $((3k/2)t-1)$-sets $T_i$, avoiding $S$ and the chosen copy of $F$, so
that $T_i$ is a reachable set for $\{v_i,w_i\}$. 
At each step at most $O(n^{(3k/2)t-2})$ reachable sets meet the previously used vertices, so there
are at least $0.9\beta n^{(3k/2)t-1}$ choices for $T_i$.

For each $i$, let $\mathcal R_i^w$ and $\mathcal R_i^v$ be the witnessing
families of $t$ disjoint length-two $k/2$-paths covering $T_i\cup\{w_i\}$ and
$T_i\cup\{v_i\}$, respectively, with corresponding paths having the same
ends. 
%The vertex sequences
%\[
%  Q=x_3x_4x_1x_2
%  \quad\text{and}\quad
%  Q'=x_3x_4w_1w_2w_3w_4x_1x_2
%\]
%form, respectively, a one-edge and a three-edge $2$-path in $F$, and they
%have the same ordered ends. 
Denote $V(F')=S_{1}\cup S_{2}\cup S_{3}\cup S_{4}$ where each $|S_{i}|=k/2$ such that $S_{1}S_{2}S_{3}S_{4}S_{1}$ forms a $k/2$-cycle of length 4.
Let $Q=S_{3}\cup S_{4}$ and $Q'=S_{3}S_{2}S_{1}S_{4}$ be $k/2$-paths.
Thus, for
$A=V(F')\cup T_1\cup T_2\cup \cdots \cup T_k$,
the families
\[
  \mathcal P=\{Q\}\cup\bigcup_{i=1}^k\mathcal R_i^w
  \quad\text{and}\quad
  \mathcal P'=\{Q'\}\cup\bigcup_{i=1}^k\mathcal R_i^v
\]
witness that $A$ is an $S$-absorber. Moreover,
$|A|=2k+k((3k/2)t-1)=b$.

The choices above give $(\beta n^{2k} - O(n^{2k-1}))(0.9\beta n^{(3k/2)t-1})^k\ge \beta^{k+1}n^{b}/2$ labelled
constructions. Each $b$-set is produced by at most $b!$ such
constructions, so, as $t$ is fixed and $\beta$ is sufficiently small,
there are at least $\beta^{k+2} n^b$ distinct $S$-absorbers, as desired.
\end{proof}

Therefore, the final task of absorption is to partition the vertices to be absorbed into $k$-sets that whose index vector are in $I_\cP^\mu(H)$, which is trivial when $\cP=\{V(H)\}$.
Readers who are familiar with the lattice-based absorption method may wonder if we can use the theory of transferrals to merge the (two) parts when our partition is non-trivial.
Unfortunately, the notion of reachability in swap-absorb only allows direct concatenation, but not concatenation via transferrals, which means that we cannot use the full theory of lattice-based absorption.
%which, after partitioning, seeks for the existence of transferrals between two parts and then merge them to reduce the dimension of the lattice.
%On the other hand, fortunately, our lattice has dimension only 2, and thus we can close the 
Fortunately, as our lattice has dimension only 2, we can ``manually'' solve the partition problem as follows.

\begin{proposition}
\label{prop:solution}
Suppose $k\in 2\mathbb N$, $k\ge 4$ and $1/n\ll \mu, \gamma, 1/k$.
     Let $H$ be an $n$-vertex $k$-graph with $\delta_{k-1}(H)\geq {n}/{3}+\gamma n$.
     Let also $\cP=\{V_{1}, V_{2}\}$ be a partition of $V(H)$ such that for $i=1,2$, $|V_i|\ge (1/3 + \gamma /3) n$ which is not $(k\mu,\gamma/3)$-extremal.
     %or there exist $\textbf v, \textbf w\in I_{\mathcal{P}}^\mu(H)$ such that $\textbf v - \textbf w = (1,-1)$.
    Then for every vector $(w_{1}, w_{2})\in \mathbb N^{2}$ such that $\log n\le w_{2}\le w_{1}\le 2w_{2}$ and $w_{1}+w_{2}\in k\mathbb N$, it holds that the equation $(w_{1}, w_{2}) = \sum_{\textbf v\in I_{\mathcal{P}}^\mu(H)} c_{\textbf v} \textbf v$ has non-negative integer solutions.
\end{proposition}

We remark that the $\log n$ term is not essential there -- we just avoid introducing another constant.

\begin{proof}
%Suppose $H$ is not $3\mu$-even/odd-extremal.
%That is, with respect to the partition $\cP=\{V_{1}, V_{2}\}$, there are at least $k\mu n^{4}$ odd edges and $k\mu n^{4}$ even edges in $H$.
%Thus, there exists $(a,b), (c,d)\in I_{\cP}^{\mu}(H)$ such that $a$ is even and $c$ is odd.
%Thus, we have that at least one of $(4,0), (2,2)$ and $(0,4)$ is in $I_{\cP}^{\mu}(H)$ and at least one of $(3,1)$ and $(1,3)$ is in $I_{\cP}^{\mu}(H)$.
Consider all $(k-1)$-sets of vertices of index vector $(i,j)$ with $i+j=k-1$.
By the minimum codegree condition, every such set is in at least $n/3+\gamma n$ edges, which gives rise to at least $\frac1k\binom{|V_{1}|}i \binom{|V_{2}|}j\cdot (1/3+\gamma)n \ge 2\mu n^{k}$ edges of $H$, with index vector $(i+1,j)$ or $(i,j+1)$.
Thus, we have: $(\dagger)$ at least one of $(i+1,j)$ and $(i,j+1)$ is in $I_{\cP}^{\mu}(H)$, for every $0\le i\le k-1$.
%That is, at least one of the two consecutive index vectors is in $I_{\cP}^{\mu}(H)$.

We now show that there exist $\textbf v, \textbf w\in I_{\mathcal{P}}^\mu(H)$ such that $\textbf v - \textbf w = (1,-1)$.
%Note that, up to symmetry, the only situation that does not give rise to $(1,-1)$ as a difference of two vectors is that $I_{\mathcal{P}}^\mu(H) = \{(4,0), (1,3)\}$, which contradicts to the observation above.
Suppose otherwise, that is, there is no $i$ with $0\le i\le k-1$ such that  $(i+1,k-1-i), (i,k-i)\in I_{\cP}^{\mu}(H)$.
Note that $I_{\cP}^{\mu}(H)$ consists of exactly all even (or odd) vectors: property $(\dagger)$ and the absence of consecutive members together show that the robust vectors must alternate.
Therefore, $H$ has at most $k\mu n^{k}$ odd (or even) edges, yielding that $H$ is $(k\mu,\gamma/3)$-extremal, a contradiction.
%first suppose that $(2,2)\in I_{\cP}^{\mu}(H)$. Then as $H$ contains $3\mu n^{4}$ odd edges, we have $(3,1)$ or $(1,3)\in I_{\cP}^{\mu}(H)$, which together with $(2,2)$ gives $(1,-1)\in L_{\mathcal{P}}^\mu(H)$.
%So we may assume that $H$ contains less than $\mu n^{4}$ edges with index vector $(2,2)$.
%Now consider the triples of vertices of index vector $(2,1)$.
%By the minimum codegree condition, every such set is in at least $n/3+\gamma n$ edges, which gives rise to at least $\frac14\binom{|V_{1}|}2 |V_{2}|\cdot (1/3+\gamma)n \ge 2\mu n^{4}$ edges of $H$, with index vector $(2,2)$ or $(3,1)$.
%Thus, we have $(2,2)$ or $(3,1)$ is in $I_{\cP}^{\mu}(H)$.
%So we can take $\textbf v = (4,0)$ and $\textbf w = (3,1)$, or $\textbf v = (2,2)$ and $\textbf w = (1,3)$.

Now we show the proposition. 
%Given vector $(w_{1}, w_{2})\in \mathbb N^{2}$ such that $0\le w_{2}\le w_{1}\le 2w_{2}$ and $w_{1}+w_{2}\in 4\mathbb N$.
We use four (not necessarily distinct) vectors of $I_{\cP}^{\mu}(H)$, $\bfv_{1}, \bfv_{2}, \bfv_{3}, \bfv_{4}$ such that $\bfv_{1}=(i,k-i)$ with $i=k$ or $k-1$, $\bfv_{2}=(i',k-i')$ with $i'=0$ or $1$, and $\bfv_{3} - \bfv_{4} = (1,-1)$.
The existence of $\bfv_{1}$ and $\bfv_{2}$ are given by $(\dagger)$.
First let $(w_{1}', w_{2}') := (w_{1}, w_{2}) - k\bfv_{4}$ and note that $w_{1}' \le 2.1w_{2}'$ as $w_{1}, w_{2}$ are large.
Therefore $(w_{1}', w_{2}')$ is in the positive cone of $\bfv_{1}$ and $\bfv_{2}$.
Let $x_0\in \mathbb R$ be the solution to the equation $(w_{1}', w_{2}') = x\bfv_{1} + (q-x)\bfv_{2}$, where  $q=(w_{1}'+w_{2}')/k$.
Then note that for $x^*=\lfloor x_0 \rfloor$, we obtain $(\ell,-\ell):=(w_{1}', w_{2}') - x^*\bfv_{1} - (q-x^*)\bfv_{2}$ satisfies that $\ell \in \{0,1,\dots,i-i'\}$.
%
%However, we may not have integer solution to the equation $(w_{1}', w_{2}') = x\bfv_{1} + (q-x)\bfv_{2}$, where  $q=(w_{1}'+w_{2}')/k$.
% and $(w_{1}', w_{2}') = (w_{1}, w_{2}) - q\bfv_{1}$.
%Still, we can take $x\in \mathbb N$ such that $(w_{1}', w_{2}') - x\bfv_{1} - (q-x)\bfv_{2} = (i,-i)$ and $i\in \{0,1,\dots, k-1\}$.
%This is possible by e.g., screening through all choices of $0\le x\le q$ and noticing that when $x$ increases by 1 then the difference $(j,-j)$ becomes $(j,-j)-\bfv_1+\bfv_2=(j-i+i', -j+i-i')$ so that the absolute value of its coordinates drops $i-i'\in \{k-2, k-1, k\}$.
%Therefore, we can choose $x$ such that 
Therefore, as $\ell \le k$, we obtain 
\[
(w_{1}, w_{2}) =  k\bfv_{4} + x^*\bfv_{1} + (q-x^*)\bfv_{2} + \ell (\bfv_{3} - \bfv_{4}) = x^*\bfv_{1} + (q-x^*)\bfv_{2} + \ell \bfv_{3} + (k-\ell)\bfv_{4}.
\]
We are done as all coefficients are non-negative integers.
%As $(w_{1}', w_{2}') = w_{2}' (\bfv_{2} - \bfv_{3})$, we obtain that $(w_{1}, w_{2}) = q\bfv_{1} + w_{2}' (\bfv_{2} - \bfv_{3})$.
\end{proof}

The swap-absorber method is powerful but also comes with a price.
Indeed, for $k/2$-cycles, the na\"ive form of an absorbing $k/2$-path should be able to absorb any set of vertices whose size is divisible by $k/2$.
However, as in our definition of $S$-absorbers, each time we absorb a $k$-set.
Therefore, we need extra treatment for this -- we need to be able to adjust the length of the cycle before executing absorption if the set of the uncovered vertices have a wrong residue.
For $k=4$, this is handled by the following result, which gives a short odd $2$-cycle in $H$.

\begin{lemma}
\label{lem:short-odd-cycle}
Suppose $1/n\ll \e\ll \gamma$.
Let $H$ be a $4$-graph on $n$ vertices with $\delta_3(H)\geq n/3+\gamma n$.
% which is not $\beta'$-odd-extremal. 
Then either $H$ contains a $2$-cycle of length $3$ or $5$ or one can find an $(\e,\gamma)$-odd-extremal bipartition in time $O(n^{4})$.
\end{lemma}

\begin{proof}
Put $V=V(H)$ and define an auxiliary graph $G$ on $\binom V2$ by joining
two pairs $P,Q$ whenever $P\cup Q\in E(H)$. In particular, adjacent pairs
are disjoint. 
Set $N=\binom n2$. For every $P\in\binom V2$,
\[
  2\deg_G(P)=\sum_{x\in V\setminus P} \deg_H(P\cup\{x\})
  \geq (n-2)(n/3+\gamma n),
\]
and hence, for sufficiently large $n$,
\begin{equation}
\label{eq:51}
  \delta(G)\geq (1/3+\gamma/2)N>N/3.                
\end{equation}

By running the breadth-first search algorithm, in time $O(N^{2})=O(n^{4})$ we can decide whether $G$ is bipartite.
First suppose that $G$ is bipartite and let $\chi:\binom V2\to\mathbb F_2$ be a proper $2$-colouring of $G$.
%, which can be constructed by Breath-first Search in time $O(N^{2})=O(n^{4})$. 
For every
$\{a,b,x,y\}\in E(H)$, the three partitions of this edge into two pairs give
\begin{equation}
\label{eq:52}
 \chi(ab)+\chi(xy)=\chi(ax)+\chi(by)
 =\chi(ay)+\chi(bx)=1.        
\end{equation}
Fix distinct $a,b\in V$ and, for $i,j\in\mathbb F_2$, let
\[
 V_{ij}=\{x\in V\setminus\{a,b\}:(\chi(ax),\chi(bx))=(i,j)\}.
\]
If $x\in V_{ij}$ and $y\in N_H(abx)$, then \eqref{eq:52} yields
$(\chi(ay),\chi(by))=(1-j,1-i)$. Thus
\begin{equation}
\label{eq:53}
 V_{ij}\neq\varnothing\quad\Longrightarrow\quad
 |V_{1-j,1-i}|\geq\delta_3(H)>n/3.  
\end{equation}
The map $(i,j)\mapsto(1-j,1-i)$ interchanges $00$ and $11$ and fixes
$01$ and $10$. Consequently, if $V_{00}\cup V_{11}\neq\varnothing$, then
both $V_{00}$ and $V_{11}$ have size greater than $n/3$, and \eqref{eq:53} rules
out $V_{01}\cup V_{10}\neq\varnothing$. Otherwise only $V_{01}$ and
$V_{10}$ occur. Hence there is $s_{ab}\in\mathbb F_2$ such that
\begin{equation}
\label{eq:54}
 \chi(ax)+\chi(bx)=s_{ab}
 \quad\text{for every }x\in V\setminus\{a,b\}.  
\end{equation}

Choose a root $r\in V$ and put $z_u=\chi(ru)$ for $u\neq r$. For distinct
$u,v\neq r$, \eqref{eq:54} gives
\[
 \chi(uv)=s_{ur}+z_v=s_{vr}+z_u.
\]
It follows that $s_{ur}+z_u$ is independent of $u\neq r$; denote this
constant by $\sigma$, and set $z_r=\sigma$. Then
\begin{equation}
\label{eq:55}
 \chi(uv)=z_u+z_v+\sigma\qquad\text{for all }uv\in\binom V2. 
\end{equation}
By \eqref{eq:52} and \eqref{eq:55}, every edge $\{a,b,x,y\}$ of $H$ satisfies
$z_a+z_b+z_x+z_y=1$. Therefore, with
$X=\{v\in V:z_v=1\}$, every edge of $H$ meets $X$ oddly.
Moreover, $X$ can be constructed in time $O(n)$.

Choose an edge $e\in E(H)$, a vertex $x\in e\cap X$, and a vertex
$y\in e\setminus X$. Since every edge meets $X$ oddly,
\[
 N_H(e\setminus\{x\})\subseteq X
 \quad\text{and}\quad
 N_H(e\setminus\{y\})\subseteq V\setminus X.
\]
It follows that
\[
 |X|,\ |V\setminus X|\geq\delta_3(H)\geq n/3+\gamma n.
\]
As $\e\ll\gamma$, the bipartition $\{X,V\setminus X\}$ is
$(\e,\gamma)$-odd-extremal and we output it in time $O(n^4)$. 
%Thus $G$ is non-bipartite.

It remains to consider the case that $G$ is non-bipartite.
Let $C=P_1\cdots P_qP_1$ be a shortest odd cycle in $G$. Then $C$ is
induced, and every vertex outside $C$ has at most two neighbours on $C$;
otherwise the arcs between three consecutive neighbours yield a shorter
odd cycle. Consequently,
\[
 q\delta(G)\leq\sum_{i=1}^q d_G(P_i)
 \leq 2q+2(N-q)=2N.
\]
Together with \eqref{eq:51}, this gives $q<6$, so $q\in\{3,5\}$.

If $q=3$, then $P_1,P_2,P_3$ are pairwise disjoint, and they directly give
a $2$-cycle in $H$ of length $3$. Suppose that $q=5$. Since $C$ is a shortest odd
cycle, $G$ is triangle-free. For every $i\in\mathbb Z/5\mathbb Z$, the
triangle-free property and \eqref{eq:51} imply
\[
\begin{aligned}
 |N_G(P_{i-1})\cap N_G(P_{i+1})|
 &\geq \deg_G(P_{i-1}) + \deg_G(P_i) + \deg_G(P_{i+1})-N\\
 &\geq 3\delta(G)-N\geq \frac{3\gamma}{2}N.
\end{aligned}
\]
Only $O(n)$ pair-vertices of $G$ meet the vertices contained in any fixed
four pairs. So we may replace $P_1,\ldots,P_5$ successively, each
time choosing the new $P_i$ in
$N_G(P_{i-1})\cap N_G(P_{i+1})$ and disjoint from the other four current
pairs. 
These replacements preserve the $5$-cycle, and after five
steps its pair-vertices are mutually disjoint. The five corresponding
edges of $H$ form a $2$-cycle of length $5$.
\end{proof}

%\red{JH. Absorption not done: need to replace Lemma 5.18.}

Finally we are ready to prove the absorbing lemma, by combining the tools developed in this subsection together with the connecting lemma.
For convenience we restate Lemma \ref{lem:abs}.

\medskip
\noindent%\begin{lemma}
\textbf{Lemma \ref{lem:abs}.}
\emph{Suppose $1/n\ll \alpha \ll \beta\ll \e\ll \gamma$.
Let $H$ be an $n$-vertex $4$-graph with $\delta_3(H)\geq n/3+\gamma n$. Then either $H$ is $(65\e,\gamma/3)$-extremal, or $H$ contains a $2$-path $P_{A}$ of length at most $\beta^2 n$ associated with a set of vertices $S^{*}$ of size at most $7\alpha n$ such that for any set $U$ of vertices satisfying that $|U|\le 2\alpha n$, $U\cap (V(P_{A})\cup S^{*})=\emptyset$, and $|U\cup S^*|\in 2\mathbb N$, there is another $2$-path on $V(P_{A})\cup U\cup S^{*}$ which has the same ends as $P_{A}$.
Moreover, in the former case a $(65\e,\gamma/3)$-extremal bipartition can be found in time $O(n^{13})$.}

\begin{proof}[Proof of Lemma \ref{lem:abs}]
%First let $t$ be the constant determined by Lemma \ref{lem:closed}.
%We apply Proposition \ref{prop:solution}, 
Choose an additional constants such that $1/n\ll \alpha \ll \beta\ll \b_{conn}\ll\mu\ll \e\ll \gamma$.
We apply Lemma \ref{lem:short-odd-cycle} and Lemma \ref{lem:conn} to $H$, which takes time $O(n^{12})$, and if either of them finds a $(65\e,\gamma/3)$-extremal bipartition, then we output the extremal bipartition and halt.
Otherwise, by Lemma \ref{lem:P} applied to $H$ with $S=V(H)$, $\delta'=1/3+\gamma/2$, and $c=2$, in time $O(n^{13})$ we obtain that either $V(H)$ is $(\beta, 2)$-closed, or there exists a bipartition $\cP=\{V_{1}, V_{2}\}$ such that for $i=1,2$, $|V_i|\ge (1/3 + \gamma /3) n$ and $V_{i}$ is $(\beta, 2)$-closed.
In the latter case, we check whether the bipartition $\cP=\{V_{1}, V_{2}\}$ is $(4\mu,\gamma/3)$-extremal by screening the edges in time $O(n^{4})$ and if yes, we output the bipartition and halt.
%So we may assume that in this case $\cP=\{V_{1}, V_{2}\}$ is not $4\mu$-extremal, which allows us to apply Proposition~\ref{prop:solution}.

So far the applications of the auxiliary results take time $O(n^{13})$, and we may assume that none of them returns an extremal bipartition, and thus we can use the structural property of Lemma \ref{lem:short-odd-cycle} and Lemma \ref{lem:conn}, that is, $H$ contains a $2$-cycle $C$ of length 3 or 5, and all pairs of $H$ are $(\beta_{conn},4)$-connectable.
Furthermore, if $\cP=\{V_{1}, V_{2}\}$, then this bipartition is not $(4\mu,\gamma/3)$-extremal, which allows us to apply Proposition~\ref{prop:solution}.
Below we show the existence of the absorbing path.

By Proposition~\ref{prop:abs} with $r=1$ or $2$, every $S$ with index vector in $I_{\cP}^{\mu}(H)$ has at least ${\beta^6}n^b$ distinct $S$-absorbers in $H$, where  $b=24\cdot 2+4=52$.
%Indeed, in the former case, this is true as $I_{\cP}^{\mu}(H)=\{(4)\}$; and in the latter case, this is true as by Proposition \ref{prop:solution}.
Moreover, by Lemma \ref{lem:short-odd-cycle}, $H$ contains a $2$-cycle $C$ of length 3 or 5, and by Lemma \ref{lem:conn}, every two pairs of vertices of $H$ are $(\beta, 4)$-connectable (we do not need the set $R$ in this proof).
Our goal is to build a desired absorbing path such that it contains many $S$-absorbers for every $S$ with $\bfi_{\cP}(S)\in I_{\cP}^{\mu}(H)$, 
and show that we can always partition the vertices to be absorbed into $4$-sets with index vector in $I_{\cP}^{\mu}(H)$ by Proposition~\ref{prop:solution}.
The short odd cycle $C$ will be used to guarantee that the number of vertices to be absorbed is a multiple of 4.

Write $\beta_{0}:=\beta^7$ and $t=2$.
We start with choosing a random family $\F$ of $b$-sets of vertices by including each $b$-set of vertices in $V(H)\setminus V(C)$ independently with probability $p:=\beta_{0} n^{-b+1}$.
By applying Chernoff's inequality (for (1) and (2) below) and Markov's inequality (for (3)) and the union bound, there exists a family $\F$ satisfying the following properties:
\begin{enumerate}
\item $|\F|\le 2p\binom nb\le \beta_{0} n/b$;
\item for every $S$ with $\bfi_{\cP}(S)\in I_{\cP}^{\mu}(H)$, $\F$ contains at least $(p/2)\beta^{6} n^{b} =\beta^{13}n/2 $ $S$-absorbers;
\item there are at most $2p^{2}b\binom nb \binom n{b-1}\le \beta^{14}n$ pairs of overlapping members of $\F$.
\end{enumerate}
By deleting one set from each overlapping pair of members of $\F$ and the members that are not $S$-absorbers for any $S$, we obtain a family $\F'$ of $b$-sets such that i) $|\F'|\le \beta_{0} n/b$, ii) each $b$-set spans a family of $4t+1$ vertex-disjoint $2$-paths that is an $S$-absorber for some $S\in \binom V4$, and iii) for every $S$ with $\bfi_{\cP}(S)\in I_{\cP}^{\mu}(H)$, $\F'$ contains at least $\beta^{13}n/2 - \beta^{14}n\ge 12(t+1)\alpha n$ $S$-absorbers, as $\alpha \ll \b$.
%Then $\mathcal F'$ is a family of disjoint $b$-sets of vertices such that ...

Let $C=S_{1}\cdots S_{q}S_{1}$ where $q=3$ or 5.
If $q=3$, then let $P_{C}=S_{1}S_{3}$ and $P_{C}'=S_{1}S_{2}S_{3}$; if $q=5$, then let $P_{C}=S_{1}S_{5}S_{4}S_{3}$ and $P_{C}'=S_{1}S_{2}S_{3}$. 
Let $T$ be a set of vertices such that $|T\cap V_{i}| \in [3\alpha n-4, 3\alpha n]$ for each $i\in [|\cP|]$, $|T|\in 4\mathbb N$, and $T$ is disjoint from all existing vertices.
%$T\cap (V(P_{A})\cup C)=\emptyset$.
Let $S^{*}=S_{2}\cup T$ and thus $|S^*| = 2\pmod 4$.
%, where i) if $\cP=\{V\}$, then $T=\emptyset$; if $\cP=\{V_{1}, V_{2}\}$, then $T$ consists of $3\alpha n$ vertices from $V_{1}$ and $3\alpha n$ vertices from $V_{2}$ disjoint from the existing vertices and satisfies that $|T|\in 4\mathbb N$.
Now since every two pairs of vertices of $H$ are $(\beta_{conn}, 4)$-connectable, we can greedily connect the $2$-paths in each member of $\mathcal F'$ and the $2$-path $P_{C}$.
This is possible as $\beta \ll \b_{conn}$: there are at most $(4t+1)\beta_{0} n +1$ such $2$-paths to connect and thus at most $8((4t+1)\beta_{0} n +1)+\beta_{0}n+10+6\alpha n\le \beta^{2} n$ vertices to avoid, while for every pair of $2$-sets in $H$ there are at least $\beta_{conn} n^{2i}$ connectors of size $2i$ for some $i\le 4$.

Denote the resulting path by $P_{A}$ and we claim it together with $S^{*}$ has the desired absorbing property.
Indeed, take any set $U$ of vertices such that $|U|\le 2\alpha n$, $U\cap (V(P_{A})\cup S^{*})=\emptyset$, and $|U|\in 2\mathbb N$.
If $|U|$ is not a multiple of $4$, then let $U'=U\cup S^{*}$ and $P'=P_{A}$; otherwise let $P'$ be the path obtained from $P_{A}$ by replacing the subpath $P_{C}$ by $P_{C}'$, and let $U'=U\cup T$ if $q=3$ and $U'=U\cup T\cup S_{5}\cup S_{4}$ if $q=5$.
Therefore $P'$ has the same ends as $P_{A}$, $|U'|\in 4\mathbb N$ and $3\alpha n \le |U'|\le 8\alpha n+4$.
By the definition of $T$, if $\cP=\{V_{1}, V_{2}\}$, then $U'$ satisfies that $\frac12 \le \frac{m_{1}}{m_{2}}\le 2$ and $m_{1}, m_{2}\ge \log n$ for $\bfi_{\cP}(U') = (m_{1}, m_{2})$.
Thus, we can partition $U'$ into at most $2\alpha n+1$ sets of size $4$ each of whose index vector is in $I_{\cP}^{\mu}(H)$, and thus each of them has at least $12(t+1)\alpha n$ absorbers in $P_{A}$, therefore also in $P'$.
Indeed, for $\cP=\{V(H)\}$ this is trivial because all $4$-sets have index vector $(4)\in I_{\cP}^{\mu}(H)$.
For $\cP=\{V_{1}, V_{2}\}$ Proposition~\ref{prop:solution} (one may need to interchange $V_{1}$ and $V_{2}$) gives an integer solution of the coefficients for vectors in $I_{\cP}^{\mu}(H)$, which can be realized by partitioning the vertices of $U'$ into $4$-sets according to the solution, because the total number of vertices requested from each part agrees with $|U'\cap V_{i}|$.
As each $S$-absorber consists of $4t+1$ disjoint subpaths of the absorbing path, we can proceed with the absorption greedily and obtain a $2$-path on $V(P_{A})\cup U\cup S^{*}$ which has the same ends as $P_{A}$.

Finally, since each $2$-path in the members of $\F$ has length $2$ and each connection uses paths of length at most 5, the length of $P_{A}$ is at most $5\cdot 2(4t+1)\beta_{0}n+3\le \beta^2 n$. The proof is completed.
\end{proof}

As one can see from the above proof, one can prove such an absorbing lemma for $k$-uniform Hamilton $k/2$-cycles for all even integer $k\ge 4$ provided a general version of Lemma~\ref{lem:short-odd-cycle} of short odd cycles for higher uniformity.

%\red{JH. Absorbing DONE.}

% \begin{lemma}
% \cite{MR4912875}
   %  Let $\alpha > 0$, and integers $c, k \geq 2$ be given and suppose $1/n \ll \delta' \ll \alpha, 1/k, 1/c$.  
%Assume that $H$ is a $k$-graph on $n$ vertices satisfying that every set of $c+1$ vertices contains two vertices that are $(2\alpha, 1)$-reachable in $H$. Then in time $O(cn^{k+1})$ we can find a set of vertices $S \subseteq V(H)$ with $|S| \geq (1-c\delta')n$ such that $|\tilde{N}_{\alpha,1}(v, H[S])| \geq \delta' n$ for any $v \in S$.
 %\end{lemma}
 \section{The even-extremal case}
 \label{sec:even}
We focus on the extremal case in this section. 
It reduces the existence of a Hamilton $2$-cycle in an even-extremal $4$-graph to a finite set of parity-type alternatives encoded by the notion of even-goodness.
We begin by fixing the hierarchy of constants and collecting the definitions
which will be used throughout the even-extremal argument. 

\subsection{Setup and basic definitions}

\begin{setup} \label{setup:evenextr}
Fix constants satisfying 
$1/n  \ll c' \ll c\ll \theta \ll \eta \ll \beta \ll \beta_2
\ll \beta'_2 \ll \beta_1 \ll \beta'_1 \ll \rho \ll \tau \ll \mu
\ll \xi \ll \gamma \ll 1$.
Let $H$ be a $4$-graph of order $n$, and let $V = V(H)$.
\end{setup}

Now suppose that $H$ is a $4$-graph. Given a bipartition $\{A,B\}$ of
$V$, let $\mathcal C:=\binom A2\cup\binom B2$. 
For a pair $p$, define 
$N_{\mathrm{sp}}(p):=N_H(p)\cap AB$ and $N_{\mathrm{con}}(p):=N_H(p)\cap \C$ be the number of split pairs and connate pairs among neighbours of $p$, respectively.
%$N_{\mathrm{sp}}(p):=\{s\in A\times B:s\cap p=\emptyset,\ s\cup p\in E(H)\}$ and $N_{\mathrm{con}}(p):=\{c\in\mathcal C:c\cap p=\emptyset,\ c\cup p\in E(H)\}$.
We concatenate 2-paths in the natural way, for
example, if $P$ is a 2-path with ends $p$ and $p_0$, and $Q$ is a 2-path with ends $p_0$ and $q$, and $P$ and $Q$ have no vertices in common outside $p_0$, then we write $pPp_0Qq$ for the $2$-path obtained by concatenating $P$ and $Q$ at their common end
$p_0$. 
Here the common pair $p_0$ is identified, not repeated.
%\begin{theorem}[Even-extremal case]
%\label{thm:evenextremal}
 %   Suppose that $1/n\ll\gamma,c'\ll1$ and let $H$ be an $n$-vertex 4-graph with $\delta_3(H)\geq n/3+\gamma n$ and $\{A',B'\}$ be a $(c',\g/3)$-even-extremal bipartition of $V$.
%    Then we construct a bipartition $\{A, B\}$ of $V$ in time $O(n^4)$ such that $H$ contains a Hamilton 2-cycle if and only if $\{A, B\}$ is even-good, which can be checked in time $O(n^8)$.
%\end{theorem}
We also need the following definitions.
\begin{definition} \label{def:good}
Under Setup~\ref{setup:evenextr}, for a fixed bipartition $\{A, B\}$ of $V$, we say that
\begin{enumerate}[label=$(\roman*)$]
\item for each $t\in[3]$, a $t$-subset of $V$ is \emph{$\theta$-good} if it is contained in at most $\theta^{2t-1}\binom{n}{4-t}$ odd edges,
\item a pair $p \in \binom{V}{2}$ is \emph{$\beta_2$-medium} if it is contained in at least $\beta_2\binom{n}{2}$ even edges; otherwise, it is \emph{$\beta_2$-bad},
%\item a vertex $v\in V$ is \emph{$(\beta_1,\beta_2)$-medium} if it is contained in at least $\beta_1 n$-many $\beta_2$-medium pairs; otherwise, it is \emph{$(\beta_1,\beta_2)$-bad},
\item A vertex $v\in V$ is \emph{$(\beta_1,\beta_2)$-medium} if $v$ is contained in at least $\beta_1 n$ $\beta_2$-medium connate pairs and at least $\beta_1 n$ $\beta_2$-medium split pairs; otherwise it is \emph{$(\beta_1,\beta_2)$-bad}.
\end{enumerate}
\end{definition}

%Here condition (\ref{E0}) does not need odd-edge transition bridge, while conditions (\ref{E1})--(\ref{E4}) correspond to odd-edge transition bridges.

A graph $G$ is \emph{Hamilton-connected}
if every two vertices of $G$ are connected by a Hamilton path.

\begin{lemma}\cite{Ore1963HamiltonConnected}\label{lem:Hamiltoncon}
    A graph $G$ of order at least 3 is Hamilton-connected if $d(u)\ge (|V(G)|+1)/2$ for each vertex $u$ of $G$.
\end{lemma}

\subsection{Consequences of even-extremality}
We next record the basic structural consequences of even-extremality.  Since there are few odd edges, there are few bad pairs and few bad vertices.
These estimates will be used repeatedly later.

\begin{proposition}\label{prop:types}
Assume Setup~\ref{setup:evenextr}, and suppose that $\delta_3(H)\geq n/3+\g n$ and that $\{A, B\}$ is a $(c,\g/4)$-even-extremal bipartition of $V$. Then
\begin{enumerate}[label=$(\alph*)$]
\item At most $\eta n^t$ many $t$-sets are non $\theta$-good for $t\in[3]$,
%\item there are at most $\eta n$ vertices which are not $\theta$-good,
%\red{Can merge (a) and (b)}
%\item there are at most $\frac{c}{(1-\beta_1)(\frac{1}{3}+\g/2-\beta_2)}n$ vertices which are $(\beta_1,\beta_2)$-bad and
\item At most $36cn^2$ pairs are $\beta_2$-bad, and
\item At most $288cn$ vertices are $(\beta_1,\beta_2)$-bad.
\end{enumerate}
\end{proposition}
\begin{proof}
For every $p\in {V\choose t}$, by the minimum codegree condition and Fact~\ref{degreeinherit}, we have
$\deg_H(p)\ge (1/3+\g){n-t\choose 4-t}$.
% for all sufficiently large $n$. Indeed, summing the codegrees of all triples
% containing $p$, each edge containing $p$ is counted twice. Similarly, for every
% vertex $v\in V$, we have
% $d_H(v)\ge (1/3+\g/2){n\choose 3}$, since summing over all triples
% containing $v$ counts each edge containing $v$ three times.

We first prove $(a)$. Suppose, to the contrary, that there are more than
$\eta n^t$ $t$-sets which are not
$\theta$-good. 
%If a $t$-set $p$ is not $\theta$-good, then $p$ is contained in at least $d_H(p)-\theta^3\binom n{4-t}\ge\frac13{n\choose {4-t}}$ odd edges.
Then the number of odd edges of $H$ is more than
${\binom{4}{t}}^{-1}\cdot \eta n^t\cdot
\theta^{2t-1}{n\choose 4-t}\ge c{n^ 4}$ since $c\ll\theta\ll\eta$,
where the factor ${\binom{4}{t}}^{-1}$ appears because each odd edge contains $\binom 4t$ many $t$-sets. 
This contradicts the assumption that $\{A,B\}$ is $(c,\g/4)$-even-extremal. 
%Suppose that there are more than $\frac{c}{\g/2-\theta}n$ vertices which are not $\theta$-good. If $v$ is not $\theta$-good, then $v$ is contained in fewer than $(1/3+\theta){n\choose 3}$ even edges. Since $d_H(v)\ge (1/3+\g/2){n\choose 3}$, the vertex $v$ is contained in at least $(\g/2-\theta){n\choose 3}$ odd edges. Thus the number of odd edges is more than $\frac14\cdot \frac{c}{\g/2-\theta}n\cdot (\g/2-\theta){n\choose 3}\ge c{n\choose 4}$, where the factor $1/4$ appears because each odd edge contains four vertices. This again contradicts $(c,\g/4)$-even-extremality. Hence (b) holds.
%For (c), let $\lambda:=1/3+\g/2-\beta_2$, which is positive since $\beta_2\ll \g$. Suppose that there are more than $\frac{c}{(1-\beta_1)\lambda}n$ vertices which are $(\beta_1,\beta_2)$-bad. If $v$ is $(\beta_1,\beta_2)$-bad, then fewer than $\beta_1 n$ pairs containing $v$ are $\beta_2$-medium. Hence at least $(1-\beta_1)n$ pairs containing $v$ are not $\beta_2$-medium. For every such pair $p$, we have $d_{H_{\mathrm{even}}}(p)<\beta_2{n\choose 2}$. Since $d_H(p)\ge (1/3+\g/2){n\choose 2}$, the pair $p$ is contained in at least $\lambda{n\choose 2}$ odd edges. Therefore the number of odd edges is more than $\frac1{12}\cdot \frac{c}{(1-\beta_1)\lambda}n\cdot (1-\beta_1)n\cdot \lambda {n\choose 2}\ge c{n\choose 4}$. Here the factor $1/12$ is used because an odd edge is counted at most four times by the choice of the bad vertex and, for each such vertex, at most three times by the choice of the pair containing it. This contradicts $(c,\g/4)$-even-extremality, and proves (c).

Now we show $(b)$.
Let $\mathcal B$ be the family of pairs which are $\beta_2$-bad. 
If $p\in\mathcal B$, then
$p$ is contained in fewer than $\beta_2\binom n2$ even edges. 
On the other hand, by 
%the pair-degree lower bound following from the minimum
%codegree condition, $p$ is contained in at least
$\deg_H(p)\ge (1/3+\gamma)\binom {n-2}2$, we infer that $p$ is contained in at least $(1/3+\gamma)\binom {n-2}2 - \beta_2\binom n2\ge n^2/6$ odd edges. 
Since $\{A,B\}$ is
$c$-even-extremal, the number of odd edges is at most $c n^4$.
Counting pairs inside odd edges gives
$|\mathcal B|n^2/6\le 6c n^4$, and thus
$|\mathcal B|\le 36cn^2$.

For $(c)$, if a vertex $v$ is $(\beta_1,\beta_2)$-bad with respect
to $\{A,B\}$, then $v$ is incident with fewer than $\beta_1n$
$\beta_2$-medium connate pairs, or it is incident with fewer than
$\beta_1n$ $\beta_2$-medium split pairs. 
Suppose $v\in A$. 
In the first
case, at least $n/3+\g n/4-1-\beta_1n\ge n/4$ connate pairs containing $v$ are
$\beta_2$-bad, and hence lie in $\mathcal B$. 
In the second case, at least $n/4$ split pairs containing $v$ are
$\beta_2$-bad.
Therefore, $v$ is incident with at least $n/4$
pairs from $\mathcal B$.
Denote the set of $(\beta_1,\beta_2)$-bad vertices by $L$.
Therefore
$|L|n/4\le 2|\mathcal B|\le 72c n^2$, and hence $|L|\le 288cn$.
\end{proof}

The previous proposition shows that bad objects are rare.  
We now use these properties to obtain a local connecting tool.  
Roughly speaking, any two good pairs can be connected by a short path
while avoiding any prescribed small set of vertices.

\begin{proposition}\label{prop:pairconnect}
Assume Setup~\ref{setup:evenextr}, and fix a $(c,\g/4)$-even-extremal bipartition $\{A,B\}$ of $V$. 
Let $R\subseteq V$ with $|R|\le\mu n$.
%and suppose that at most $\rho n^2$ pairs of $V$ are not $\theta$-good. 
%satisfy $|A\cap R|\le \mu n$ and $|B\cap R|\le \mu n$.
Then the following holds.
\begin{enumerate}
    \item For any two disjoint $\theta$-good split pairs $s_1$ and $s_2$,
    there exists a  split pair $s_3\in {V\setminus R\choose 2}$ such that
    $s_1\cup s_3\in E(H)$ and $s_3\cup s_2\in E(H)$.

    \item For any two disjoint $\theta$-good connate pairs $p_1$ and $p_2$,
    there exists a  connate pair $p_3\in {V\setminus R\choose 2}$ such that
    $p_1\cup p_3\in E(H)$ and $p_3\cup p_2\in E(H)$.

    % \item For any $\theta$-good split pair $s_1$, there exists a
    % $\theta$-good split pair $s_2\in {V\setminus R\choose 2}$ such that
    % $s_1\cup s_2\in E(H)$. \red{Might be redundant.}

    % \item For any $\theta$-good connate pair $p_1$, there exists a
    % $\theta$-good connate pair $p_2\in {V\setminus R\choose 2}$ such that
    % $p_1\cup p_2\in E(H)$.
\end{enumerate}
\end{proposition}

\begin{proof}
Write $a:=|A|$ and $b:=|B|$. Since $\{A,B\}$ is $(c,\g/4)$-even-extremal, we have
$a,b\ge n/3+\g n/4$.
Recall that a $\theta$-good pair $s$ is contained in at least
$(1/3+\g-\theta^3)\binom n2\ge(1/3+\g/2)\binom n2$ even edges. 
Thus, if $s$ is split, then $|N_{\mathrm{sp}}(s)|\ge (1/3+\g/2)\binom n2$; if $s$ is connate, then $|N_{\mathrm{con}}(s)|\ge (1/3+\g/2)\binom n2$.

(1) Let $s_1,s_2$ be two disjoint $\theta$-good split pairs.
Then $N_{\mathrm{sp}}(s_1),N_{\mathrm{sp}}(s_2)\subseteq A B$. Hence by inclusion-exclusion and $ab\le n^2/4$, we get
%$|N_{\mathrm{sp}}(s_1)\cap N_{\mathrm{sp}}(s_2)| \ge |N_{\mathrm{sp}}(s_1)|+|N_{\mathrm{sp}}(s_2)|-ab$.
%Since $ab\le n^2/4$, for all sufficiently large $n$ we get
$|N_{\mathrm{sp}}(s_1)\cap N_{\mathrm{sp}}(s_2)|
\ge 2(1/3+\g/2)\binom n2- ab\ge n^2/20$.
Note that the number of split pairs that meet $R$ is at most
$|R|n\le 2\mu n^2$. Since $\mu\ll 1$, there exists a split pair $s_3\in N_{\mathrm{sp}}(s_1)\cap N_{\mathrm{sp}}(s_2)$ with
$s_3\subseteq V\setminus R$. Then $s_1\cup s_3\in E(H)$ and
$s_3\cup s_2\in E(H)$.

(2) 
%If a connate pair $p$ is $\theta$-good, then $p$ is contained in at least
%$(1/3+\theta^3)\binom n2$ even edges. 
%Thus, $|N_{\mathrm{con}}(p)|\ge (1/3+\theta^3)\binom n2$.
Let $p_1,p_2$ be two disjoint $\theta$-good connate pairs.
Then $N_{\mathrm{con}}(p_1),N_{\mathrm{con}}(p_2)\subseteq\binom{A}2\cup\binom{B}2$.
Let $\mathcal C=\binom{A}2\cup\binom{B}2$.
Moreover, since $a+b=n$ and $a,b\ge (1/3+\g/4)n$, we have
$|\mathcal C|=\binom a2+\binom b2\le 5n^2/18$.
Therefore
$|N_{\mathrm{con}}(p_1)\cap N_{\mathrm{con}}(p_2)|
\ge |N_{\mathrm{con}}(p_1)|+|N_{\mathrm{con}}(p_2)|-|\mathcal C|
\ge 2(1/3+\g/2)\binom n2-5n^2/18\ge n^2/30$
for sufficiently large $n$. The number of connate pairs that meet $R$ is
at most $|R|n\le 2\mu n^2$. Since $\mu\ll1$, there
exists a connate pair $p_3\in N_{\mathrm{con}}(p_1)\cap N_{\mathrm{con}}(p_2)$
with $p_3\subseteq V\setminus R$. Then $p_1\cup p_3\in E(H)$ and
$p_3\cup p_2\in E(H)$. 
%
%Note that (3) and (4) are easier versions of (1) and (2), respectively.
% (3) Let $s_1$ be a $\theta$-good split pair disjoint from $R$. As above,
% $|N_{\mathrm{sp}}(s_1)|\ge (1/3+\theta^3)\binom n2\ge n^2/7$ for all
% sufficiently large $n$. Among these split pairs, at most $\rho n^2$ are not
% $\theta$-good, and at most $2\mu n^2$ meet $R$. Since $\rho,\mu\ll1$, there
% exists a $\theta$-good split pair $s_2\in N_{\mathrm{sp}}(s_1)$ with
% $s_2\subseteq V\setminus R$. Then $s_1\cup s_2\in E(H)$.
%
% (4) Similarly, let $p_1$ be a $\theta$-good connate pair disjoint from $R$.
% As above, $|N_{\mathrm{con}}(p_1)|\ge (1/3+\theta^3)\binom n2\ge n^2/7$
% for all sufficiently large $n$. Among these connate pairs, at most
% $\rho n^2$ are not $\theta$-good, and at most $2\mu n^2$ meet $R$. Since
% $\rho,\mu\ll1$, there exists a $\theta$-good connate pair
% $p_2\in N_{\mathrm{con}}(p_1)$ with $p_2\subseteq V\setminus R$. Then
% $p_1\cup p_2\in E(H)$.
\end{proof}

Our final result for constructing a Hamilton 2-path in an auxiliary 4-graph $G$ uses the following type of codegree condition: every vertex is in $o(n^2)$ triples which have codegree less than $(1/2+o(1))|V(G)|$ (there is also a bipartite variant of it). 
To prepare to plug in such a result we make the following definitions.

\begin{definition}\label{def:exc}
Assume Setup~\ref{setup:evenextr}, given a fixed bipartition $\{A, B\}$ of $V$. 
Let $\mathcal T_Y:=\{T\in\binom Y3:\deg_{H[Y]}(T)<(1/2+2\xi)|Y|\}$ where $Y\in\{A,B\}$ and let $\mathcal E_{\mathrm{sp}}:=\{s\in AB: |N_H(s)\cap A B|<(1/2+2\xi)|A||B|\}$.
We say that
\begin{enumerate}[label=(\roman*)]
\item A vertex $v\in Y$ is called $Y$-\emph{exceptional} if $v$ is contained in more than
$\frac{\rho}{4}|Y|^2$ triples of $\mathcal T_Y$,
\item A vertex $v\in V$ is called \emph{split-exceptional} if it is contained in more than
$\frac{\rho}{4}n$ pairs of $\mathcal E_{\mathrm{sp}}$,
\item A pair $p\in\binom Y2$ is called
$Y$-\emph{atypical} if $p$ is contained in more than $\frac{\tau}{4}|Y|$ triples of $\mathcal T_Y$.
\end{enumerate}
Finally, we say a vertex is exceptional if it is $Y$-exceptional for $Y\in \{A, B\}$ or it is split-exceptional, and denote by $X_E$ the set of exceptional vertices.
Let $\mathcal P_{at}\subseteq \binom A2\cup \binom B2$ be the family of pairs that are either $A$-atypical or $B$-atypical.
\end{definition}

%Denote the set of $Y$-exceptional vertices by $X_Y$ where $Y\in\{A,B\}$.
%Let $X_{\mathrm {int}}=X_A\cup X_B$.
%Denote the set of split-exceptional vertices by $X_{\mathrm{sp}}$.

\begin{claim}\label{claim:few-exceptional-vertices}
Assume Setup \ref{setup:evenextr}.
Let $H$ be an $n$-vertex 4-graph with $\delta_3(H)\geq n/3+\gamma n$ and $\{A,B\}$ be a $(c,\g/4)$-even-extremal bipartition of $H$.
Then we have $|\cP_{at}|\le \eta n^2$, $|\mathcal E_{\mathrm{sp}}|\le 60cn^2$ and $|X_E|\le \eta n$.
% The number of $A$-atypical pairs and
% $B$-atypical pairs is at most $\eta n^2$ and $|\mathcal E_{\mathrm{sp}}|\le5cn^2$.
% %$|\mathcal T_A|,|\mathcal T_B|\le3c\g^{-1}n^3$.
% We also have
% $|X_{\mathrm {int}}|\le\eta n$ and $|X_{\mathrm {sp}}|\le \eta n$. 
\end{claim}

\begin{proof}
Let $T\in\mathcal T_A$.
Since $\delta_3(H)\ge (1/3+\gamma)n$ and $|A|\le (2/3-\g/4)n$, we have
$|N_H(T)\cap B|\ge (1/3+\gamma)n-(1/2+2\xi)|A|\ge \gamma n/2$,
provided $\xi\ll\gamma$.
% Since $e(H_{\mathrm{odd}})\le c n^4$, we have
% $|\mathcal T_A|\cdot \frac{\gamma n}{2}\le e(H_{\mathrm{odd}})\le c n^4$, yielding $|\mathcal T_A|\le 2c\g^{-1}n^3$.
% It is easy to see that
Let $x_Y$ be the number of $Y$-exceptional vertices for $Y\in \{A, B\}$.
Now we can bound the number of odd edges by 
$x_A\cdot \frac{\rho}{4}|A|^2\cdot \frac{\gamma n}{2}\le 3e(H_{\mathrm{odd}}) \le 3c n^4$.
Since $|A|\ge (1/3+\g/4)n$, this gives
$x_A\le 216c(\rho\g)^{-1}n\le\eta n/4$.
Same arguments give $x_B\le\eta n/4$.
%Thus, $|X_{\mathrm{int}}|\le\eta n$.

Similarly, if there are $m_A$ $A$-atypical pairs, then
$m_A\cdot \frac{\tau}{4}|A|\cdot \frac{\gamma n}{2}\le 3c n^4$,
which gives $m_A\le 72c(\g\tau)^{-1}n^2\le\eta n^2/2$.
The same argument works with $A$ replaced by $B$ and thus $|\cP_{at}|\le \eta n^2$.

Let $s\in\mathcal E_{\mathrm{sp}}$. Since $s$ is a split pair, its even
neighbours are split pairs. 
%By the pair-degree lower bound,
%$d_H(s)\ge (1/3+\gamma/2)\binom n2$.
By the definition of $\mathcal E_{\mathrm{sp}}$, $s$ has fewer than
$(1/2+2\xi)|A||B|\le (1/4+\xi)n^2/2$ even split-neighbours. 
%Since $|A||B|\le n^2/4$, 
Therefore, the number of odd completions of $s$ is at least
$(1/3+\gamma)\binom {n-2}2-(1/4+\xi)n^2/2
\ge n^2/20$,
as $\xi\ll\gamma$.
Since each odd edge contains three split pairs,
$|\mathcal E_{\mathrm{sp}}|\cdot n^2/20\le 3e(H_{\mathrm{odd}})
\le 3c n^4$ and thus $|\mathcal E_{\mathrm{sp}}|\le 60cn^2$.
Also, let $x_{\mathrm{sp}}$ be the number of split exceptional vertices, we have
$x_{\mathrm{sp}}\cdot \frac{\rho}{4}n\le 2|\mathcal E_{\mathrm{sp}}|$ and thus $x_{\mathrm{sp}}\le \eta n/2$.

Overall we obtain $|X_E|\le x_A + x_B + x_{\mathrm{sp}}\le \eta n$.
\end{proof}

For convenience we restate Lemma \ref{evenextremal}.
\medskip

\noindent
%\begin{lemma}
\textbf{Lemma \ref{evenextremal} (Even-extremal case.)}
%\label{evenextremal}
\emph{Suppose that $1/n\ll c'\ll\gamma\ll1$ and let $H$ be an $n$-vertex $4$-graph with $\delta_3(H)\geq n/3+\gamma n$ and $\{A',B'\}$ be a $(c',\g/3)$-even-extremal bipartition of $V$.
Then we construct a bipartition $\{A, B\}$ of $V$ in time $O(n^4)$ such that $H$ contains a Hamilton $2$-cycle if and only if $\{A, B\}$ is even-good, which can be checked in time $O(n^8)$.}
%\end{lemma}

\subsection{Proof of Lemma \ref{evenextremal}}
Before delving into the details, we give an overview of the proof.
The proof of Lemma \ref{evenextremal} consists of three
steps and we will also introduce some auxiliary results, whose proofs we will defer later.

\emph{Step 1: Refine the extremal bipartition.}
Starting from a $(c',\g/3)$-even-extremal bipartition $\{A',B'\}$, we move every vertex which is $(\beta'_1,\beta'_2)$-bad to the opposite side. This produces a new bipartition
$\{A,B\}$ which is $(c,\g/4)$-even-extremal, but has much better local properties.
%almost all vertices and pairs are $\theta$-good, and every vertex is medium with respect to the new partition. 
%This cleaning step ensures that the later local connecting arguments can be carried out inside the even structure determined by $\{A_1,B_1\}$.

\emph{Step 2: Construct a bridge.}
We next build a small bridge which breaks the parity obstruction and contains all exceptional vertices. 
There are two possible
types of bridges. A \emph{connate bridge} is a path whose two end pairs are connate pairs with good properties. 
A \emph{split bridge} is a
path whose two end pairs are split pairs with good properties. 
%The construction of the core bridge is divided according to the bridge-good cases. 
%Next, we extend it to a short path which contains all vertices that are exceptional for the final covering step.  This extension preserves the relevant parity condition.
%: in the connate case both $|A_1\setminus V(Q)|$ and $|B_1\setminus V(Q)|$ are even, while in the split case $|A_1\setminus V(P)|=|B_1\setminus V(P)|$.

\emph{Step 3: Cover the remaining vertices.}
Finally, we cover the unused vertices. 
In the connate-bridge case, we choose an edge $p_Ap_B$ between $A$ and $B$ where $p_A, p_B$ are connate pairs and apply Lemma \ref{lem:connateHP}  separately inside $A$ and inside $B$.
In the split-bridge case, the bridge leaves a balanced bipartite graph, Then apply Lemma \ref{lem:splitHP}. 
Each case gives a Hamilton
$2$-cycle of $H$.

Now we give the necessary lemmas for the proof of Lemma \ref{evenextremal}.
Given a $(c',\g/3)$-even-extremal bipartition $\{A',B'\}$ of $V$.
Define $A_{\mathrm{bad}}$ to be the set of vertices in $A'$ which are $(\beta'_1,\beta'_2)$-bad with respect to $\{A',B'\}$ and $B_{\mathrm{bad}}$ to be the set of vertices in $B'$ which are $(\beta'_1,\beta'_2)$-bad with respect to $\{A',B'\}$.
Put
$A:=(A'\setminus A_{\mathrm{bad}})\cup B_{\mathrm{bad}}$ and
$B:=(B'\setminus B_{\mathrm{bad}})\cup A_{\mathrm{bad}}$.
We say that the vertices in $A_{\mathrm{bad}}\cup B_{\mathrm{bad}}$ are
\emph{moved}. 

We call this the \emph{moving process} of $\{A',B'\}$.
The first lemma is the refining step. 

\begin{lemma}\label{lem:cleaned-even-partition}
Assume Setup~\ref{setup:evenextr}.
Executing the moving process of a $(c',\g/3)$-even-extremal bipartition $\{A',B'\}$,
we obtain a bipartition $\{A,B\}$ which is $(c,\g/4)$-even-extremal. Moreover, with respect to
$\{A,B\}$ every vertex of $H$ is $(\beta_1,\beta_2)$-medium.
\end{lemma}
%For convenience, denote the set of $\theta$-good pairs by $\mathcal G_\theta$, the set of $Y$-atypical pairs by $\mathcal A_Y$ where $Y\in\{A,B\}$.

Once the bipartition has been refined, we construct a bridge. 
The bridge is short, avoids the future reservoir, and covers all exceptional vertices.
More importantly, it leaves behind a vertex set satisfying the divisibility
condition required by one of the two Hamilton path lemmas.

\begin{lemma}\label{lem:bridge}
Assume Setup~\ref{setup:evenextr}, and let $\{A,B\}$ be a fixed $(c,\g/4)$-even-extremal
bipartition of $V$.  
Suppose that $\{A,B\}$ is even-good and every vertex is $(\beta_1,\beta_2)$-medium. 
Then we obtain either
\begin{itemize}
    \item a path $Q$ with $\theta$-good connate ends such that $|A\setminus V(Q)|,\ |B\setminus V(Q)|$ are even, and $X_E\subseteq V(Q)$. Moreover, $|V(Q)|\le \mu n$; or
\item a path $P$
with  $\theta$-good split ends such that
$|A\setminus V(P)|=|B\setminus V(P)|$, and $X_E\subseteq V(P)$.
Moreover, $|V(P)|\le \mu n$.
\end{itemize}
\end{lemma}

We use the following two lemmas to cover the vertices not used by the bridge.

\begin{lemma}\label{lem:connateHP}
Assume Setup~\ref{setup:evenextr}. Let $G$ be a $4$-graph on
a vertex set $U$ with $|U|=n$ even. 
Let $\mathcal T_U=\{T\in\binom{U}{3}:\deg_G(T)<(1/2+\xi)n\}$.
Suppose that 
\begin{enumerate}[label=(C\arabic*), ref=C\arabic*, start=1]
    \item Every vertex $v\in U$ is contained in at most
    $\rho n^2$ triples of $\mathcal T_U$.
\label{C1}
    \item There are disjoint pairs $p,q\in\binom{U}{2}$ each of which is contained in at
    most $\tau n$ triples of $\mathcal T_U$.
    \label{C2}
\end{enumerate}

Then $G$ contains a Hamilton $2$-path with ends $p$ and $q$.
\end{lemma}
\begin{lemma}\label{lem:splitHP}
Assume Setup \ref{setup:evenextr}.
Let $G$ be a $4$-graph on
$U=A_U\cup B_U$, where $|A_U|=|B_U|=m$ and $n/4\le m<n$. 
Let 
%$\mathcal E_U:=\{s\in A_U\times B_U:|\{s'\in A_U\times B_U: s\cup s'\in E(G)\}|<(1/2+2\xi)|A_U||B_U|\}$.
$\mathcal E_U:=\{s\in A_U B_U:| N_G(s)\cap A_U B_U|<(1/2+2\xi)m^2\}$.
Suppose that 

\begin{enumerate}[label=(S\arabic*), ref=S\arabic*, start=1]
\item Every vertex $v\in U$ is contained in at most $\rho m$
    pairs of $\mathcal E_U$.
\label{S1}
    \item The prescribed disjoint split pairs $s_1,s_2\in
    A_U B_U$ are not in $\mathcal E_U$.
    \label{S2}
\end{enumerate}

Then $G$ contains a Hamilton $2$-path with ends $s_1$ and $s_2$ whose
pair sequence consists only of split pairs. 
\end{lemma}

\begin{proof}[Proof of Lemma \ref{evenextremal}]
%Choose constants satisfying
%$1/n \ll \theta \ll c' \ll c \ll \eta \ll \beta\ll \beta_2\ll \beta'_2 \ll \beta_1 \ll \beta'_1 \ll \rho \ll \tau \ll \mu\ll \xi \ll \gamma \ll 1$.

If $n$ is odd, then $H$ contains no Hamilton $2$-cycle, since every
$2$-cycle in a $4$-graph has an even number of vertices. In this case
the algorithm outputs this divisibility obstruction as a certificate.
Thus we may assume from now on that $n$ is even.

\textbf{Step 1. Refine the extremal bipartition.}
Given a $(c',\g/3)$-even-extremal bipartition $\{A',B'\}$ of $V$.
%We first clean this extremal partition. 
Define $A_{\mathrm{bad}}$ as the set of vertices in $A'$ which are $(\beta'_1,\beta'_2)$-bad with respect to $\{A',B'\}$ and $B_{\mathrm{bad}}$ as the set of vertices in $B'$ which are $(\beta'_1,\beta'_2)$-bad with respect to $\{A',B'\}$. 
Put
$A:=(A'\setminus A_{\mathrm{bad}})\cup B_{\mathrm{bad}}$ and
$B:=(B'\setminus B_{\mathrm{bad}})\cup A_{\mathrm{bad}}$.
% By Proposition~\ref{prop:types}, applied with $(c',\beta'_1,\beta'_2,\g/3)$ in
% place of $(c,\beta_1,\beta_2,\g/4)$, the number of moved vertices is at most
% $288c'n$.
By Lemma \ref{lem:cleaned-even-partition}, the resulting bipartition $\{A,B\}$ is $(c,\g/4)$-even-extremal and with respect to
$\{A,B\}$ every vertex of $H$ is  $(\beta_1,\beta_2)$-medium.
Note that $(\beta'_1,\beta'_2)$-badness of a vertex $v$ can be checked by checking the edges containing $v$, and thus the sets $A_{\mathrm{bad}}, B_{\mathrm{bad}}$ and the new bipartition $\{A, B\}$ can be constructed in time $O(n^4)$.
As the even-goodness of $\{A,B\}$ can be checked in time $O(n^8)$ and even-goodness is a necessary condition for $2$-Hamiltonicity, to complete the proof of the theorem it suffices to assume that $\{A,B\}$ is even-good, and show that $H$ contains a Hamilton 2-cycle.

% We now test whether the cleaned bipartition $\{A,B\}$ is even-good.
% This can be done in time $O(n^8)$ by checking the five conditions
% (\ref{E0})--(\ref{E4}) in Definition~\ref{def:evenoddgood}.

% If $\{A,B\}$ is not even-good, then the algorithm outputs the bipartition $\{A,B\}$ as a certificate that $H$ contains no Hamilton $2$-cycle.
% Indeed, by Lemma~\ref{lem:hcevenoddgood}, every bipartition of the vertex set of a
% Hamiltonian $4$-graph must be even-good. 
% Thus a bipartition which is not
% even-good certifies non-existence of a Hamilton $2$-cycle.
% It remains to consider the case that $\{A,B\}$ is even-good. 

\textbf{Step 2. Construct a bridge.}
Now, we apply Lemma~\ref{lem:bridge} to $\{A,B\}$ and obtain either
\begin{itemize}
    \item a path $Q$ with $\theta$-good connate ends such that $|A\setminus V(Q)|,\ |B\setminus V(Q)|$ are even, and $X_E\subseteq V(Q)$. Moreover, $|V(Q)|\le \mu n$; or
\item a path $P$
with  $\theta$-good split ends such that
$|A\setminus V(P)|=|B\setminus V(P)|$, and $X_E\subseteq V(P)$.
Moreover, $|V(P)|\le \mu n$.
\end{itemize}

\textbf{Step 3. Cover the remaining vertices.}
Finally we use Lemmas \ref{lem:connateHP} and \ref{lem:splitHP} to cover the remaining vertices. 
Recall that $\mathcal P_{at}\subseteq \binom A2\cup \binom B2$ is the family of pairs that are either $A$-atypical or $B$-atypical.

(1) By Lemma~\ref{lem:bridge}, we obtain a path $Q_0$ with $\theta$-good connate ends $c,c'$ such that $|A\setminus V(Q_0)|,\ |B\setminus V(Q_0)|$ are even, $X_E\subseteq V(Q_0)$, and $|V(Q_0)|\le \mu n$.
Next we choose $\theta$-good connate pairs $q_A\in \binom{A\setminus V(Q_0)}{2}$ and $q_B\in \binom{B\setminus V(Q_0)}{2}$ such that $q_A,q_B\notin \mathcal P_{at}$.
This can be done since $|\mathcal P_{at}|\le\eta n^2$ by Claim \ref{claim:few-exceptional-vertices}.
By Proposition \ref{prop:pairconnect}, there exist connate pairs $h,h'$ such that $h\cup q_A,h\cup c\in E(H)$ and $c'\cup h',h'\cup q_B\in E(H)$ avoiding used vertices.

Set $Q=q_AhcQ_0c'h'q_B$ and $U_A=(A\setminus V(Q))\cup q_A$. 
Note that $|U_A|$ is even.
It is easy to see that $|U_A|\ge|A|/2$ since $|V(Q)|\le2\mu n$.
Define $\mathcal T_A:=\{T\in \binom{A}{3}: \deg_{H[A]}(T)<(1/2+2\xi)|A|\}$ and 
$\mathcal T_{U_A}=\{T\in\binom{U_A}{3}:\deg_{H[U_A]}(T)<(1/2+\xi)|U_A|\}$.
We claim that $\mathcal T_{U_A}\subseteq \mathcal T_A$. 
Indeed, since $|V(Q)|\le 2\mu n$ and
$\mu\ll\xi$, for every $T\in\mathcal T_{U_A}$, we have
$\deg_{H[A]}(T)\le \deg_{H[U_A]}(T)+2\mu n<(1/2+\xi)|U_A|+2\mu n<(1/2+2\xi)|A|$,
so $T\in\mathcal T_A$.

Since $X_E\subseteq V(Q_0)$ and $q_A\cap V(Q_0)=\emptyset$, we obtain that every vertex of $U_A$ is not $A$-exceptional and thus is contained in at most $\frac{\rho}{4}|A|^2\le\rho|U_A|^2$ triples of $\mathcal T_A$.
Thus, every vertex of $U_A$ is contained in at most $\rho|U_A|^2$ triples of $\mathcal T_{U_A}$.
Let $U_B=(B\setminus V(Q))\cup q_B$, for the vertex in $U_B$, we can also draw the same conclusion and thus (\ref{C1}) in Lemma \ref{lem:connateHP} holds.

%Set $U_B:=(B\setminus V(Q))\cup q_B$.
Also, $q_A\notin \mathcal P_{at}$, thus it is contained in at most $\frac{\tau}{4}|A|$ triples of $\mathcal T_A$ and at most $\tau|U_A|$ triples of $\mathcal T_{U_A}$.
For $q_B$, it is the same. 
We now choose an even edge $p_Ap_B$, where $p_A\in\binom{A\setminus V(Q)}2, p_B\in\binom{B\setminus V(Q)}2$ and $p_A,p_B\notin \mathcal P_{at}$. 
The number of even edges $e$ such that $|e\cap A|=2$ is at least $\frac12\left(\binom{|A|}{2}|B|(\frac13+\g)n-3cn^4\right)\ge\frac{n^4}{400}$.
Since $|\mathcal P_{at}|\le\eta n^2$ and the number of edges containing $V(Q)$ is at most $2\mu n^4$, we can choose an even edge $p_Ap_B$ such that $|p_A\cap A|=2, |p_B\cap B|=2$  and $p_A,p_B\notin \mathcal P_{at}$. 
Since $p_A\notin \mathcal P_{at}$, it is contained in at most $\frac{\tau}{4}|A|$
triples of $\mathcal T_A$, and hence in at most $\tau |U_A|$ triples of
$\mathcal T_{U_A}$; similar assumption holds for $p_B$.
Thus, (\ref{C2}) in Lemma \ref{lem:connateHP} holds.

Using Lemma~\ref{lem:connateHP} separately on $H[(A\setminus V(Q))\cup q_A]$ and on $H[(B\setminus V(Q))\cup q_B]$, we obtain Hamilton $2$-paths
$P_A$ on $H[(A\setminus V(Q))\cup q_A]$ with ends $q_A,p_A$, and
$P_B$ on $H[(B\setminus V(Q))\cup q_B]$ with ends $q_B,p_B$.
Then $q_AQq_BP_Bp_Bp_AP_Aq_A$ gives rise to a Hamilton
$2$-cycle of $H$. 
%Indeed, $Q$ ends at $q_B$, $P_B$ starts at $q_B$ and ends at $p_B$, $F$ goes from $p_B$ to $p_A$, and $P_A$ goes from $p_A$ back to $q_A$, the initial end of $Q$. The vertex set covered is $V(Q)\cup(A\setminus V(Q))\cup(B\setminus V(Q))=V$. Thus, we obtain a Hamilton 2-cycle.

(2) By Lemma~\ref{lem:bridge}, we obtain a path $P_0$
with $\theta$-good split ends $s,s'$ such that
$|A\setminus V(P_0)|=|B\setminus V(P_0)|$, $X_E\subseteq V(P_0)$, and $|V(P_0)|\le\mu n$.
%Orient $P_0$ from $s$ to $s'$.
Next we choose two disjoint $\theta$-good split pairs $s_1,s_2$ outside $V(P_0)$ such that $s_1,s_2\notin\mathcal E_{\mathrm{sp}}$, which can be done as  $|\mathcal E_{\mathrm{sp}}|\le 60cn^2$ by Claim \ref{claim:few-exceptional-vertices}.
Note that $(s_1\cup s_2)\cap X_E=\emptyset$ since $X_E\subseteq V(P_0)$.
By Proposition \ref{prop:pairconnect}, there exist split pairs $f,f'$ such that $s_1\cup f,f\cup s\in E(H)$ and $s'\cup f',f'\cup s_2\in E(H)$.

Set $P=s_1fsP_0s'f's_2$ and
$U:=(V\setminus V(P))\cup s_1\cup s_2$. 
Let
$A_U:=A\cap U$ and $B_U:=B\cap U$. 
It is easy to see that
$|A_U|=|B_U|=:m$. %Write $m:=|A_U|=|B_U|$. 
Since $|V(P)|\le 2\mu n$ and
$|A|,|B|\ge n/3+\g n/4$, we have $m\ge n/4$ for sufficiently large $n$.

Let $\mathcal E_{\mathrm{sp}}:=\{s\in A B:|N_H(s)\cap A B|<(1/2+2\xi)|A||B|\}$ and $\mathcal E_U:=\{s\in A_U B_U:|N_H(s)\cap A_U B_U|<(1/2+\xi)m^2\}$.
We claim that $\mathcal E_U\subseteq \mathcal E_{\mathrm{sp}}$. 
Indeed, if $s\in\mathcal E_U$, then
$|N_H(s)\cap A B|
\le (1/2+\xi)m^2+2\mu n^2
< (1/2+2\xi)|A||B|$,
since $m\ge n/4$, $\mu\ll\xi$ and $|A||B|=m^2+O(\mu n^2)$. Hence
$s\in\mathcal E_{\mathrm{sp}}$.

Since $X_E\subseteq V(P_0)$ and $(s_1\cup s_2)\cap V(P_0)=\emptyset$, every $v\in U$ is contained in at most $\rho n/4$ pairs of $\mathcal E_{\mathrm{sp}}$ and thus incident with at most $\rho n/4$ pairs of
$\mathcal E_U$, which is at most $\rho m$ because $m\ge n/4$.
Thus, (\ref{S1}) in Lemma \ref{lem:splitHP} holds.
%Let $X_{\mathrm{sp}}$ be the set of split-exceptional vertices. By the construction of the split bridge, the end pairs $s_1,s_2$ are disjoint from $X_{\mathrm{sp}}$, and all vertices of $X_{\mathrm{sp}}$ are contained in $D=V(P)\setminus(s_1\cup s_2)$. Therefore no vertex of $U$ is split-exceptional. Since $\mathcal E_U\subseteq\mathcal E_{\mathrm{sp}}$,

Also, $s_1,s_2\notin \mathcal E_{\mathrm{sp}}$ yields
$s_1,s_2\notin\mathcal E_U$ and thus (\ref{S2}) in Lemma \ref{lem:splitHP} holds. 
Applying Lemma~\ref{lem:splitHP} with
ends $s_2,s_1$, we obtain a Hamilton $2$-path $P_1$ in $H[U]$ from
$s_2$ to $s_1$. 
%\red{whose pair sequence consists only of split pairs (JH. Look redundant to me?)}.
Hence $s_1Ps_2P_1s_1$ is a
Hamilton $2$-cycle of $H$.
\end{proof}

In the following subsections, we just need to consider the cleaned bipartition and we may assume that $\{A,B\}$ is $(c,\g/4)$-even-extremal and every vertex of $\{A,B\}$ is $(\beta_1,\beta_2)$-medium.

\subsection{Proof of Lemma \ref{lem:cleaned-even-partition}}

Now we prove the refining lemma by moving all bad vertices to the opposite side of the bipartition. 
Since the number of bad vertices is very small, the new bipartition remains extremal. 
Moreover, every vertex is contained in enough medium pairs.

\begin{proof}
Let $L:=A_{\mathrm{bad}}\cup B_{\mathrm{bad}}$. 
By Proposition~\ref{prop:types} (c), applied with $c',\beta'_1,\beta'_2$ in
place of $c,\beta_1,\beta_2$, we have $|L|\le 288c'n$.
Since at most $288c'n$ vertices are moved, we have $|A|,|B|\ge(1/3+\g/3)n-288c'n\ge(1/3+\g/4)n$
and the number of odd edges with respect to
$\{A,B\}$ is at most
$c' n^4+|L|(n-1)^{3}\le c n^4$, as $c'\ll c$. 
Thus $\{A,B\}$ is $(c,\g/4)$-even-extremal.

Since $c'\ll\beta_2\ll\beta_2'\ll\beta_1\ll\beta_1'$ and $\beta_1'n-288c'n\ge\beta_1n$, every unmoved vertex was $(\beta_1',\beta_2')$-medium under $\{A',B'\}$ and is $(\beta_1,\beta_2)$-medium under $\{A,B\}$.
Now let $v\in L$. 
Without loss of generality, assume that $v\in A_{\mathrm{bad}}$, so $v$ is
moved from $A'$ to $B$.

Note that it suffices to show that $v$ is contained in $\b_1 n^3$ edges $e$ of $H$ with $|e\cap A'|=3$ and $|e\cap B'|=1$.
Indeed, after moving the vertices in $L$, at least $\b_1n^3 - 288c' n^3\ge 0.9\b_1 n^3$ such edges become even edges with two vertices in $A$ and two vertices in $B$.
Thus the number of $\b_2$-medium split pairs containing $v$ is at least $(0.9\b_1 n^3 - n\cdot \b_2 n^2)/\binom n2 > \b_1 n$, and similarly the number of $\b_2$-medium connate pairs containing $v$ is at least $(0.9\b_1 n^3 - n\cdot \b_2 n^2)/\binom n2 > \b_1 n$, as desired.

Since $v$ is $(\b_1',\b_2')$-bad with respect to $\{A', B'\}$, it is contained in at most $\beta_1'n$ $\beta_2'$-medium split (or connate) pairs under $\{A',B'\}$.
Thus, $v$ was in at least $(1/3+\g/3-\beta_1')n\ge n/3$ $\beta_2'$-bad split (or connate) pairs $p$ under $\{A',B'\}$, where every such $p$ is contained in at most $\beta_2'\binom n2$ even edges.
By the minimum codegree condition, we infer that $p$ is in at least $\frac12 \frac n3\delta_3(H)\ge \frac{n^2}{18}$ edges of $H$ with two or three vertices in $A'$ -- indeed, if $p$ is split then we can consider all triples $p\cup \{u\}$ with $u\in A'$, or $p$ is connate and we consider all triples $p\cup \{u\}$ with $u\in B'$.
Therefore, at least $\frac{n^2}{20}$ such edges contains exactly three vertices of $A'$.
Since there are $n/3$ choices for such pair $p$, we obtain at least $\frac13\frac n3\frac{n^2}{20} = \frac{n^3}{180}$ edges containing $v$ and two other vertices of $A'$ and we are done.

The argument
for vertices moved from $B'$ to $A$ is identical with the roles of $A'$ and
$B'$ interchanged.
Hence every vertex of $H$ is $(\beta_1,\beta_2)$-medium with
respect to $\{A,B\}$. 
\end{proof}

\subsection{Proof of Lemma \ref{lem:bridge}}

We prove the bridge lemma.  The construction is split into several
local ingredients.  
First we show that medium pairs can be upgraded to good pairs inside suitable edges, and then we show that we can cover the exceptional vertices by a short path.  
Finally we use the even-goodness to adjust the residue of the bridge so that the remaining vertex set is suitable for the final Hamilton path lemma.

\begin{proposition}\label{prop:medium-to-good}
Assume Setup~\ref{setup:evenextr}, and suppose that $\delta_3(H)\geq n/3+\g n$ and that $\{A, B\}$ is a $(c,\g/4)$-even-extremal bipartition of $V$.
If $R\subseteq V$ satisfies $|R| \leq \frac{1}{3} \beta n$, then for every $\beta$-medium pair $p_1$ there exists a $\theta$-good pair $p_2 \in \binom{V \setminus R}{2}$ such that $p_1\cup p_2$ is an even edge.
\end{proposition}
\begin{proof}
Let $p_1$ be a $\beta$-medium pair. Then there are at
least $\beta {n\choose 2}$ pairs $p_2$ such that $p_1\cup p_2$ is an even edge.
By Proposition \ref{prop:types}, the number of non $\theta$-good pairs is at most
$\eta n^2 < \frac{\beta}{4}{n\choose 2}$.
%, which is smaller than
%$\frac{\beta}{4}{n\choose 2}$ by the hierarchy of constants. 
Moreover, the number of pairs meeting $R$ is at most $|R|n\le \beta n^2/3<\frac{3\beta}{4}{n\choose 2}$ for sufficiently large $n$. 
%Hence not all even completions of $p_1$ are either non-$\theta$-good or meet $R$. 
Therefore there exists a $\theta$-good pair $p_2\in {V\setminus R\choose 2}$ such that $p_1\cup p_2$ is
an even edge. 
\end{proof}

The following result allows us to find a short path covering a small given set of vertices.
\begin{proposition}\label{prop:coverex}
Assume Setup \ref{setup:evenextr}, and suppose that $\delta_3(H)\geq n/3+\g n$ and that $\{A, B\}$ is a $(c,\g/4)$-even-extremal bipartition of $V$.
Given a set $X\subseteq V$ of $(\beta_1,\beta_2)$-medium vertices with $|X|\le\eta n$
we can obtain a path $Q$ with $|V(Q)|\le8\eta n$ that consists of even edges and contains all vertices of $X$. %but also obtain a path $Q$ which contains all vertices of $X$ and is a sequence of connate pairs, whose ends are $\theta$-good connate pairs.
Furthermore, we can choose the path $Q$ so that the two ends of $Q$ are either both $\theta$-good connate pairs, or both $\theta$-good split pairs. 
\end{proposition}

\begin{proof}
 Let $X=\{x_1,\ldots,x_t\}$.
Now suppose we need to choose $Q$ with split ends, and the connate case can be dealt with verbatim.
Since every vertex is $(\beta_1,\beta_2)$-medium, each vertex is contained in at least $\beta_1n$ $\beta_2$-medium connate pairs and at least $\beta_1n$ $\beta_2$-medium split pairs.
Thus, for each $x_i$, choose greedily a vertex $y_i$ such that $r_i:=\{x_i,y_i\}$ is a
$\beta_2$-medium split pair, and all pairs $r_i$ are mutually disjoint and avoid all
previously used vertices. 
This is possible since $t\le \eta n$ and $\eta\ll\beta_1$.
Using Proposition~\ref{prop:medium-to-good} twice, for each $i\in[t]$ we may choose disjoint $\theta$-good split pairs
$a_i,b_i$, avoiding all vertices already used, such that
$a_i\cup r_i\in E(H)$ and $r_i\cup b_i\in E(H)$ since $\eta\ll\beta$.
Thus $a_ir_ib_i$ is a short split segment containing $x_i$ for each $i\in[t]$.

% By Proposition~\ref{prop:pairconnect}, we obtain that if $u,v$ are
% two disjoint $\theta$-good connate (or split) pairs and $W$ is a set of at most
% $\mu n$ forbidden vertices disjoint from $u\cup v$, then there exists a
% connate (split) pair $h\subseteq V\setminus W$ such that
% $u\cup h\in E(H)$ and $h\cup v\in E(H)$. 
Now by Proposition~\ref{prop:pairconnect}, we can connect all short segments $a_ir_ib_i, i\in[t]$ greedily, each time with $R$ as the set of vertices already chosen. 
For each $i\in[t-1]$, we choose a split pair $h_i$ such that $b_i\cup h_i\in E(H)$ and $h_i\cup a_{i+1}\in E(H)$.
%All these pairs are chosen disjointly and avoid all previously used vertices since $\eta\ll\mu$.

Now define $Q:=a_1r_1b_1h_1a_2r_2b_2\cdots
h_{t-1}a_tr_tb_t$, which is a sequence of split pairs.
%%If $t=0$, we omit all segments $a_ir_ib_i$ and take the sequence $q_Ah_0c_AQ_0c_Bh_1q_B$ instead.
By construction, every two consecutive pair-vertices in this sequence
form an even edge of $H$ and the ends are $\theta$-good, which finishes the proof.
\end{proof}

The following result is the key result for us to build a constant-length bridge, overcoming the parity issue.

\begin{proposition}\label{prop:two_odd_edges_connate_bridge}
Assume Setup \ref{setup:evenextr}.
Let $Z\subseteq V$ be a set with $|Z|\le 10$. Suppose that $H_{\mathrm{odd}}-Z$ contains an edge.
Then one can find either
\begin{itemize}
\item a path with one $\theta$-good connate end and one
$\theta$-good split end, using 8 vertices with an even number of vertices from each of
$A,B$, or 
\item a path with two $\theta$-good connate ends, using 10 vertices with an odd number of vertices from each of $A,B$.
\end{itemize}
\end{proposition}
\begin{proof}
    We first claim that if $H_{\mathrm{odd}}-Z$ contains an odd edge, then $H_{\mathrm{odd}}-Z$ contains an odd edge $e=s\cup p$ where $s$ is a split pair and $p$ is a $\beta_2$-medium connate pair.

Indeed, let $\mathcal B$ be the set of all $\beta_2$-bad pairs. 
By Proposition \ref{prop:types} (b), we obtain
$|\mathcal B|\le {36c n^2}$.
Suppose, for a contradiction, that every odd edge contains no
$\beta_2$-medium connate pair. 
Take an odd edge $f_0$. 
By symmetry we may assume that $|f_0\cap A|=3$ and $|f_0\cap B|=1$. Then all three pairs of
$f_0\cap A$ are $\beta_2$-bad. Choose one of them, say $uv\in\binom A2$.

Let $G$ be the graph on $A$ whose edges are precisely the $\beta_2$-bad pairs contained in
$A$. Since $uv$ is $\beta_2$-bad, it is contained in at least $\frac13\binom n2$ odd edges in $H$.
Thus, it is contained in at least $\frac16\binom n2$ odd edges in $H-Z$.
Each such odd edge has the form $uvab$ with
$a\in A$ and $b\in B$. By our assumption, both $ua$ and $va$ are $\beta_2$-bad,
and hence $a\in N_{G}(u)\cap N_{G}(v)$. Since for each fixed $a\in A$ there are at most
$|B|\le 2n/3$ choices for $b$, it follows that
$|N_{G}(u)\cap N_{G}(v)|\ge\frac{\binom{n}{2}}{6|B|}\geq\frac{n}{12}$.
In particular, we have $d_{G}(u)\ge \frac{n}{12}$.

%Set $T:=N_{F_A}(u)\cap N_{F_A}(v)$. Then $|T|\ge\frac{n}{12}$. 
% For every $a\in T$, the pair
% $ua$ is $\beta_2$-bad. Repeating the preceding argument with $ua$ in place of $uv$, we obtain
% $d_{F_A}(a)\ge \frac{n}{12}$ for every $a\in T$. 
The same conclusion can be obtained for each edge $uv$ of $G$.
Thus, for any $a\in A$, if $d_{G}(a)>0$, then $d_{G}(a)\ge \frac{n}{12}$.
As $e(G)>0$, we have 
$e(G)\ge \frac12 (\frac{n}{12})^2\ge \frac{n^2}{288}$. 
This contradicts $e(G)\le |\mathcal B|\le {36c n^2}$,
provided $c\ll 1$. The case $|f_0\cap A|=1$ is symmetric, using the
$\beta_2$-bad pair graph in $B$.

Now we can choose an odd edge $e=s\cup p$ where $s$ is a split pair and $p$ is a $\beta_2$-medium connate pair.
First suppose that $s$ is also $\beta_2$-medium. Since $\beta\ll\beta_2$, both $p$ and $s$ are
$\beta$-medium. By Proposition \ref{prop:medium-to-good}, avoiding $Z$, we may
choose a $\theta$-good connate pair $p^+$ avoiding $Z\cup e$ such that $p^+\cup p\in E(H)$ and a $\theta$-good
split pair $s^+$ avoiding $Z\cup e\cup p^+$ such that $s\cup s^+\in E(H)$. Here $p^+$ is connate because $p\cup p^+$ is
even, and $s^+$ is split because $s\cup s^+$ is even. Then $p^+pss^+$ is a path
with a $\theta$-good connate end and a $\theta$-good split end, using even number of vertices from $A$ and from $B$.

Finally suppose that $s$ is $\beta_2$-bad. Since $s$ is split and $\beta_2$-bad, it has at least
$\frac13\binom n2$ odd completions, which are connate pairs.
Since $|\mathcal B|\le {36c n^2}$, and the
number of connate pairs meeting $Z\cup e$ is only $O(n)$, we may choose a $\beta_2$-medium connate pair $p'$ disjoint from
$Z\cup e$ such that $s\cup p'\in E(H)$. Applying Proposition \ref{prop:medium-to-good} to $p$ and $p'$, again
avoiding chosen vertices, we obtain disjoint $\theta$-good connate pairs
$p^+$ and $p^*$ such that $p^+\cup p\in E(H)$ and $p'\cup p^*\in E(H)$. Then
$p^+psp'p^*$ is a path whose two ends are $\theta$-good connate pairs using odd number of vertices from $A$ and from $B$.
%
%Moreover, in the first outcome, namely the path $p^+pss^+$, the path uses an even number of vertices from $A$ and an even number of vertices from $B$. %Indeed, the only split pairs on this path are $s$ and $s^+$, and each split pair contributes one vertex to each of $A$ and $B$, while every connate pair contributes an even number of vertices to each side.
%In the second outcome, namely the path $p^+psp'p^{\prime +}$, the path uses an odd number of vertices from $A$ and an odd number of vertices from $B$. 
%Indeed, this path contains exactly one split pair, namely $s$, and all other pairs are connate. Hence the split pair contributes one vertex to each side, while the connate pairs contribute even numbers to each side.
\end{proof}

After the exceptional vertices have been covered, we still need to correct
the global residue.  
The following lemma extracts the required small bridge from the parity configurations appearing in the definition of even-goodness.

\begin{lemma}\label{lem:bridgeconstruct} 
Assume Setup~\ref{setup:evenextr}, and let $\{A,B\}$ be a $(c,\g/4)$-even-extremal bipartition of $V$. 
Suppose that $\{A,B\}$ is even-good and every vertex is $(\beta_1,\beta_2)$-medium. 
Then we obtain either

\begin{enumerate}
    \item[(i)] a connate bridge $Q$ with two $\theta$-good connate ends such that both $|A\setminus V(Q)|$ and
    $|B\setminus V(Q)|$ are even; or

    \item[(ii)] a split bridge $P$ with two $\theta$-good split ends such that
    $|A\setminus V(P)|=|B\setminus V(P)|$.
\end{enumerate}

Moreover, the bridge has at most 22 vertices.
%and can be constructed avoiding a given set of at most $\mu n$ vertices.
\end{lemma}

\begin{proof}
%Let $Z\subseteq V$ be a set of already used vertices with $|A\cap Z|,|B\cap Z|\le \mu n$.
% Given two $\theta$-good split pairs $s_1,s_2$, we can find a split-pair path
% $s_1t_1t_2s_2$.  
% Indeed, by Proposition \ref{prop:pairconnect}(3), choose a
% $\theta$-good split pair $t_1\in V\setminus Z$ with $s_1\cup t_1\in E(H)$.  Then apply
% Proposition \ref{prop:pairconnect} to $t_1$ and $s_2$ to find a split pair
% $t_2$ such that $t_1\cup t_2\in E(H)$ and $t_2\cup s_2\in E(H)$.
% Thus between two good split pairs we may choose a split-pair connector
% using either three split pairs or four split pairs.  This allows us to
% control parity.
We now separate the cases according to what structure the even-goodness gives us.

\textbf{Case (\ref{E0}).}  
In this case there is no parity issue, and we just include an even path of length 2 with good ends.
Suppose first that $|A|$ is even. Since $n$ is even,
$|B|$ is also even.  
Choose a $\theta$-good connate pair $q$, and by Proposition~\ref{prop:medium-to-good}, we find a $\theta$-good connate pair $q'\in N_H(q)$. 
% This can be done since, by
% Proposition \ref{prop:types} and Claim \ref{claim:few-exceptional-vertices}, only $\eta n^2$ pairs are not $\theta$-good, while $|A|,|B|\geq n/3+\g n/4$.
Let $Q=qq'$.
Hence both $|A\setminus V(Q)|$ and
$|B\setminus V(Q)|$ are even.  
%\red{BW: Why do we use Proposition~\ref{prop:medium-to-good} twice? $qq'$ is enough. You are right.}

Now suppose that $|A|=|B|$.
Choose a $\theta$-good split pair $s$ and by Proposition \ref{prop:medium-to-good}, we find a $\theta$-good split pair $s'\in N_H(s)$.
Let $P=ss'$.
Therefore
$|A\setminus V(P)|=|B\setminus V(P)|$.  

\textbf{Case (\ref{E1}).} Suppose that $H_{\rm odd}$ contains two
vertex-disjoint odd edges $e_1$ and $e_2$.
We may assume that Case (\ref{E0}) does not hold.
Thus $|A|$ is odd, and since $n$ is even, $|B|$ is odd as well.

Apply Proposition \ref{prop:two_odd_edges_connate_bridge} to $e_1$ and $e_2$ in the following manner.
First, apply the proposition to $e_1$ with $Z=e_2$ and denote the path we obtain by $G_1$; then apply it to $e_2$ with $Z=V(G_1)$ and denote the path we obtain by $G_2$.
%avoiding $e_2$, we obtain a path $G_1$ disjoint from $e_2$ with at most 10 vertices.
Therefore, both $G_1$ and $G_2$ have $\theta$-good ends and at most 10 vertices.
If $G_i$ has two $\theta$-good connate ends, then we set $Q:=G_i$ and discard the other path, and note that both $|A\cap V(Q)|$ and $|B\cap V(Q)|$ are odd.

So we may assume that for $i=1,2$, $G_i$ has a $\theta$-good connate end, denoted by $c_i$ and a $\theta$-good split end, denoted by $s_i$. 
%Orient $G_1$ from its connate end to its split end, and denote these ends by $c_1$ and $s_1$, respectively. 
Also, $G_i$, $i=1,2$ uses an even number
of vertices from $A$ and an even number of vertices from $B$.
As $s_1$ and $s_2$ are two disjoint $\theta$-good split pairs, by Proposition~\ref{prop:pairconnect},
avoiding $V(G_1)\cup V(G_2)$, there exists a split pair $h\in N_H(s_1)\cap N_H(s_2)$.
Set $Q:=G_1s_1hs_2G_2$. 
Then $Q$ is a path with 22 vertices whose two ends $c_1$ and $c_2$ are $\theta$-good connate pairs.
Furthermore,  $|A\cap V(Q)|$ and $|B\cap V(Q)|$ are both odd.

Thus, in all cases, we obtain a path $Q$ with at most 22 vertices whose two ends are
$\theta$-good connate pairs, and such that $|A\cap V(Q)|$ and $|B\cap V(Q)|$ are both odd. 
Since $|A|$ and $|B|$ are both odd, both
$|A\setminus V(Q)|$ and $|B\setminus V(Q)|$ are even, as desired.

\textbf{Case (\ref{E2}).}  We may assume that Cases (\ref{E0}) and (\ref{E1}) do not hold. Thus $|A|$ and $|B|$ are both
odd, and $H_{\mathrm{odd}}$ contains no two vertex-disjoint edges. 
By Even-goodness, suppose
$H_{\mathrm{odd}}$ contains two odd edges
$e_1=q_A\cup s$ and $e_2=s\cup q_B$,
where $s$ is a split pair and $q_A,q_B$ are connate pairs.

We note that $q_A$ and $q_B$ are $\theta$-good since otherwise there will be two vertex-disjoint odd edges.
Define
$Q:=q_Asq_B$.
Then $Q$ is a path with 6 vertices whose ends are $\theta$-good connate pairs. 
As $Q$ contains one split pair and two connate pairs, $|A\cap V(Q)|$ and
$|B\cap V(Q)|$ are both odd. Since $|A|$ and $|B|$ are both odd, it follows that
$|A\setminus V(Q)|$ and $|B\setminus V(Q)|$ are both even. 

\textbf{Case (\ref{E3}).}  Suppose that $|A|=|B|+2$ and that $H_{\rm odd}$
contains two odd edges whose intersection is a connate pair in $A$.
Write these edges as $e_1=s_1\cup p$ and $e_2=p\cup s_2$, where
$p\in\binom A2$ and $s_1,s_2$ are split pairs.

Since Case (\ref{E1}) does not hold, both $s_1$ and $s_2$ must be
$\theta$-good since otherwise we would have two disjoint odd edges.
The path $P:=s_1ps_2$ has split ends.  It uses four vertices from $A$
and two vertices from $B$.  Since $|A|=|B|+2$, we have
$|A\setminus V(P)|=|B\setminus V(P)|$.  Hence $P$ is a split bridge.

\textbf{Case (\ref{E4}).}  This case is symmetric to Case (\ref{E3}).

Now the proof is completed.
\end{proof}

%Next, our goal is to extend the bridges to cover all exceptional vertices with the same parity as bridges in Lemma \ref{lem:bridgeconstruct}.
Now we are ready to prove Lemma~\ref{lem:bridge}.

\begin{proof}[Proof of Lemma~\ref{lem:bridge}]
Recall that $X_E$ is the set of exceptional vertices and $|X_E|\le \eta n$ by Claim~\ref{claim:few-exceptional-vertices}.
%, we have $|X_E|\le \eta n$.
We apply Lemma~\ref{lem:bridgeconstruct} to $H$ and separate the proof according to outcome of the lemma.

(1) Suppose Lemma~\ref{lem:bridgeconstruct} returns a connate bridge $Q_0$ with two $\theta$-good connate ends $c_1$ and $c_2$ such that both $|A\setminus V(Q_0)|$ and $|B\setminus V(Q_0)|$ are even.
By Proposition \ref{prop:coverex}, we obtain a path $Q_1$ containing all vertices of $X_E\setminus V(Q_0)$ whose ends $c,c'$ are $\theta$-good connate pairs.
Furthermore, $Q_1$ is a sequence of connate pairs.
Finally, we connect $\theta$-good connate pairs $c_2$ and $c$ by $c_2hc$, and such connate pair $h$ exists by Proposition~\ref{prop:pairconnect}.

Now let $Q:=c_1Q_0c_2hcQ_1c'$; if $X_E\setminus V(Q_0)=\emptyset$, then take
$Q:=Q_0$ instead.
Thus, $X_E\subseteq V(Q)$ and both ends of $Q$ are $\theta$-good connate pairs.
Since $|A\setminus V(Q_0)|$ and $|B\setminus V(Q_0)|$ are even, 
we have $|A\setminus V(Q)|$ and $|B\setminus V(Q)|$ are even.
Finally, 
$|V(Q)|\le |V(Q_0)|+8\eta n+2\le \mu n$ as $\eta\ll\mu$.

(2) Suppose Lemma~\ref{lem:bridgeconstruct} returns a split bridge $P_0$ with good
split ends $s_1,s_2$. 
By Proposition \ref{prop:coverex}, we obtain a path $P_1$ containing all vertices of $X_E\setminus V(P_0)$ whose ends $s,s'$ are $\theta$-good split pairs.
Furthermore, $P_1$ is a sequence of split pairs.
Finally, we connect $\theta$-good split pairs $s_2$ and $s$ by $s_2s_0s$, and such split pair $s_0$ exists by Proposition \ref{prop:pairconnect}. 

Similarly, let $P=s_1P_0s_2s_0sP_1s'$; if $X_E\setminus V(P_0)=\emptyset$, then take $P:=P_0$.
Thus, $X_E\subseteq V(P)$ and both ends of $P$ are $\theta$-good split pairs. 
Since $|A\setminus V(P_0)|=|B\setminus V(P_0)|$, we have $|A\setminus V(P)|=|B\setminus V(P)|$.
Finally, 
$|V(P)|\le |V(P_0)|+8\eta n+2\le \mu n$ as $\eta\ll\mu$.
\end{proof}

\subsection{Proofs of Lemmas \ref{lem:connateHP} and \ref{lem:splitHP}}
After the exceptional vertices have been removed, the final covering step is essentially a Hamilton path problem in a dense auxiliary graph.
We first need the following lemma.

\begin{lemma}\label{lem:hpframework}
Let $0<1/n\ll \lambda\ll \gamma\ll 1$. 
Let $R$ be either the
complete graph on a set $U$ of $n\in 2\mathbb N$ vertices, or the complete balanced bipartite graph with vertex classes $A_U,B_U$ and $|A_U|=|B_U|=n/2$.
Let $F=(U,E)$ be a spanning subgraph of $R$ such that
$d_F(v)\ge d_R(v)-\lambda n$ for every $v\in U$. Let $\Gamma$ be a graph
with vertex set $E$ such that $\delta(\Gamma)\ge (1/2+\gamma)|E|$.

Then for any two disjoint edges $p,q\in E$, there exists a perfect matching
$M\subseteq E$ of $F$ containing $p$ and $q$ such that, for every $e\in E$,
$|N_\Gamma(e)\cap M|\ge (1/2+\gamma/6)|M|$.
\end{lemma}
\begin{proof}
Choose uniformly at random a perfect matching $M_0$ of $R_0:=R-V(p\cup q)$, and set $m:=|M_0|=(n-4)/2$. 
Fix $e\in E$, and let $X_e:=|N_\Gamma(e)\cap M_0|$. 
It is easy to see that for every edge $r\in E(R_0)$, 
\[
\mathbb P[r\in M_0]
=
\frac{m}{|E(R_0)|}.
\]
Hence
\[
\mathbb E X_e
=
\sum_{r\in N_\Gamma(e)\cap E(R_0)}\mathbb P[r\in M_0]
=
\frac{m}{|E(R_0)|}|N_\Gamma(e)\cap E(R_0)|.
\]
Since deleting $V(p\cup q)$ removes at most $4n$ edges from $N_\Gamma(e)$, we have
$
|N_\Gamma(e)\cap E(R_0)|
\ge\left(\frac12+\gamma\right)|E|-4n.
$
Moreover, from $d_F(v)\ge d_R(v)-\lambda n$ for every $v\in U$, we get $|E|\ge |E(R)|-\lambda n^2/2$, while $|E(R_0)|\le|E(R)|$. 
Therefore, since $\lambda\ll\gamma$ and $n$ is sufficiently large,
$
\frac{|E|}{|E(R_0)|}\ge 1-\gamma/10.
$
Consequently,
\[
\mathbb E X_e
\ge
\frac{m}{|E(R_0)|}\left(\left(\frac12+\gamma\right)|E|-4n\right)
\ge
\left(\frac12+\frac{2\gamma}{3}\right)m.
\]

By the standard Chernoff bound for random perfect matchings, we have
\[
\mathbb P\left[
X_e<
\left(\frac12+\frac{\g}{2}\right)m
\right]
\le
\exp(-\Omega_\g(n)).
\]
Taking a union bound over all $O(n^2)$ choices of $e\in E(F)$, with positive probability we have $
X_e\ge
\left(\frac12+\frac{\g}{2}\right)m$
for every $e\in E(F)$.

Since $d_F(v)\ge d_R(v)-\lambda n$ for every $v\in U$, the number of edges
of $R\setminus F$ is at most $\lambda n^2/2$.
Let $Y:=|M_0\cap (E(R)\setminus E)|$. 
%Since $d_F(v)\ge d_R(v)-\lambda n$ for every $v\in U$, we have
Recall that
$|E(R)\setminus E|\le \lambda n^2/2$. Hence
\[\mathbb E Y\le |E(R)\setminus E|\cdot \frac{m}{|E(R_0)|}\le \frac{\lambda n^2}{2}\cdot \frac{m}{|E(R_0)|}\le2\lambda  n.\]
By Markov's inequality, $\mathbb P[Y>\sqrt{\lambda}n]\le \mathbb E[ Y]/(\sqrt{\lambda}n)\le 2\sqrt{\lambda}$. 
Hence, with probability at least $1-2\sqrt{\lambda} > 1/2$, $Y\le \sqrt{\lambda}n$, that is, $M_0$ contains at most $\sqrt{\lambda}n$ non-edges of $F$.

Fix a matching $M_0$ satisfying both properties above.  
We now repair the non-edges of $F$ in $M_0$ one by one. Suppose that
$ab\in M_0\setminus E$. Since $M_0$ is a matching of $R_0$, the edge $ab$
is disjoint from $p\cup q$. We distinguish the two possible choices of
$R$.

If $R$ is the complete graph on $U$ where $|U|=n$, then every vertex has at most
$\lambda n$ non-neighbours in $F$. Hence all but at most $2\lambda n$
vertices are adjacent in $F$ to both $a$ and $b$. Since the current matching
has at most $\sqrt{\lambda}n$ non-edges of $F$, we can choose an edge
$xy$ of the current matching, disjoint from $\{a,b\}$, such that
$xy\in E$ and all four pairs $ax,ay,bx,by$ lie in $E$. Replace $ab,xy$
by $ax,by$.

If $R$ is the complete balanced bipartite graph with vertex classes $A_U,B_U$ and $|A_U|=|B_U|=n/2$, write $a,x\in A_U$ and
$b,y\in B_U$ for an edge $xy$ of the current matching. Since $a$ has at
most $\lambda n$ non-neighbours in $B_U$ and $b$ has at most $\lambda n$
non-neighbours in $A_U$, we can choose an edge $xy$ of the current matching
with $xy\in E$, disjoint from $\{a,b\}$, such that $ay,xb\in E$. Replace
$ab,xy$ by $ay,xb$.

In both cases the operation produces another matching of $R_0$, decreases
the number of non-edges of $F$, and does not touch the vertices of
$p\cup q$. Repeating this operation at most $\sqrt{\lambda}n$ times, we
obtain a perfect matching $M_1\subseteq E$ of $R_0$.

During this repair process, at most $4\sqrt{\lambda}n$ matching edges are
changed. Therefore, for every $e\in E$,
$|N_\Gamma(e)\cap M_1|
\ge
\left(\frac12+\frac{\gamma}{2}\right)m-4\sqrt{\lambda}n
\ge
\left(\frac12+\frac{\gamma}{3}\right)m$,
since $\lambda\ll\gamma$.
Finally, set $M:=M_1\cup\{p,q\}$.
Then $M$ is a perfect matching of $F$ on $U$, contains $p$ and $q$, and
$|M|=m+2=n/2$. Moreover, for every $e\in E$,
\[
|N_\Gamma(e)\cap M|
\ge
|N_\Gamma(e)\cap M_1|
\ge
\left(\frac12+\frac{\gamma}{3}\right)m\ge
\left(\frac12+\frac\gamma 6\right)(m+2)=\left(\frac12+\frac \gamma 6\right)|M|,
\]
as required.
\end{proof}
\begin{proof}[Proof of Lemma \ref{lem:connateHP}]
Call a pair $r\in\binom U2$ \emph{nice} if $r$ is contained in at most $\tau n$ triples from $\mathcal B:=\mathcal T_U$. 
%Let $R$ be the complete graph on $U$, and 
Let $F$ be a graph defined on $U$ whose edges are precisely the
nice pairs.

We first show that $F$ is dense. Fix a vertex $v\in U$. If $vw\notin E(F)$, then $\{v,w\}$ is contained in more than $\tau n$ triples from
$\mathcal B$. On the other hand, by (\ref{C1}), the vertex $v$ is
contained in at most $\rho n^2$ triples from $\mathcal B$. 
Since each triple from $\mathcal B$ containing $v$ induces two pairs
containing $v$, the number of vertices $w$ such that $vw\notin E(F)$ is at most $2\rho n^2/(\tau n)=2\rho n/\tau$. Put $\lambda:=2\rho/\tau$. Then
$d_F(v)\ge n-1-\lambda n$ for every $v\in U$. Since $\rho\ll \tau$, and
in fact $\rho$ is chosen sufficiently small with respect to $\tau$ and
$\gamma$, we have $\lambda\ll\xi$.

Next define the transition graph $\Gamma$ on vertex set $\binom U2$ by
joining two disjoint pairs $r,r'$ if and only if $r\cup r'\in E(G)$. Let
$\Gamma_F:=\Gamma[E(F)]$ be the subgraph induced on the nice pairs. We claim
that $\Gamma_F$ has large minimum degree.

Fix a nice pair $r=\{x,y\}\in E(F)$. Since $r$ is nice, there are at most
$\tau n$ vertices $z$ such that $\{x,y,z\}\in\mathcal B$. For every other
vertex $z\in U\setminus r$, we have
$\deg_G(xyz)\ge (1/2+\xi)n$. Hence the number of ordered pairs $(z,w)$
with $z,w\in U\setminus r$, $z\ne w$, and $xyzw\in E(G)$ is at least
$(n-2-\tau n)((1/2+\xi)n-3)$. Each unordered neighbour pair $\{z,w\}$ is
counted at most twice. Therefore, for sufficiently large $n$ and
$\tau\ll\xi$, we have
$\deg_\Gamma(r)\ge (1/2+\xi/2)\binom n2$.

Now we restrict to nice pairs. Since $d_F(v)\ge n-1-\lambda n$
for every $v\in U$, the number of non-nice pairs is at most
$\lambda n^2/2$. Thus, for every $r\in E(F)$,
$\deg_{\Gamma_F}(r)\ge \deg_\Gamma(r)-\lambda n^2/2
\ge (1/2+\xi/3)\binom n2$.
Since $|E(F)|\le \binom n2$, we have
$\deg_{\Gamma_F}(r)\ge (1/2+\xi/3)|E(F)|$ for every $r\in E(F)$.

By assumption (\ref{C2}), the prescribed pairs $p,q$ are nice, so
$p,q\in E(F)$. Apply Lemma~\ref{lem:hpframework} 
%with $R=K_U$, with this graph $F$, with $\Gamma=\Gamma_F$, 
with $(K_U, F, \Gamma_F, \xi/3)$ in place of $(R, F, \Gamma, \gamma)$ where $K_U$ is the complete graph on vertex set $U$.
%replaced by $\gamma/3$, and with $\gamma'$ replaced by, say, $\gamma/100$.
It outputs a perfect matching $M\subseteq E(F)$ of $U$ containing $p$ and $q$ such
that, for every $e\in E(F)$,
$|N_{\Gamma_F}(e)\cap M|\ge (1/2+\xi/18)|M|$.
In particular,
$\delta(\Gamma_F[M])\ge (1/2+\xi/18)|M|$.
By Lemma~\ref{lem:Hamiltoncon}, the graph $\Gamma_F[M]$ is
Hamilton-connected. Hence it contains a Hamilton path from $p$ to $q$, say
$p=r_1,r_2,\ldots,r_{n/2}=q$.
Since $M$ is a perfect matching of $U$, the pairs
$r_1,\ldots,r_{n/2}$ cover every vertex of $U$ exactly once. Moreover, for
every $i$, we have $r_ir_{i+1}\in E(\Gamma_F)$, and hence
$r_i\cup r_{i+1}\in E(G)$. Therefore the pair sequence
$r_1r_2\cdots r_{n/2}$ is a Hamilton $2$-path in $G$ with ends $p$ and
$q$.
\end{proof}

\begin{proof}[Proof of Lemma \ref{lem:splitHP}]
Let $R$ be the complete bipartite graph with vertex classes $A_U$ and
$B_U$. 
We view $\mathcal E_{U}$ as a subgraph of $R$ and $F$ be its complement with both graphs defined as bipartite graphs on $A_U B_U$, that is, $E(F)=E(R)\setminus \mathcal E_{U}$, and we call edges of $F$ \emph{nice} (split) pairs.
%if it is not in $\mathcal E_{\mathrm{sp}}$, and let $F$ be the spanning subgraph of $R$ whose edges are precisely the nice split pairs.

We first show that $F$ is dense in $R$. By (\ref{S1}), 
%every vertex of $U$ is contained in at most $\rho m$ pairs of
we have $\Delta(\mathcal E_{U})\le \rho m$ and thus $\delta(F)\ge (1-\rho)m$. 
% Hence every vertex loses at most $\rho m$
% neighbours from the opposite side, and so
% $d_F(v)\ge d_R(v)-\rho m\ge d_R(v)-\rho |U|$ for every $v\in U$.
Thus $F$ satisfies the density condition of Lemma~\ref{lem:hpframework},
with $\lambda=\rho$.
Now define the transition graph $\Gamma$ on the set of all split pairs
$A_U B_U$ as follows: two disjoint pairs $s,s'\in A_U B_U$ are adjacent
in $\Gamma$ if and only if $s\cup s'\in E(G)$. Let $\Gamma_F$ be the
subgraph of $\Gamma$ induced on $E(F)$.

We claim that $\Gamma_F$ has large minimum degree. Fix a nice split pair
$s\in E(F)$. Since $s\notin\mathcal E_{U}$, by the definition
of $\mathcal E_{U}$ we have
$|N_G(s)\cap A_U B_U|
\ge (1/2+2\xi)m^2$.
Passing from $\Gamma$ to $\Gamma_F$ only removes neighbours which are
non-nice pairs. 
%By assumption (\ref{S1}), it is easy to see that the total number of non-nice split pairs is at most $\rho m^2$. 
Therefore,
$\deg_{\Gamma_F}(s)\ge (1/2+2\xi)m^2-|\mathcal E_{U}|\ge (1/2+\xi)m^2$,
as $|\mathcal E_{U}|\le \rho m^2$ and $\rho\ll\xi$. As $|E(F)|\le m^2$, this gives
$\deg_{\Gamma_F}(s)\ge (1/2+\xi)|E(F)|$ for every $s\in E(F)$.

By (\ref{S2}), the prescribed split pairs $s_1,s_2$ are in $E(F)$. Apply Lemma~\ref{lem:hpframework} with $R$ equal
to the complete bipartite graph between $A_U$ and $B_U$, and with the present
graph $F$, with $\Gamma=\Gamma_F$, and with the framework parameter
$(F, \Gamma_F, \xi)$ in place of $(F, \Gamma, \gamma)$. We obtain a perfect matching
$M\subseteq E(F)$ containing $s_1$ and $s_2$ such that, for every
$e\in E(F)$, $|N_{\Gamma_F}(e)\cap M|\ge (1/2+\xi/6)|M|$. 
In particular,
$\delta(\Gamma_F[M])\ge (1/2+\xi/6)|M|$.
By
Lemma~\ref{lem:Hamiltoncon}, the graph $\Gamma_F[M]$ is
Hamilton-connected. Hence it contains a Hamilton path from $s_1$ to $s_2$,
say $s_1=r_1,r_2,\ldots,r_m=s_2$.
Since $M$ is a perfect matching of the complete bipartite graph between
$A_U$ and $B_U$, the split pairs $r_1,\ldots,r_m$ cover every vertex of
$U$ exactly once. Moreover, for every $i$, we have
$r_i\cup r_{i+1}\in E(\Gamma_F)$, and hence $r_i\cup r_{i+1}\in E(G)$ by the
definition of the transition graph. Therefore the pair sequence
$r_1r_2\cdots r_m$ is a Hamilton $2$-path in $G$ with ends $s_1$ and
$s_2$, and every pair in the sequence is a split pair.
\end{proof}

\section{The odd-extremal case}
\label{sec:odd}
In this section we treat the odd-extremal case, where almost all edges of $H$ are odd with respect to the bipartition $\{A, B\}$.  
%Here the extremal bipartition is such that almost all edges are odd.  
We begin by fixing the hierarchy of constants and collecting the definitions
which will be used throughout the odd-extremal argument. 
We use the same basic terminology as in the even-extremal case, but the
roles of odd and even edges are interchanged. 

\subsection{Setup and basic definitions}
\begin{setup} \label{setup:oddextr}
Suppose that
$1/n \ll c' \ll c\ll \theta \ll \eta \ll \rho \ll\alpha\ll \beta \ll \beta_2\ll\beta_2' \ll \beta_1 \ll\beta_1'\ll \mu \ll \xi \ll \gamma \ll 1$.
Let $H$ be a $4$-graph of order $n$, and let $V = V(H)$.
\end{setup}

%In this section, our main goal is to prove the following theorem.
%\begin{theorem}[Odd-extremal case]
%\label{thm:oddextremal}
 %Suppose that $1/n\ll\gamma,c'\ll1$ and let $H$ be an $n$-vertex 4-graph with $\delta_3(H)\geq n/3+\gamma n$ and $\{A',B'\}$ be a $(c',\g/3)$-odd-extremal bipartition of $V$.
%Then we construct a bipartition $\{A, B\}$ of $V$ in time $O(n^4)$ such that $H$ contains a Hamilton 2-cycle if and only if $\{A, B\}$ is odd-good, which can be checked in time $O(n^{12})$.
%\end{theorem}
\begin{definition}
Under Setup \ref{setup:oddextr}, for a fixed bipartition $\{A,B\}$ of $V$, we say that
\begin{enumerate}[label=(\roman*)]
\item for each $t\in[3]$, a $t$-subset of $V$ is \emph{$\theta$-good} if it is contained in at most $\theta^{2t-1}\binom{n}{4-t}$ even edges,
    %\item a triple $T\in\binom V3$ is \emph{$\theta$-good} if it is contained in at least $(1/3+\theta^5)n$ odd edges;
    %\item a pair $p\in\binom V2$ is \emph{$\theta$-good} if it is contained in at least $(1/3+\theta^3)\binom n2$ odd edges;
   % \item a vertex $v\in V$ is \emph{$\theta$-good} if it is contained in at least $(1/3+\theta)\binom n3$ odd edges;
    \item a pair $p$ is \emph{$\beta_2$-medium} if it is contained in at least
    $\beta_2\binom n2$ odd edges; otherwise, it is \emph{$\beta_2$-bad},
    \item 
    a vertex $v\in V$ is \emph{$(\beta_1,\beta_2)$-medium} if $v$ is contained in at least $\beta_1 n$ $\beta_2$-medium connate pairs and at least $\beta_1 n$ $\beta_2$-medium split pairs; otherwise it is \emph{$(\beta_1,\beta_2)$-bad},
    \item a $(\beta_1,\beta_2)$-medium vertex $v\in V$ is \emph{$Y$-heavy} if it is contained in at least $|Y|/2$ vertices in $Y$ which form a $\beta_2$-medium pair with $v$ where $Y\in\{A,B\}$. 
\end{enumerate}
\end{definition}

Note that by definition a $\theta$-good pair is also $\beta_2$-medium.
\subsection{Consequences of odd-extremality}

We first record the analogues of the structural estimates from the
even-extremal case. 
Since there are very few even edges, almost every small set lies mostly in odd edges.  
Consequently, the number of bad pairs
and bad vertices is small. 

\begin{proposition}\label{prop:oddtypes}
Assume Setup \ref{setup:oddextr}.
Let $\{A,B\}$ be a $(c,\g/4)$-odd-extremal bipartition of $V$ and suppose
$\delta_3(H)\ge (1/3+\gamma)n$. Then the following holds.

\begin{enumerate}
    \item At most $\eta n^t$ many $t$-sets are non $\theta$-good for $t\in[3]$.
    \item At most $36cn^2$ pairs are $\beta_2$-bad.
\item At most $288cn$ vertices are $(\beta_1,\beta_2)$-bad.

\end{enumerate}
\end{proposition}

 \begin{proof}
 The proofs of (1) and (2) are identical to those of Proposition \ref{prop:types} with the words ``odd'' and ``even'' interchanged.
 The proof of (3) is identical to that of Proposition \ref{prop:types}.
 \end{proof}

 \begin{proposition}\label{prop:oddpairconnect}
     Assume Setup~\ref{setup:oddextr}, and fix a $(c,\g/4)$-odd-extremal bipartition $\{A,B\}$ of $V$. 
Let $R\subseteq V$ with $|R|\le\mu n$.
%and suppose that at most $\rho n^2$ pairs of $V$ are not $\theta$-good. 
%satisfy $|A\cap R|\le \mu n$ and $|B\cap R|\le \mu n$.
Then the following holds.
\begin{enumerate}
    \item For any two disjoint $\theta$-good split pairs $s_1$ and $s_2$,
    there exists a connate pair $p\in {V\setminus R\choose 2}$ such that
    $s_1\cup p\in E(H)$ and $s_2\cup p\in E(H)$.

    \item For any two disjoint $\theta$-good connate pairs $p_1$ and $p_2$,
    there exists a split pair $s\in {V\setminus R\choose 2}$ such that
    $p_1\cup s\in E(H)$ and $p_2\cup s\in E(H)$.
\end{enumerate}
 \end{proposition}

 \begin{proof}
Write $a:=|A|$ and $b:=|B|$. Since $\{A,B\}$ is $(c,\g/4)$-odd-extremal, we have
$a,b\ge n/3+\g n/4$.
Recall that a $\theta$-good pair $s$ is contained in at least
$(1/3+\g-\theta^3)\binom n2\ge(1/3+\g/2)\binom n2$ odd edges. 
Thus, if $s$ is split, then $|N_{\mathrm{con}}(s)|\ge (1/3+\g/2)\binom n2$; if $p$ is connate, then $|N_{\mathrm{sp}}(p)|\ge (1/3+\g/2)\binom n2$.

(1) Let $s_1,s_2$ be two disjoint $\theta$-good split pairs.
Then $N_{\mathrm{con}}(s_1),N_{\mathrm{con}}(s_2)\subseteq \binom A2\cup \binom B2$. 
Let $\mathcal C=\binom{A}2\cup\binom{B}2$.
Moreover, since $a+b=n$ and $a,b\ge (1/3+\g/4)n$, we have
$|\mathcal C|=\binom a2+\binom b2\le 5n^2/18$.
Therefore
$|N_{\mathrm{con}}(s_1)\cap N_{\mathrm{con}}(s_2)|
\ge |N_{\mathrm{con}}(s_1)|+|N_{\mathrm{con}}(s_2)|-|\mathcal C|
\ge 2(1/3+\g/2)\binom n2-5n^2/18\ge n^2/30$
for sufficiently large $n$. The number of connate pairs that meet $R$ is
at most $|R|n\le 2\mu n^2$. Since $\mu\ll1$, there
exists a connate pair $p\in N_{\mathrm{con}}(s_1)\cap N_{\mathrm{con}}(s_2)$
with $p\subseteq V\setminus R$. Then $p\cup s_1\in E(H)$ and
$p\cup s_2\in E(H)$. 

(2) Let $p_1,p_2$ be two disjoint $\theta$-good connate pairs.
Then $N_{\mathrm{sp}}(p_1),N_{\mathrm{sp}}(p_2)\subseteq A B$.
Hence by inclusion-exclusion and $ab\le n^2/4$, we get
$|N_{\mathrm{sp}}(p_1)\cap N_{\mathrm{sp}}(p_2)|
\ge 2(1/3+\g/2)\binom n2- ab\ge n^2/20$.
Note that the number of split pairs
that meet $R$ is at most
$|R|n\le 2\mu n^2$. Since $\mu\ll 1$, there exists a split pair $s\in N_{\mathrm{sp}}(p_1)\cap N_{\mathrm{sp}}(p_2)$ with
$s\subseteq V\setminus R$. Then $p_1\cup s\in E(H)$ and
$p_2\cup s\in E(H)$.
\end{proof}
%Let $\mathcal B$ be the family of $\beta_2$-bad pairs. 
%A vertex $v$ is called \emph{bad} if $d_{\mathcal B}(v)>\rho n$.
%Define $X_{\mathrm{bad}}:=\{v\in V: v \text{ is bad}\}$.
%\begin{claim}\label{claim:bad-associated-small}
%Suppose that $\delta(H)\ge n/3+\gamma n$ and that $\{A,B\}$ is
%$(c,\g/3)$-odd-extremal. Then $|X_{\mathrm{bad}}|\le  3c\rho^{-1}n$.
%\end{claim}

%\begin{proof}
%By Proposition \ref{prop:oddtypes}, we have $|\mathcal B|\le 3c n^2/2$. On the other hand, every vertex in $X_{\mathrm{bad}}$ is incident with more than $\rho n$ bad pairs. Therefore $|X_{\mathrm{bad}}|\rho n\le2|\mathcal B|\le 3c n^2$. Thus $|X_{\mathrm{bad}}|\le 3c\rho^{-1}n$.
%\end{proof}

\begin{definition}
Let $\mathcal E_p$ be the set of non $\theta$-good pairs.
A vertex $v\in V$ is called \emph{exceptional} if it is contained in more than $\rho n$ pairs of $\mathcal E_p$. 
\end{definition}

Denote the set of exceptional vertices by $X_E$.
Then $|X_E|\rho n\le 2|\mathcal E_p|\le2\eta n^2$ and
\begin{equation}\label{ine:xtr}
    |X_E|\le 2\eta\rho^{-1}n.
\end{equation}

A balanced bipartite graph is \emph{weakly Hamilton-connected} if any two different vertices in different parts can be connected by a Hamilton path.

\begin{lemma}\cite{bHamiltonConnected}\label{lem:bHamiltoncon}
A bipartite graph $G = G(X,Y,E)$ with $|X| = |Y | = n$ is weakly Hamilton-
connected if $d(x) + d(y)\ge n + 2$ for any two nonadjacent vertices $x\in X$ and $y\in Y$.
\end{lemma}
%Let $\mathcal C=\binom{A}{2}\cup\binom{B}{2}$.
%Define the family of \emph{exceptional split pairs} by $E_{\mathrm{sp}}^{\mathrm{odd}}:=\{s\in A\times B: |N_{\mathrm{con}}(s)|<(1/2+2\xi)|\mathcal C|\}$.
%Define the family of \emph{exceptional connate pairs} by $E_{\mathrm{con}}^{\mathrm{odd}}:=\{p\in\mathcal C: |N_{\mathrm{sp}}(p)|<(1/2+2\xi)|A||B|\}$.
%Set $E_{\mathrm{odd}}:=E_{\mathrm{sp}}^{\mathrm{odd}}\cup E_{\mathrm{con}}^{\mathrm{odd}}$.
%Define $X_{\mathrm{odd}}$ to be the vertices which are contained in more that $\rho n$ pairs from $E_{\mathrm{odd}}$.
For convenience, we restate Lemma \ref{oddextremal}.
\medskip

\noindent
%\begin{lemma}
\textbf{Lemma \ref{oddextremal} (Odd-extremal case)}
%\label{evenextremal}
\emph{Suppose that $1/n\ll c'\ll\gamma\ll1$ and let $H$ be an $n$-vertex $4$-graph with $\delta_3(H)\geq n/3+\gamma n$ and $\{A',B'\}$ be a $(c',\g/3)$-odd-extremal bipartition of $V$.
Then we construct a bipartition $\{A, B\}$ of $V$ in time $O(n^4)$ such that $H$ contains a Hamilton $2$-cycle if and only if $\{A, B\}$ is odd-good, which can be checked in time $O(n^{12})$.}

\subsection{Proof of Lemma \ref{oddextremal}}
%Before delving into the details, we give an overview of the proof.
The proof idea is similar to Lemma \ref{evenextremal}.
It is worth noting that a bridge with $\theta$-good split ends is sufficient to overcome the parity obstruction.
%\red{Will remove and mention bridges}

Now we give the necessary lemmas for the proof of Lemma~\ref{oddextremal}.
The first lemma is an analogue of the refinement lemma in the
even-extremal case, where we move all bad vertices to the opposite side.  
Since there are few bad vertices, the bipartition remains odd-extremal.

\begin{lemma}\label{lem:oddmoved}
Assume Setup \ref{setup:oddextr}.
Let $\{A',B'\}$ be a $(c',\g/3)$-odd-extremal bipartition of $V$. Move all
$(\beta_1',\beta_2')$-bad vertices to the opposite side and denote the
resulting bipartition by $\{A,B\}$. Then $\{A,B\}$ is $(c,\g/4)$-odd-extremal and
every vertex is $(\beta_1,\beta_2)$-medium with respect to $\{A,B\}$.
\end{lemma}

%After refining the bipartition, 
Similar to the even-extremal case, we also need a bridge to break the parity issue.  
In the odd-extremal setting the bridge should have split ends.
The bridge also covers the exceptional vertices and leaves the remaining vertex set in
one of the two admissible congruence classes.

\begin{lemma}\label{lem:oddbridge}
Assume Setup \ref{setup:oddextr}.
Let $H$ be an $n$-vertex $4$-graph with $\delta_3(H)\ge n/3+\gamma n$, and let $\{A,B\}$ be a
$(c,\g/4)$-odd-extremal and odd-good bipartition such that every vertex is $(\beta_1,\beta_2)$-medium. 
Then one can find a path $P$ with $\theta$-good split ends $s,s'$ such that $X_E\subseteq V(P)$ and $|V(P)|\le\mu n$ vertices.
Furthermore, let $V^*:=(V\setminus V(P))\cup s\cup s'$,
$A^*:=A\cap V^*$ and $B^*:=B\cap V^*$, either
$|V^*|\equiv 6\pmod 8$ and $|A^*|-|B^*|\equiv 2\pmod 4$, or
$|V^*|\equiv 2\pmod 8$ and $|A^*|-|B^*|\equiv 0\pmod 4$.
\end{lemma}

The final ingredient is the Hamilton path lemma for the odd-extremal
setting.  
Once the bridge has removed all exceptional vertices and arranged
the correct residue, this lemma covers the remaining vertices by a Hamilton $2$-path between the two prescribed good split pairs.

\begin{lemma}
\label{lem:oddHP}
Assume Setup \ref{setup:oddextr}.
Let $U=A_U\cup B_U$ be a vertex set such that $|V\setminus U|\le2\mu n$. Let $s_1,s_2$ be two disjoint $\theta$-good split pairs in $A_U B_U$. Suppose that either $|U|\equiv 6\pmod 8,\ |A_U|-|B_U|\equiv 2\pmod 4$, or $|U|\equiv 2\pmod 8,\ |A_U|-|B_U|\equiv 0\pmod 4$.
Assume that every vertex of $U$ is contained in at most $\rho n$ non $\theta$-good pairs.
%Write $N:=|U|$, $a:=|A_U|$, $b:=|B_U|$, and define $x:=\frac{N+2}{4}$, $y:=\frac{N-2}{4}$, $y_A:=\frac{a-x}{2}$, and $y_B:=\frac{b-x}{2}$. 
%Assume the following transition degree conditions.
Then $H[U]$ contains a Hamilton $2$-path from $s_1$ to $s_2$.
\end{lemma}

\begin{proof}[Proof of Lemma \ref{oddextremal}]
% If $n$ is odd, then $H$ contains no Hamilton $2$-cycle, since every
% $2$-cycle in a $4$-graph has an even number of vertices. In this case
% the algorithm outputs this divisibility obstruction as a certificate.
% Thus we may assume from now on that $n$ is even.
Note that we may assume that $n$ is even, as otherwise there is no Hamilton 2-cycle in $H$ and the algorithm should just output the value of $n$ as a certificate.

\textbf{Step 1. Refine the extremal bipartition.}
Since $H$ is $c'$-odd-extremal, fix a $(c',\g/3)$-odd-extremal bipartition $\{A',B'\}$ of $V$ for which there are at most $c'n^4$ even edges.
We begin by moving all $(\beta_1',\beta_2')$-bad vertices to their opposite side.
To be precise, let $A_{\text{bad}}=\{a\in A':a\ \text{is}\ (\beta_1',\beta_2')\text{-bad}\}$ and $B_{\text{bad}}=\{b\in B':b\ \text{is}\ (\beta_1',\beta_2')\text{-bad}\}$ and set $A:=(A'\setminus A_{\text{bad}})\cup B_{\text{bad}}$
and $B:=(B'\setminus B_{\text{bad}})\cup A_{\text{bad}}$, we say the vertices $A_{\text{bad}}\cup B_{\text{bad}}$ are \emph{moved}.

By Lemma \ref{lem:oddmoved}, we obtain that the bipartition $\{A,B\}$ is $(c,\g/4)$-odd-extremal, where every vertex is $(\beta_1,\beta_2)$-medium.
Note that $(\beta'_1,\beta'_2)$-badness of a vertex $v$ can be checked by checking the edges containing $v$, and thus the sets $A_{\mathrm{bad}}, B_{\mathrm{bad}}$ and the new bipartition $\{A, B\}$ can be constructed in time $O(n^4)$.
As the odd-goodness of $\{A,B\}$ can be checked in time $O(n^{12})$ and odd-goodness is a necessary condition for 2-Hamiltonicity, to complete the proof of the theorem it suffices to assume that $\{A,B\}$ is odd-good, and show that $H$ contains a Hamilton 2-cycle.

\textbf{Step 2. Construct a bridge.}
By Lemma~\ref{lem:oddbridge}, we obtain a path $P_0$ with split ends $s,s'$, where $s,s'$ are $\theta$-good, $X_E\subseteq V(P_0)$ and $|V(P_0)|\le \mu n$.
Next we choose two $\theta$-good split pairs $s_1,s_2$ outside $V(P_0)$, this can be done since  $|\mathcal E_p|\le \eta n^2$ by Proposition \ref{prop:oddtypes}.
By Proposition \ref{prop:oddpairconnect}, there exist connate pairs $p,p'$ such that $s_1\cup p,p\cup s\in E(H)$ and $s_2\cup p',p'\cup s'\in E(H)$ avoiding the used vertices.

\textbf{Step 3. Cover the remaining vertices.}
Let $P=s_1psP_0s'p's_2$, $U=(V\setminus V(P))\cup s_1\cup s_2$, $A_U:=A\cap U$, and $B_U:=B\cap U$, we have either $|U|\equiv 6\pmod 8,\ |A_U|-|B_U|\equiv 2\pmod 4,$ or $|U|\equiv 2\pmod 8,\ |A_U|-|B_U|\equiv 0\pmod 4$. 
%It is easy to see that $|A_U|,|B_U|\ge n/3+\g n/4-\mu n\ge n/3+\gamma n/5$ since $\mu\ll\g$.
%Let $\mathcal E_{p}^U=\mathcal E_{p}\cap \binom U2$.
Note that $|V(P)|\le2\mu n$.
Since $X_E\subseteq V(P_0)$, every vertex of $U$ is contained in at most $\rho n$ non $\theta$-good pairs.
Also, $s_1,s_2$ are $\theta$-good.
Now, apply Lemma~\ref{lem:oddHP} to $H[U]$ with endpoint split pairs $s_2,s_1$. We obtain a Hamilton $2$-path $P_1$ in $H[U]$.
Finally, $s_2Ps_1P_1s_2$ is a Hamilton $2$-cycle of $H$. 
\end{proof}

Now it remains to prove Lemmas~\ref{lem:oddmoved} --~\ref{lem:oddHP}.

\subsection{Proof of Lemma \ref{lem:oddmoved}}

The proof follows the same argument as in the even-extremal case.
We move all bad vertices to the opposite side. 
The number of moved
vertices is small, so the new bipartition is still odd-extremal.  Moreover, the choice of the parameters ensures that each vertex is medium with
respect to the new bipartition.

\begin{proof}
Let $L:=A_{\mathrm{bad}}\cup B_{\mathrm{bad}}$. 
By Proposition~\ref{prop:oddtypes} (3), applied with $c',\beta'_1,\beta'_2$ in
place of $c,\beta_1,\beta_2$, we have $|L|\le 288c'n$.
Since at most $288c'n$ vertices are moved, we have
$|A|,|B|\ge(1/3+\g/3)n-288c'n\ge(1/3+\g/4)n$
and the number of even edges with respect to $\{A,B\}$ is at most
$c' n^4+|L|(n-1)^{3}\le c n^4$, as $c'\ll c$. 
Thus $\{A,B\}$ is $(c,\g/4)$-odd-extremal.

Since $c'\ll\beta_2\ll\beta_2'\ll\beta_1\ll\beta_1'$ and
$\beta_1'n-288c'n\ge\beta_1n$, every unmoved vertex which was
$(\beta_1',\beta_2')$-medium under $\{A',B'\}$ is
$(\beta_1,\beta_2)$-medium under $\{A,B\}$.
Now let $v\in L$. 
Without loss of generality, assume that $v\in A_{\mathrm{bad}}$, so $v$ is
moved from $A'$ to $B$.

Note that it suffices to show that $v$ is contained in $\beta_1n^3$ edges
$e$ of $H$ with $|e\cap A'|=|e\cap B'|=2$.
Indeed, after moving the vertices in $L$, at least
$\beta_1n^3-288c'n^3\ge 0.9\beta_1n^3$ such edges become odd edges with
respect to $\{A,B\}$. Thus the number of $\beta_2$-medium pairs $va$ with
$a\in A$ is at least
$(0.9\beta_1n^3-n\cdot\beta_2n^2)/\binom n2>\beta_1n$, and similarly the
number of $\beta_2$-medium pairs $vb$ with $b\in B$ is at least
$(0.9\beta_1n^3-n\cdot\beta_2n^2)/\binom n2>\beta_1n$, as desired.

Since $v$ is $(\beta_1',\beta_2')$-bad with respect to $\{A',B'\}$, it is
contained in at most $\beta_1'n$ $\beta_2'$-medium split (or connate)  pairs under $\{A',B'\}$. Thus, $v$ was in at least $(1/3+\g/3-\beta_1'-288c')n\ge n/3$ $\beta_2'$-bad split (or connate) pairs $p$ under $\{A',B'\}$.

Fix such a pair $p$. Every such $p$ is contained in fewer than
$\beta_2'\binom n2$ old odd edges. 
By the minimum codegree condition, we
infer that $p$ is in at least $\frac12\frac{n}3\delta_3(H)\ge \frac{n^2 }{18}$ edges of $H$ with appropriate intermediate vertices: if $p$ is split, then we sum over all vertices of $A'$; if $p$ is connate, then we sum over all vertices of $B'$.
Therefore, at least $n^2/20$ such edges contains exactly two vertices of $A'$ and two vertices of $B'$.
Since there are at least $n/3$ choices for such a pair $p$, we obtain that at least $\frac12\cdot\frac n3\cdot\frac{n^2}{20}=\frac{n^3}{120}$ edges containing $v$ with two vertices in $A'$ and two vertices in $B'$. Since $\beta_1\ll 1$, this is at least $\beta_1n^3$, and we are done.

The argument for vertices moved from $B'$ to $A$ is identical with the
roles of $A'$ and $B'$ interchanged.
Hence every vertex of $H$ is $(\beta_1,\beta_2)$-medium with respect to
$\{A,B\}$.
\end{proof}

\subsection{Proof of Lemma \ref{lem:oddbridge}}
We now prove the bridge lemma.  
The proof has two parts.  
First we establish several local tools which allow us to replace bad pairs by medium pairs and to upgrade medium pairs to good pairs. 
Then we use the
odd-goodness to construct a short split-ended path satisfying the
desired residue.  
Finally, this short path is extended to cover all exceptional vertices without changing the required residue.
The first local tool upgrades medium ends to good ends via extensions.

\begin{proposition}\label{prop:oddmedium-to-good}
Assume Setup~\ref{setup:oddextr}, and suppose that $\delta_3(H)\geq n/3+\g n$ and that $\{A, B\}$ is a $(c,\g/4)$-odd-extremal bipartition of $V$.
If $R\subseteq V$ satisfies $|R| \leq \frac{1}{3} \beta n$, then for every $\beta$-medium pair $p$ there exists a $\theta$-good pair $s\in \binom{V \setminus R}{2}$ such that $p\cup s$ is an odd edge.
\end{proposition}
\begin{proof}
Let $p$ be a $\beta$-medium pair. 
Then there are at least $\beta {n\choose 2}$ pairs $s$ such that $p\cup s$ is an odd edge.
By Proposition \ref{prop:oddtypes}, the number of non $\theta$-good pairs is at most
$\eta n^2 < \frac{\beta}{4}{n\choose 2}$.
Moreover, the number of pairs meeting $R$ is at most $|R|n\le \beta n^2/3<\frac{3\beta}{4}{n\choose 2}$ for sufficiently large $n$. 
Therefore there exists a $\theta$-good pair $s\in {V\setminus R\choose 2}$ such that $p\cup s$ is an odd edge. 
\end{proof}

For some replacement arguments we also need a vertex-level supply of medium pairs on a prescribed side of the bipartition. 
The following claim shows
that almost every vertex is heavy towards both sides. 

\begin{claim}\label{clm:heavycount}
    Assume Setup \ref{setup:oddextr}. Let $H$ be an $n$-vertex $4$-graph with $\delta_3(H)\ge (1/3+\gamma)n$, and let $\{A,B\}$ be a $(c,\g/4)$-odd-extremal bipartition of $V$. Assume furthermore that, with respect to this bipartition, every vertex is $(\beta_1,\beta_2)$-medium. 
    Then all but at most $504cn$ vertices are $Y$-heavy for each $Y\in\{A,B\}$. 
\end{claim}
\begin{proof}
Let $\mathcal B$ be the family of all $\beta_2$-bad pairs.
By Proposition \ref{prop:oddtypes},  we have $|\mathcal B|\le 36c n^2$.
Now define $X_A:=\{v\in V:v\text{ is not }A\text{-heavy}\}$. 
Since every vertex is $(\beta_1,\beta_2)$-medium, every vertex already has at least $\beta_1 n$ $\beta_2$-medium neighbours in $B$. 
Hence, if $v\in X_A$, then the failure of being $A$-heavy must come from the $A$-side, namely $|\{a\in A:va\text{ is }\beta_2\text{-medium}\}|<|A|/2$. 
Consequently, $|\{a\in A:va\in\mathcal B\}|\ge |A|/2-1$. Using $|A|\ge n/3+\g n/4$, we get, for all sufficiently large $n$, $|\{a\in A:va\in\mathcal B\}|\ge n/7$. We double-count the set $\mathcal I_A:=\{(v,a):v\in X_A,\ a\in A,\ va\in\mathcal B\}$. Thus, $|\mathcal I_A|\ge |X_A|n/7$.
On the other hand, each bad pair is counted at most twice in $\mathcal I_A$, so $|\mathcal I_A|\le 2|\mathcal B|\le 72c n^2$. Therefore $|X_A|\le 504c n$.  
Let $X_B:=\{v\in V:v\text{ is not }B\text{-heavy}\}$.
Same arguments give $|X_B|\le504cn$ . 
\end{proof}

We next prove the main replacement claim.  
Its purpose is to handle the case when a bad pair appears in a parity block.  
The claim shows that such a local object may be replaced by an even edge consisting of two medium pairs.  

\begin{claim}\label{clm:csreplace}
Assume Setup \ref{setup:oddextr}. 
Let $H$ be an $n$-vertex $4$-graph with $\delta_3(H)\ge (1/3+\gamma)n$, and let $\{A,B\}$ be a $(c,\g/4)$-odd-extremal bipartition of $V$ and every vertex is $(\beta_1,\beta_2)$-medium.
Let $R\subseteq V$ have size at most $\eta n$. 
If $p\in\binom{V\setminus R}{2}$ is a
$\beta_2$-bad connate (or split) pair, then there is an even
edge $p_1\cup p_2\subseteq V\setminus R$ such that $p_1,p_2$ are both $\beta_2$-medium connate (or split) pairs.
\end{claim}
\begin{proof}
For the connate case, let $p=uv$ be such a $\beta_2$-bad connate pair. 
By symmetry we may assume that $p=uv\subseteq A$. 
If one of $u,v$ is already
$A$-heavy, we label that vertex as $v'$ and the other vertex as $u$; if neither endpoint is heavy, then
we claim that there exists another vertex $v'$ such that $uv'$ is a $\beta_2$-bad connate pair and $v'\in V\setminus R$ is $A$-heavy.

Indeed, let $F$ be the graph with vertex set $V$ and edge set as the collection of $\beta_2$-bad connate pairs in $A$.
By Proposition \ref{prop:oddtypes}, we obtain that $e(F)\le 36c n^2$. 
Now $u$ is $(\beta_1,\beta_2)$-medium but not $A$-heavy. 
Since $u$ is $(\beta_1,\beta_2)$-medium, it already has at least $\beta_1 n$ $\beta_2$-medium neighbours in $B$. 
Therefore the failure of being $A$-heavy implies that $|\{a\in A:ua\text{ is }\beta_2\text{-medium}\}|<|A|/2$.
Consequently, $|N_{F-R}(u)|\ge(1/2-\eta)|A|$. 
We claim that $N_{F}(u)\setminus R$ contains an $A$-heavy vertex.
Suppose not, that is, every vertex $v'\in N_{F}(u)\setminus R$ is not $A$-heavy. 
Since every vertex is $(\beta_1,\beta_2)$-medium, the same argument applied to $v'$ gives $d_{F-R}(v')\ge (1/2-\eta)|A|\ge n/7$ for every $v'\in N_{F}(u)$.
%As $(1/3+\g/4)n\le|A|\le(2/3-\g/4)n$, we obtain
We get $\sum_{v'\in N_{F}(u)\setminus R} d_{F}(v')\ge n^2/49$ and $e(F)\ge n^2/98$, which contradicts $e(F)\le 36c n^2$. 
Thus $N_{F}(u)\setminus R$ contains an $A$-heavy vertex. 
We choose a $A$-heavy vertex $v'\in N_{F}(u)\setminus R$ and $uv'$ is a $\beta_2$-bad pair. 

Next we claim that there is an even edge  consisting of two $\beta_2$-medium connate pairs.
Let \(\sigma\) be a constant with \(\eta\ll \beta_2\ll \sigma\ll \beta_1,\gamma\). For every \(w\in A\setminus\{u,v'\}\), put \(b(w):=|N_H(uv'w)\cap B|\). Since \(uv'\) is \(\beta_2\)-bad, we have \(\deg_{H_{\mathrm{odd}}}(uv')<\beta_2\binom n2\).
Hence \(\sum_{w\in A\setminus\{u,v'\}} b(w)\le \deg_{H_{\mathrm{odd}}}(uv')<\beta_2\binom n2\). 
Therefore the set \(W:=\{w\in A:b(w)\ge \sigma n\}\) satisfies \(|W|\le \beta_2\sigma^{-1}n\). 
Since $u$ is \((\beta_1,\beta_2)\)-medium, the vertex \(u\) has at least \(\beta_1 n\) vertices \(w\in A\) such that \(uw\) is a \(\beta_2\)-medium pair.
Since $|W|\le\beta_2\sigma^{-1}n, |R|\le\eta n$, and \(\eta\ll\beta_2\ll\beta_1\), we may choose \(w\in A\setminus(\{u,v'\}\cup W\cup R)\) such that \(uw\) is \(\beta_2\)-medium. For this choice of \(w\), we have \(|N_H(uv'w)\cap B|<\sigma n\). 
By the minimum codegree condition, \(|N_H(uv'w)\cap A|\ge (1/3+\gamma)n-\sigma n\ge(1/2+\g/2)|A|\) since \(|A|\le (2/3-\g/4)n\) and $\sigma\ll\g$. 
Now use that \(v'\) is \(A\)-heavy. 
Set \(S:=\{z\in A:v'z\text{ is }\beta_2\text{-medium}\}\). Then \(|S|\ge|A|/2\).
Therefore \(|(N_H(uv'w)\cap A)\cap S|\ge \g|A|/2\). 
Thus, we choose a vertex \(z\in (N_H(uv'w)\cap A)\cap S\) with \(z\notin R\cup\{u,v',w\}\) and \(uv'wz\in E(H)\), the pair \(uw\) is \(\beta_2\)-medium by the choice of \(w\), and the pair \(v'z\) is \(\beta_2\)-medium by the choice of \(z\). 
Thus \(p'_1:=uw\) and \(p'_2:=v'z\) are two vertex-disjoint \(\beta_2\)-medium connate pairs.

For the split case,
say $x\in A$ and $y\in B$. 
We claim that either $y$ is  $A$-heavy, or $x$ is $B$-heavy. 
Indeed, let $s=xy$ be a
$\beta_2$-bad split pair disjoint from $R$, where $x\in A$ and $y\in B$.
Let
$X_Y:=\{v\in V:v\text{ is not }Y\text{-heavy}\}$ where $Y\in\{A,B\}$.
Suppose that $y\in X_A$ and that $x\in X_B$. 
Since $x$ is
$(\beta_1,\beta_2)$-medium, it is contained in at least $\beta_1 n$ $\beta_2$-medium
pairs in $A$. 
Hence the failure of $x$ being $B$-heavy must come from the $B$-side, namely
$|\{b\in B:xb\text{ is }\beta_2\text{-medium}\}|<|B|/2$.
Therefore
$|\{b\in B:xb\text{ is }\beta_2\text{-bad}\}|\ge|B|/2$.
Thus, we can choose a vertex $y'\in B\setminus (R\cup X_A)$ such that $xy'$ is
$\beta_2$-bad since $|X_A|\le\eta n$ and $|R|\le\eta n$  by Claim \ref{clm:heavycount}.
Now we may assume that either $y$ is  $A$-heavy, or $x$ is $B$-heavy. 

%Replacing $y$ by $y'$, 
First we assume that $y$ is $A$-heavy. 
Choose $\sigma$ with $\beta_2\ll\sigma\ll\beta_1,\gamma$. For each
$w\in B\setminus\{y\}$, set $b(w):=|N_H(xyw)\cap B|$. Since $xy$ is
$\beta_2$-bad, it is contained in fewer than $\beta_2\binom n2$ odd
edges.
Also we have $\sum_{w\in B\setminus\{y\}}b(w)\le 2\deg_{H_{\mathrm{odd}}}(xy)<2\beta_2\binom n2$.
Thus the set $W:=\{w\in B:b(w)\ge\sigma n\}$ has size at most
$\beta_2\sigma^{-1}n$. 
Since $x$ is $(\beta_1,\beta_2)$-medium, it has
at least $\beta_1 n$ $\beta_2$-medium neighbours in $B$. 
As
$\beta_2/\sigma\ll\beta_1$, we may choose
$w\in B\setminus(R\cup\{y\}\cup W)$ such that $xw$ is
$\beta_2$-medium. 
Then $|N_H(xyw)\cap B|<\sigma n$. 
By the minimum codegree condition,
$|N_H(xyw)\cap A|\ge(1/3+\gamma)n-\sigma n\ge(1/2+\g/2)|A|$.
On the other hand, since $y$ is
$A$-heavy, the set
$S:=\{z\in A:yz\text{ is }\beta_2\text{-medium}\}$ has size at least
$|A|/2$. Therefore
$|(N_H(xyw)\cap A)\cap S|\ge\gamma|A|/2$.
Then we choose a vertex $z\in (N_H(xyw)\cap A)\cap S$ outside $R\cup\{x,y,w\}$ and $xywz\in E(H)$.
Moreover, $xw$ and $yz$ are two vertex-disjoint $\beta_2$-medium split pairs.

If $x$ is $B$-heavy, then we run the argument with the roles of $A$ and $B$ interchanged.
%Consequently, whenever an even edge $e$ is disjoint from a constant-size set $R$, we may replace $e$, if needed, by an even edge outside $R$ which is the union of two disjoint $\beta_2$-medium connate pairs; if $|e\cap A|=|e\cap B|=2$, we may instead replace it by an even edge which is the union of two disjoint $\beta_2$-medium split pairs.
\end{proof}
With the local replacement tools in hand, we now formulate the remaining tools for split-ended paths.  Since the final Hamilton path lemma
requires one of two specific congruence classes, the bridge must delete a number of vertices with a prescribed residue modulo $8$ and prescribed imbalance modulo $4$.

Let $P$ be a path whose two ends are split pairs $s_1,s_2$. Define the
effective deletion of $P$ by
$D(P):=V(P)\setminus(s_1\cup s_2)$. We define the residue of $P$ by
$\operatorname{res}(P):=(|D(P)|\bmod 8,\,
|D(P)\cap A|-|D(P)\cap B|\bmod 4)$.
For $(m,d)\in \{0,2,4,6\}\times\{0,2\}$ define
$\mathcal R_{m,d}^s$ to be the set of pairs $(r_8,r_4)$ such that
$(m-r_8,d-r_4)$ is equal to either $(6,2)$ or $(2,0)$ modulo $(8,4)$.
Equivalently,
\begin{equation}\label{tabel}
\begin{array}{c|c}
(m,d) & \mathcal R_{m,d}^s \\ \hline
(0,0),(4,2) & \{(2,2),(6,0)\} \\
(2,2),(6,0) & \{(4,0),(0,2)\} \\
(4,0),(0,2) & \{(2,0),(6,2)\} \\
(6,2),(2,0) & \{(0,0),(4,2)\}.
\end{array}
\end{equation}

The next lemma constructs the bridge we use.

\begin{lemma}\label{lem:28oddbridge}
Assume Setup \ref{setup:oddextr}.
Let $H$ be an $n$-vertex $4$-graph with $\delta_3(H)\ge n/3+\gamma n$, and let $\{A,B\}$ be a
$(c,\g/4)$-odd-extremal bipartition such that every vertex is $(\beta_1,\beta_2)$-medium. 
Assume moreover that $\{A,B\}$ is odd-good.
Then one can find a path $P_0$ with at most $28$ vertices, whose two ends are $\theta$-good split pairs $s_1,s_2$, such
that, with $V^*:=(V\setminus V(P_0))\cup s_1\cup s_2$,
$A^*:=A\cap V^*$ and $B^*:=B\cap V^*$, either
$|V^*|\equiv 6\pmod 8$ and $|A^*|-|B^*|\equiv 2\pmod 4$, or
$|V^*|\equiv 2\pmod 8$ and $|A^*|-|B^*|\equiv 0\pmod 4$.
\end{lemma}
\begin{proof}
    Let $m\in\{0,2,4,6\}$ and $d\in\{0,2\}$ be such that
$m\equiv n\pmod 8$ and $d\equiv |A|-|B|\pmod 4$.
We consider separately four cases for the pair $(m,d)$ as in Definition \ref{def:evenoddgood}.
To obtain the desired path, our goal is to construct a path $P_0$ whose two ends are $\theta$-good split pairs and $\operatorname{res}(P_0)\in \mathcal R_{m,d}^s$.

\textbf{Case (\ref{O0}).}  Suppose that $(m,d)\in\{(0,0),(4,2)\}$. 
We choose a
$\theta$-good split pair $s_1$. 
Using Proposition \ref{prop:oddmedium-to-good} twice, we choose a $\theta$-good connate pair $p$ such that $s_1\cup p\in E(H)$, and choose a
$\theta$-good split pair $s_2$ disjoint from all previously chosen
vertices such that $p\cup s_2\in E(H)$.
Set $P_0=s_1ps_2$. 
Hence
$\operatorname{res}(P_0)=(2,2)\in\mathcal R_{m,d}^s$ in (\ref{tabel}).

\textbf{Case (\ref{O1}).} Suppose that $(m,d) \in \{(2,2),(6,0)\}$ and let $e$ be an even edge. 
It is easy to see that every even edge admits a decomposition into two disjoint connate pairs.
If $e$ contains a $\beta_2$-bad connate pair, then we apply Claim \ref{clm:csreplace} to obtain another even edge $p_1\cup p_2$ with $R=e$ such that both $p_1$ and $p_2$ are $\beta_2$-medium.
By replacing $e$ with $p_1\cup p_2$ if necessary, we may assume that $e=p_1\cup p_2$ consists of two $\beta_2$-medium connate pairs.
%Since otherwise using Claim \ref{clm:csreplace}, we can replace this edge by another vertex-disjoint even edge consisting of two $\beta_2$-medium connate pairs, rename it by $e:=p_1\cup p_2$. 
By Proposition \ref{prop:oddmedium-to-good}, we may choose disjoint $\theta$-good split pairs $s_1,s_2$, avoiding $p_1\cup p_2$, such that $s_1\cup p_1\in E(H)$ and $p_2\cup s_2\in E(H)$. 
Set $P_0:=s_1p_1p_2s_2$. 
Hence $\operatorname{res}(P_0)=(4,0)\in\mathcal  R_{m,d}^s$ in (\ref{tabel}). 

\textbf{Case (\ref{O2}).} Suppose that
$(m,d)\in\{(4,0),(0,2)\}$ and
$H_{\mathrm{even}}$ has total $2$-pathlength at least two, so it contains
either two disjoint even edges or a $2$-path of two even edges.

First suppose that there are two disjoint even edges $e_1$ and $e_2$.
If $e_1$ contains a $\beta_2$-bad connate pair, then we use Claim~\ref{clm:csreplace} with $R=e_1\cup e_2$ to obtain an even edge $p_1\cup p_2$ such that $p_1,p_2$ are $\beta_2$-medium.
By replacing $e_1$ with $p_1\cup p_2$ if necessary, we may assume that $e_1=p_1\cup p_2$ consists of two $\beta_2$-medium connate pairs.
The same argument to $e_2$ shows that we may assume
$e_2=p_3\cup p_4$ consists of two
$\beta_2$-medium connate pairs. 
Using Proposition~\ref{prop:oddmedium-to-good} twice, choose disjoint $\theta$-good split pairs
$s_1,s_3,s_4,s_2$ such that
$s_1\cup p_1, p_2\cup s_3, s_4\cup p_3, p_4\cup s_2\in E(H)$.
By Proposition \ref{prop:oddpairconnect}, we can choose a connate pair $q$ with $s_3\cup q,q\cup s_4\in E(H)$. 
Let $P_0=s_1p_1p_2s_3qs_4p_3p_4s_2$.
Hence
$\operatorname{res}(P_0)=(6,2)\in\mathcal R_{m,d}^s$ in (\ref{tabel}).

Now suppose that there are even edges $e_1,e_2$ with $|e_1\cap e_2|=2$.
Write $f_1=e_1\setminus e_2$, $f_2=e_1\cap e_2$, and
$f_3=e_2\setminus e_1$. Then $f_1f_2f_3$ is a $2$-path and the three
pairs are either all connate or all split. If $f_1$ is
$\beta_2$-bad, then replace $f_1$ by another even edge by Claim \ref{clm:csreplace} with $R=e_2\cup f_1$, reducing to the previous paragraph; similarly we may
assume that $f_3$ is $\beta_2$-medium.

If $f_1,f_2,f_3$ are connate pairs, then choose disjoint $\theta$-good split
pairs $s_1,s_2$ with $s_1\cup f_1,f_3\cup s_2\in E(H)$ using Proposition \ref{prop:oddmedium-to-good} twice.
Set
$P_0=s_1f_1f_2f_3s_2$. 
We have
$\operatorname{res}(P_0)=(6,2)\in\mathcal R_{m,d}^s$ in (\ref{tabel}).

If $f_1,f_2,f_3$ are split pairs, then choose disjoint $\theta$-good connate
pairs $p_1,p_2$ with $p_1\cup f_1,f_3\cup p_2\in E(H)$, and then choose
disjoint $\theta$-good split pairs $s_1,s_2$ with
$s_1\cup p_1,p_2\cup s_2\in E(H)$ using Proposition \ref{prop:oddmedium-to-good}. 
Set
$P_0=s_1p_1f_1f_2f_3p_2s_2$, we have
$\operatorname{res}(P_0)=(2,0)\in\mathcal R_{m,d}^s$ in (\ref{tabel}).

\textbf{Case (\ref{O3}).} Suppose that
$(m,d)\in\{(6,2),(2,0)\}$ and either there is an
even edge $e$ with $|e\cap A|=|e\cap B|=2$, or $H_{\mathrm{even}}$ has
total $2$-pathlength at least three.

Suppose $e$ is an even edge of $H$ with $|e\cap A|=|e\cap B|=2$.  
If $e$ contains a $\beta_2$-bad split pair, then we apply Claim \ref{clm:csreplace} to obtain another even edge $s_1'\cup s_2'$ with $R=e$ such that both $s_1'$ and $s_2'$ are $\beta_2$-medium. 
By replacing $e$ with $s_1'\cup s_2'$ if necessary, we may assume that $e=s_1'\cup s_2'$ consists of two $\beta_2$-medium connate pairs.
Choose disjoint $\theta$-good connate pairs $p_1,p_2$ with
$p_1\cup s'_1,s'_2\cup p_2\in E(H)$, and then choose disjoint
$\theta$-good split pairs $s_1,s_2$ with $s_1\cup p_1,p_2\cup s_2\in E(H)$ using Proposition \ref{prop:oddmedium-to-good}.
Let $P_0=s_1p_1s'_1s'_2p_2s_2$ and
$\operatorname{res}(P_0)=(0,0)\in\mathcal R_{m,d}^s$ in (\ref{tabel}).

Now we assume every even edge is contained in $A$ or in $B$, and all pairs
in the following $2$-paths are connate. Since the total $2$-pathlength is
at least three, one of the following configurations exists.

If there are three disjoint even edges $e_1,e_2,e_3$, by using Claim \ref{clm:csreplace} as in previous cases to replace these three edges if necessary, we obtain three disjoint even edges
$e_i=q_i\cup q'_i$ for $i\in[3]$, where all six pairs are $\beta_2$-medium connate pairs. 
Choose $\theta$-good split pairs
$s_1,s_3,s_4,s_5,s_6,s_2$ joining $q_1,q_1',q_2,q_2',q_3,q_3'$ respectively by Proposition \ref{prop:oddmedium-to-good}.
Then choose connate pairs $r_1,r_2$ joining $s_3$ to $s_4$ and $s_5$ to $s_6$ by Proposition \ref{prop:oddpairconnect}. 
Let
$P_0=s_1q_1q'_1s_3r_1s_4q_2q'_2s_5r_2s_6q_3q'_3s_2$ and
$\operatorname{res}(P_0)=(0,0)\in\mathcal R_{m,d}^s$ in (\ref{tabel}).

If there is a $2$-path of two even edges and a disjoint even edge, write
$f_1=e_1\setminus e_2$, $f_2=e_1\cap e_2$,
$f_3=e_2\setminus e_1$.
As in Case (\ref{O1}), replacing if necessary, we may assume
that $f_1$ and $f_3$ are $\beta_2$-medium. Write the disjoint even edge as
$q_1\cup q_2$ with $q_1,q_2$ medium connate. Choose good split pairs
$s_1,s_3,s_4,s_2$ with
$s_1\cup f_1,f_3\cup s_3,s_4\cup q_1,q_2\cup s_2\in E(H)$ by Proposition \ref{prop:oddmedium-to-good}, and choose a
connate pair $r$ with $s_3\cup r,r\cup s_4\in E(H)$ by Proposition \ref{prop:oddpairconnect}. 
Let $P_0=s_1f_1f_2f_3s_3rs_4q_1q_2s_2$ and
$\operatorname{res}(P_0)=(0,0)\in\mathcal R_{m,d}^s$ in (\ref{tabel}).

Finally suppose there is a $2$-path of three even edges. Write
$f_1=e_1\setminus e_2$, $f_2=e_1\cap e_2$,
$f_3=e_2\cap e_3$, and $f_4=e_3\setminus e_2$. 
As in Case (\ref{O1}), replacing if necessary
allows us to assume that $f_1,f_4$ are
$\beta_2$-medium. 
Choose good split pairs $s_1,s_2$ with
$s_1\cup f_1,f_4\cup s_2\in E(H)$ by Proposition \ref{prop:oddmedium-to-good}.
Let $P_0=s_1f_1f_2f_3f_4s_2$ and thus
$\operatorname{res}(P_0)=(0,0)\in\mathcal R_{m,d}^s$ in (\ref{tabel}).
\end{proof}

The preceding lemma gives a constant-size block with the correct residue,
but it does not yet cover the exceptional vertices.  
We now extend this
block by inserting short local paths through the exceptional vertices while the residue is preserved.

\begin{proof}[Proof of Lemma \ref{lem:oddbridge}]
By Lemma \ref{lem:28oddbridge}, we obtain a path $P_0$ with at most $28$ vertices, whose two ends are $\theta$-good split pairs $s_1,s_2$.
Let $V^*:=(V\setminus V(P_0))\cup s_1\cup s_2$,
$A^*:=A\cap V^*$ and $B^*:=B\cap V^*$, then either
$|V^*|\equiv 6\pmod 8$ and $|A^*|-|B^*|\equiv 2\pmod 4$, or
$|V^*|\equiv 2\pmod 8$ and $|A^*|-|B^*|\equiv 0\pmod 4$.
Recall that $X_E$ is the set of exceptional vertices. 
Let $X_0:=X_E\setminus V(P_0)$. 
If $|X_0|$ is odd, then choose one vertex
$x_*\in V\setminus (V(P_0)\cup X_E)$ and set $X':=X_0\cup\{x_*\}$; otherwise
set $X':=X_0$. 
Let $q:=|X'|$ and write $X'=\{x_1,\ldots,x_q\}$. 
Then $q$ is
even and $q\le |X_E|+1\le \a n$ since $\eta\ll\a$ and inequality (\ref{ine:xtr}) holds. 

Since every vertex is $(\beta_1,\beta_2)$-medium and $\eta\ll\beta_1$, we greedily choose distinct $y_i$ avoiding $V(P_0)$ and used vertices in the same part as $x_i$ such that $r_i:=\{x_i,y_i\}$ is a $\beta_2$-medium connate pair for each $i\in[q]$.  
By Proposition~\ref{prop:oddmedium-to-good},
applying twice to the medium connate pair $r_i,i\in[q]$, we can choose two
pairwise disjoint $\theta$-good split pairs $a_i,b_i$ avoiding $V(P_0)$ and used vertices such that
$r_i\cup a_i,r_i\cup b_i\in E(H)$. 
Next choose a $\theta$-good split pair
$g_i$ disjoint from all previously chosen vertices for each $i\in[q]$.
Set $g_0:=s_2$.
By Proposition~\ref{prop:oddpairconnect}, applying to the two good split pairs
$g_{i-1},a_i,i\in[q]$, we choose a connate pair $c_i$ avoiding $V(P_0)$ and used vertices such that
$g_{i-1}\cup c_i,c_i\cup a_i\in E(H)$.
Applying the same proposition to
$b_i,g_i$, we choose a connate pair $d_i$ avoiding $V(P_0)$ and used vertices  such that
$b_i\cup d_i,d_i\cup g_i\in E(H)$. 
All these choices are possible because
the forbidden set has size at most $20\a n+O(1)\le\mu n$ in total.

%Next we choose two disjoint $\theta$-good split pairs $s,s'$. By Proposition \ref{prop:oddpairconnect}, there exist connate pairs $p,q$ such that $s\cup p,p\cup s_1\in E(H)$ and $g_q\cup q,q\cup s'\in E(H)$.

Let $P:=P_0s_2c_1a_1r_1b_1d_1g_1c_2a_2r_2b_2d_2g_2\cdots c_qa_qr_qb_qd_qg_q$ and  $X_E\subseteq V(P)$.
If $q=0$, then take $P:=P_0$.
Observe that $|V(P)\setminus V(P_0)|=12q \pmod 8$ and each $a_i,b_i,g_i$ are split pairs,
$c_i,r_i,d_i$ are connate pairs. 
Thus, $|A\cap(V(P)\setminus V(P_0))|-|B\cap(V(P)\setminus V(P_0))|=2q \pmod 4$.
Thus,
$\operatorname{res}(P)=\operatorname{res}(P_0)$ since $q$ is even.
Finally,
$|V(P)|\le |V(P_0)|+12q\le 28+24\a n\le \mu n$,
as $\a\ll\mu$ and $n$ is sufficiently large. 
\end{proof}

\subsection{Proof of Lemma \ref{lem:oddHP}}

It remains to prove the Hamilton path lemma.  
The proof converts the
problem into a matching and Hamilton path problem in an auxiliary graph.
%The main point is to choose enough disjoint good split pairs so that the remaining vertices on the two sides can be paired in the correct proportions.
We need the following technical result in the proof.

\begin{claim}\label{claim:odd-weighted-estimate}
Let $0<1/N\ll \xi\ll 1$. 
Let $a+b=N$ with
$a,b\ge N/3$.
Suppose that
$0\le m_A\le \binom{a}{2}$, $0\le m_B\le \binom{b}{2}$, and
$m_A+m_B\ge N^2/6$. 
Then
$y_Am_A\binom{a}{2}^{-1}+y_Bm_B\binom{b}{2}^{-1}\ge \left(\frac12+\xi\right)\frac{N-2}{4}$ where $y_A=(a-\frac{N+2}{4})/2$ and $y_B=(b-\frac{N+2}{4})/2$.
\end{claim}

\begin{proof}
By symmetry, assume $a\le b$. 
We claim that $y_A\binom a2^{-1}\le y_B\binom b2^{-1}$. 
Indeed, we have
$y_A\binom a2^{-1}-y_B\binom b2^{-1}
=\frac{(b-a)\left(ab-\frac{N+2} 4(N-1)\right)}{a(a-1)b(b-1)}\le0$ since $ab\le N^2/4$ and
$\frac{(N+2)(N-1)}4>N^2/4$ for $N>2$.
Set $m=m_A+m_B$.
Thus $y_Am_A\binom{a}{2}^{-1}+y_Bm_B\binom{b}{2}^{-1}=\left(y_A\binom a2^{-1}-y_B\binom b2^{-1}\right)m_A+y_B\binom b2^{-1}m$ is a decreasing function of $m_A$. 
The minimum is attained at
$m_A=\binom{a}{2}$ and $m=N^2/6$. Thus the minimum is
$s_{\min}:=y_A+y_B\binom b2^{-1}\left(\frac{N^2}{6}-\binom a2\right)$.
%Since $y_B>0$, the latter term is increasing in $m$, hence it is minimized at $m=N^2/6$.

Write $r:=a/N$. Then $r\in[1/3,1/2]$,
 and
$y_A=\frac{r-1/4}{2}N-O(1)$, $y_B=\frac{3/4-r}{2}N-O(1)$,
$\binom a2=\frac{r^2}{2}N^2-O(N)$ and
$\binom b2=\frac{(1-r)^2}{2}N^2-O(N)$. Therefore
$\frac1Ns_{\min}\ge f(r)-O(\frac1N)$,
where
$f(r):=\frac{r-1/4}{2}
+\frac{3/4-r}{(1-r)^2}\left(\frac16-\frac{r^2}{2}\right)
=\frac{r(12r^2-18r+7)}{12(1-r)^2}$.

We use the identity
$f(r)-\frac7{48}
=
\frac{(3r-1)(16r^2-21r+7)}{48(1-r)^2}\ge0$ for $r\in[1/3,1/2]$. 
Thus $\left(\frac7{48}-O(\frac1N)\right)N\ge\left(\frac18+\frac{\xi}{4}\right)N\ge\left(\frac12+\xi\right)y$, as desired.
\end{proof}

We now prove Lemma~\ref{lem:oddHP}.  
The claim above
ensures that the random or deterministic choice of auxiliary pairs has
enough flexibility.  
We use this flexibility to build the desired matching of pairs, and then obtain a Hamilton $2$-path with the prescribed split ends.

\begin{proof}[Proof of Lemma \ref{lem:oddHP}]
Let $N:=|U|\ge n-2\mu n$, $a:=|A_U|$, and $b:=|B_U|$. 
Let $S=A_U B_U$, $C_A=\binom{A_U}{2}$, $C_B=\binom{B_U}{2}$, and $C=C_A\cup C_B$, and we also view them as graphs defined on $U$.

Define a bipartite graph $\Gamma$ with parts $S$ and $C$ by joining $s\in S$ and $p\in C$ if and only if $s\cup p\in E(H)$. 
Thus every edge of $\Gamma$ corresponds to an odd edge of $H$.

%We first check the required numbers of pair types. 
Since the desired pair sequence starts and ends with split pairs and alternates between split and connate pairs, a Hamilton $2$-path on $N$ vertices must use $N/2$ pairs, of which $x:=\frac{N+2}{4}$ are split pairs and $y:=\frac{N-2}{4}$ are connate pairs. If $y_A$ connate pairs lie inside $A_U$ and $y_B$ inside $B_U$, then the split pairs use exactly $x$ vertices of $A_U$ and $x$ vertices of $B_U$. Hence $y_A=\frac{a-x}{2}$ and $y_B=\frac{b-x}{2}$. Clearly $y_A+y_B=y$.

The congruence assumptions imply that $x,y,y_A,y_B$ are integers. Indeed, if $N\equiv 2\pmod 8$ and $a-b\equiv 0\pmod 4$, then $x$ is odd, $y$ is even, and $a,b$ are both odd, so $a-x$ and $b-x$ are even. If $N\equiv 6\pmod 8$ and $a-b\equiv 2\pmod 4$, then $x$ is even, $y$ is odd, and $a,b$ are both even, so again $a-x$ and $b-x$ are even.
Moreover, since $a,b\ge n/3+\gamma n/4-2\mu n\ge N/3$, we have, for all sufficiently large $n$, $x=\frac{N+2}{4}\le \frac n4+1\le a-\frac{\gamma n}{4}$, and similarly $y\le b-\frac{\gamma n}4$. Therefore $y_A,y_B\ge \frac{\gamma n}{10}$.
We first verify these two conditions.
\begin{enumerate}[label=(F\arabic*), ref=F\arabic*, start=1]
    \item For every $\theta$-good connate pair $p$,
    $|N_S(p)|\ge \left(\frac12+\xi\right)|S|$.
    \label{F1}  
        \item For every $\theta$-good split pair $s$,
    $y_A\frac{|N_{C_A}(s)|}{|C_A|}+y_B\frac{|N_{C_B}(s)|}{|C_B|}\ge \left(\frac12+\xi\right)y$,
    \label{F2}
\end{enumerate}

To see (\ref{F1}), let $p$ be a $\theta$-good connate pair. 
We have $\deg_{H_{\rm even}}(p)\le \theta^3 \binom n2$. 
Hence $|N_S(p)|\ge(\frac13+\g)\binom{n-2}{2}-\theta^3\binom n2-2\mu n^2\ge(\frac13+\frac{\g}3)\binom{n}{2}\ge\left(\frac12+\xi\right)|S|$, by $|S|=ab\le \frac{N^2}{4}\le \frac{n^2}{4}$.
For (\ref{F2}), let $s$ be a $\theta$-good split pair. 
We have $\deg_{H_{\rm even}}(s)\le \theta^3 \binom n2$.
Hence $|N_C(s)|\ge(\frac13+\g)\binom{n-2}{2}-\theta^3\binom n2-\mu n^2\ge(\frac13+\frac{\g}3)\binom{n}{2}$.
Write $m_A:=|N_{C_A}(s)|$ and $m_B:=|N_{C_B}(s)|$.
Thus, $m_A+m_B\ge \left(\frac13+\frac{\gamma}{3}\right)\binom n2\ge\frac{N^2}{6}$. Recall that $a,b\ge N/3$. 
By Claim \ref{claim:odd-weighted-estimate}, the following holds
\[
y_A\frac{m_A}{|C_A|}+y_B\frac{m_B}{|C_B|}\ge \left(\frac12+\xi\right)y.    
\]

Let $s_1$ and $s_2$ be $\theta$-good split pairs.
Choose uniformly at random a decomposition $M_0$ of $U\setminus(s_1\cup s_2)$ into pairs, and conditioned on containing exactly $x-2$ split pairs, $y_A$ pairs from $C_A$, and $y_B$ pairs from $C_B$ uniformly among all such decompositions.

We claim that with positive probability $M_0$ satisfies the following two properties: (i) for every $\theta$-good split pair $s\in S$, $|N_\Gamma(s)\cap M_0\cap C|\ge \left(\frac12+\frac{2\xi}{3}\right)y$ , (ii) for every $\theta$-good connate pair $p\in C$, $|N_\Gamma(p)\cap M_0\cap S|\ge \left(\frac12+\frac{2\xi}{3}\right)x$, and (iii) $M_0$ contains at most $\sqrt{\rho}n$ non $\theta$-good pairs.

Indeed, fix a $\theta$-good split pair $s$. A fixed pair in $C_A$ appears in $ M_0$ with probability $\frac{y_A}{|C_A|}+O(n^{-2})$, and a fixed pair in $C_B$ appears in $M_0$ with probability $\frac{y_B}{|C_B|}+O(n^{-2})$. 
Therefore, by (\ref{F2}), 
\begin{align*}
\mathbb E |N_\Gamma(s)\cap M_0\cap C|&=\left(\frac{y_A}{|C_A|}+O(n^{-2})\right)m_A+\left(\frac{y_B}{|C_B|}+O(n^{-2})\right)m_B\\
&\ge \left(\frac12+\xi\right)y-O(1)\ge \left(\frac12+\frac{5\xi}{6}\right)y
\end{align*}
for all sufficiently large $n$. By Chernoff's bound,
$\mathbb P\left[|N_\Gamma(s)\cap M_0\cap C|<\left(\frac12+\frac{2\xi}{3}\right)y\right]\le \exp(-\Omega_\xi(n))$.

Similarly, fix a $\theta$-good connate pair $p$. A fixed split pair appears in $M_0$ with probability $\frac{x}{|S|}+O(n^{-2})$. Using (\ref{F1}), we get 
\[
\mathbb E |N_\Gamma(p)\cap M_0\cap S|=\left(\frac{x}{|S|}+O(n^{-2})\right)|N_s(p)|\ge \left(\frac12+\xi\right)x-O(1)\ge \left(\frac12+\frac{5\xi}{6}\right)x.
\]
Again Chernoff's bound gives
$\mathbb P\left[|N_\Gamma(p)\cap M_0\cap S|<\left(\frac12+\frac{2\xi}{3}\right)x\right]\le \exp(-\Omega_\xi(n))$.
Taking a union bound over all $O(n^2)$ pairs, $M_0$ satisfies the above two properties with probability $1-o(1)$.

%Second, we claim that $M_0$ contains at most $\sqrt{\rho}n$ not $\theta$-good pairs. 
Now we consider (iii). 
Note that every vertex of $U$ is incident with at most $\rho n$ non $\theta$-good pairs, so the total number of non $\theta$-good pairs in $U$ is at most $\frac12 \rho n|U|\le \frac12\rho n^2$. Every fixed pair appears in $M_0$ with probability $O(1/n)$. Hence the expected number of non $\theta$-good pairs in $M_0$ is $O(\rho n)$. 
By Markov's inequality, the probability that $M_0$ contains more than $\sqrt{\rho}n$ non $\theta$-good pairs is $O(\sqrt{\rho})$. 
Since $\rho\ll \xi$, all properties (i)--(iii) hold simultaneously for some choice of $M_0$.

Fix such a matching $M_0$. We replace the non $\theta$-good pairs in $M_0$ one by one, while preserving the numbers of split, $A_U$-connate, and $B_U$-connate pairs.

Suppose first that $ab\in M_0$ is a non $\theta$-good split pair, where $a\in A_U$ and $b\in B_U$. 
We choose a $\theta$-good split pair $xy\in M_0$, with $x\in A_U$ and $y\in B_U$, disjoint from $\{a,b\}$, such that both $ay$ and $xb$ are $\theta$-good. 
This is possible because $a$ and $b$ are each incident with at most $\rho n$ non $\theta$-good pairs, while $M_0$ contains linearly many split pairs and at most $\sqrt{\rho}n$ non $\theta$-good pairs. 
Replace $ab,\ xy$ by $ay,\ xb$. 
This reduces the number of non $\theta$-good pairs and preserves the number of split pairs.

Now we deal with non $\theta$-good connate pairs.
%Without loss of generality, we may 
First assume that $ab\in M_0$ is a non $\theta$-good connate pair inside $A_U$. 
By construction, 
$|M_0\cap C_A|=y_A\ge \gamma n/10$.
Among these pairs, at most
$\sqrt{\rho}n$ are non $\theta$-good, at most $O(1)$ meet
$\{a,b\}$, and at most $2\rho n$ contain a vertex $z$ for
which one of $az,bz$ is non $\theta$-good. Since $\rho\ll \gamma$, there is a
$\theta$-good connate pair $xy\in M_0\cap C_A$ disjoint from
$\{a,b\}$ such that all pairs $ax,ay,bx,by$ are $\theta$-good.
Replace $ab,\ xy$ by $ax,\ by$.
This reduces the number of non $\theta$-good pairs and preserves the number of $A_U$-connate pairs. 
Similar operations for non $\theta$-good connate pairs in $B_U$ can be executed analogously.

After at most $\sqrt{\rho}n$ such operations, we obtain a perfect matching $M_0'$ of $U\setminus(s_1\cup s_2)$, all pairs in $ M_0'$ are $\theta$-good.
Let $M:=M_0'\cup s_1\cup s_2$ and thus $M$ contains exactly $x$ split pairs, $y_A$ pairs from $C_A$, and $y_B$ pairs from $C_B$. 
During the repair process, at most $4\sqrt{\rho}n$ matching pairs are changed. Therefore, for every $\theta$-good split pair $s$,
$|N_\Gamma(s)\cap M\cap C|\ge \left(\frac12+\frac{2\xi}{3}\right)y-4\sqrt{\rho}n\ge \left(\frac12+\frac{\xi}{2}\right)y,$ 
and for every $\theta$-good connate pair $p$,
$|N_\Gamma(p)\cap M\cap S|\ge \left(\frac12+\frac{2\xi}{3}\right)x-4\sqrt{\rho}n\ge \left(\frac12+\frac{\xi}{2}\right)x,$
because $\rho\ll \xi^2$ and $y,x=\Omega(n)$.

Let $X:=M\cap S$ and $Y:=M\cap C$. Then $|X|=x$, $|Y|=y$, and $x=y+1$.
The bipartite graph $\Gamma[M]$ satisfies $d_{\Gamma[M]}(s)\ge \left(\frac12+\frac{\xi}{2}\right)y$ for every $s\in X$, and $d_{\Gamma[M]}(p)\ge \left(\frac12+\frac{\xi}{2}\right)x$ for every $p\in Y$.

Choose $y_0\in N_{\Gamma[M]}(s_1)\cap Y$. Let
$G':=\Gamma[M]-s_1$, with parts $X':=X\setminus\{s_1\}$ and $Y$.
Then $|X'|=|Y|=y$. For every nonedge $xy$ with $x\in X'$ and $y\in Y$,
the degree-sum assumption in Lemma~\ref{lem:bHamiltoncon} holds. Hence
$G'$ has a Hamilton path from $y_0$ to $s_2$. Attaching the edge
$s_1y_0$ gives a Hamilton path from $s_1$ to $s_2$ in $\Gamma[M]$.
Since $M$ is a perfect matching of $U$, the pairs $r_1,\ldots,r_{N/2}$ cover every vertex of $U$ exactly once. Moreover, for every $i$, $r_i\cup r_{i+1}\in E(\Gamma)$, and hence $r_i\cup r_{i+1}\in E(H)$. Finally, because $\Gamma$ is bipartite between split pairs and connate pairs, the pair sequence alternates between split pairs and connate pairs.
Thus $r_1r_2\cdots r_{N/2}$ is a Hamilton $2$-path in $H[U]$ from $s_1$ to $s_2$, as required.
\end{proof}

\bibliographystyle{abbrv}
\bibliography{refs, Bibref}
\end{document}